\documentclass[twoside]{article}
\usepackage[a4paper]{geometry}
\usepackage[latin1]{inputenc} 
\usepackage[T1]{fontenc} 
\usepackage{RR}
\usepackage{hyperref}

\usepackage{amsmath, amsthm, mathrsfs, amssymb, units, wrapfig}
\usepackage{mathtools}
\usepackage{authblk} 
\usepackage{textcomp} 
\usepackage{graphicx}
\usepackage{multirow, longtable}
\usepackage{tabvar}
\usepackage{enumitem}
\theoremstyle{plain} 
\newtheorem{theorem}{Theorem}[section]
\newtheorem{proposition}[theorem]{Proposition} 
 
\newtheorem{lemma}[theorem]{Lemma}
\theoremstyle{definition}
\newtheorem{definition}[theorem]{Definition}

\theoremstyle{remark}
\newtheorem{remark}[theorem]{Remark} 

\usepackage{bbm} 
\usepackage{xcolor}
\usepackage{comment}
\usepackage[skip=1ex]{caption}
\usepackage{subcaption, cleveref} 

\definecolor{niceblue}{rgb}{0.25, 0.4, 0.96}
\hypersetup{colorlinks,urlcolor=niceblue,linkcolor=niceblue,citecolor=niceblue,breaklinks}

\newcommand{\C}{\mathbb{C}}

\newcommand{\N}{\mathbb{N}}
 
\newcommand{\R}{\mathbb{R}} 

\newcommand{\Hso}{\mathbb{H}}

\DeclareMathOperator{\dive}{\mathbf{div}}
\DeclareMathOperator{\pa}{\partial}

\newcommand{\norm}[1]{\left\lVert#1\right\rVert}
\newcommand{\scalar}[1]{\langle#1\rangle}
\newcommand{\ic}{\mathrm{i}} 

\newcommand{\sigmar}{\sigma_{\mathrm{r}}}
\newcommand{\sigmatz}{\tilde{\sigma}_0}

\RRNo{9477}
\RRdate{July 2022}
\RRauthor{
Marcella Bonazzoli\thanks[sfn]{Inria, UMA, ENSTA Paris, Institut Polytechnique de Paris}%
  \and
Houssem Haddar\footnotemark[1]
\and Tuan Anh Vu\footnotemark[1]}
\authorhead{M. Bonazzoli \& H. Haddar \& T. A. Vu}
\RRtitle{Analyse de convergence pour des
	m\'ethodes d'inversion multi-\'etapes de type one-shot}
\RRetitle{Convergence analysis of multi-step one-shot methods for
	 linear inverse problems}
\RRresume{Dans ce travail nous nous int\'eressons \`a des probl\`emes inverses lin\'eaires g\'en\'eraux o\`u le probl\`eme direct correspondant est r\'esolu de fa\c con it\'erative en utilisant des m\'ethodes de point fixe. Ainsi, les m\'ethodes de type one-shot, qui it\`erent en m\^eme temps sur la solution du probl\`eme direct et l'inconnue du probl\`eme inverse, peuvent \^etre appliqu\'ees. Nous consid\'erons deux variantes des m\'ethodes multi-\'etapes de type one-shot et nous \'etablissons des conditions suffisantes et n\'ecessaires sur le pas de descente pour leur convergence, en \'etudiant les valeurs propres de la matrice par blocs des it\'erations coupl\'ees. Plusieurs tests num\'eriques sont pr\'esent\'es pour illustrer la convergence de ces m\'ethodes par rapport aux m\'ethodes de descente de gradient usuelle et d\'ecentr\'ee. En particulier, nous observons que tr\`es peu d'it\'erations internes pour le probl\`eme direct sont suffisantes pour garantir une bonne convergence de l'algorithme d'inversion.   
}
\RRabstract{
In this work we are interested in general linear inverse problems where the corresponding forward problem is solved iteratively using fixed point methods. Then one-shot methods, which iterate at the same time on the forward problem solution and on the inverse problem unknown, can be applied. We analyze two variants of the so-called multi-step one-shot methods and establish sufficient conditions on the descent step for their convergence, by studying the eigenvalues of the block matrix of the coupled iterations. Several numerical experiments are provided to illustrate the convergence of these methods in comparison with the classical usual and shifted gradient descent. In particular, we observe that very few inner iterations on the forward problem are enough to guarantee good convergence of the inversion algorithm.}
\RRmotcle{probl\`emes inverses, m\'ethodes de type one-shot, analyse de convergence, identification de param\`etres}
\RRkeyword{inverse problems, one-shot methods, convergence analysis, parameter identification}
\RRprojets{IDEFIX}
\RCSaclay 

\begin{document}
\makeRR   

\newpage
\tableofcontents

\newpage
	\section{Introduction}

For large-scale inverse problems, which often arise in real life applications, the solution of the corresponding forward and adjoint problems is generally computed using an iterative solver, such as (preconditioned) fixed point or Krylov subspace methods. Indeed, the corresponding linear systems could be too large to be handled with direct solvers (e.g.~LU-type solvers), and iterative solvers are easier to parallelize on many cores. Naturally this leads to the idea of \emph{one-step one-shot methods}, which iterate at the same time on the forward problem solution (the state variable), the adjoint problem solution (the adjoint state) and on the inverse problem unknown (the parameter or design variable).  
If two or more inner iterations are performed on the state and adjoint state before updating the parameter (by starting from the previous iterates as initial guess for the state and adjoint state), we speak of \emph{multi-step one-shot methods}. 
Our goal is to rigorously analyze the convergence of such inversion methods. In particular, we are interested in those schemes where the inner iterations on the direct and adjoint problems are incomplete, i.e.~stopped before achieving convergence. Indeed, solving the forward and adjoint problems exactly by direct solvers or very accurately by iterative solvers could be very time-consuming with little improvement in the accuracy of the inverse problem solution.   

The concept of one-shot methods was first introduced by Ta'asan \cite{taasan91} for optimal control problems. Based on this idea, a variety of related methods, such as the all-at-once methods, where the state equation is included in the misfit functional, were developed for aerodynamic shape optimization, see for instance \cite{taasan92, shenoy97, hazra05, schulz09, gauger09} and the literature review in the introduction of \cite{schulz09}. All-at-once approaches to inverse problems for parameter identification were studied in, e.g., \cite{haber01, burger02, kaltenbacher14}. An alternative method, called Wavefield Reconstruction Inversion (WRI), was introduced for seismic imaging in \cite{leeuwen13}, as an improvement of the classical Full Waveform Inversion (FWI) \cite{tarantola82}. WRI is a penalty method which combines the advantages of the all-at-once approach with those of the reduced approach (where the state equation represents a constraint and is enforced at each iteration, as in FWI), and was extended to more general inverse problems in \cite{leeuwen15}. 

Few convergence proofs, especially for the multi-step one-shot methods, are available in literature. In particular, for non-linear design optimization problems, Griewank \cite{griewank06} proposed a version of one-step one-shot methods where a Hessian-based preconditioner is used in the design variable iteration. 
The author proved conditions to ensure that the real eigenvalues of the Jacobian of the coupled iterations are smaller than $1$, but these are just necessary and not sufficient conditions to exclude real eigenvalues smaller than $-1$. In addition, no condition to also bound complex eigenvalues below $1$ in modulus was found, and multi-step methods were not investigated. 
In \cite{hamdi09, hamdi10, gauger12} an exact penalty function of doubly augmented Lagrangian type was introduced to coordinate the coupled iterations, and global convergence of the proposed optimization approach was proved under some assumptions. In \cite{guenther16} this particular one-step one-shot approach was extended to time-dependent problems. 

In this work, we consider two variants of multi-step one-shot methods where the forward and adjoint problems are solved using fixed point methods and the inverse problem is solved using gradient descent methods. This is a preparatory work where we focus on (discretized) linear inverse problems. Note that the present analysis in the linear case implies also local convergence in the non-linear case. The only basic assumptions we require are the inverse problem uniqueness and the convergence of the fixed point iteration for the forward problem. To analyze the convergence of the coupled iterations we study the real and complex eigenvalues of the block iteration matrices. We prove that if the descent step is small enough then the considered multi-step one-shot methods converge. Moreover, the upper bounds for the descent step in these sufficient conditions are explicit in the number of inner iterations and in the norms of the operators involved in the problem. In the particular scalar case (Appendix~\ref{app:1D}), we establish sufficient and also necessary convergence conditions on the descent step. 

This paper is structured as follows. In Section~\ref{sec:intro-k-shot}, we introduce the principle of multi-step one-shot methods and define two variants of these algorithms. Then, in Section~\ref{sec:one-step}, respectively Section~\ref{sec:multi-step}, we analyze the convergence of one-step one-shot methods, respectively multi-step one-shot methods: first, we establish eigenvalue equations for the block matrices of the coupled iterations, then we derive sufficient convergence conditions on the descent step by studying both real and complex eigenvalues. In Section~\ref{sec:complex_extension} we show that the previous analysis can be extended to the case where the state variable is complex. Finally, in Section~\ref{sec:num-exp} we test numerically the performance of the different algorithms on a toy 2D Helmholtz inverse problem.  

Throughout this work, $\scalar{\cdot,\cdot}$ indicates the usual Hermitian scalar product in $\C^n$, that is $\scalar{x,y} \coloneqq \overline{y}^\intercal x, \forall x,y\in\C^n$, and $\norm{\cdot}$ the vector/matrix norms induced by $\scalar{\cdot,\cdot}$. We denote by $A^*=\overline{A}^\intercal$ the adjoint operator of a matrix $A\in\C^{m\times n}$, and likewise by $z^*=\overline{z}$ the conjugate of a complex number $z$. 
The identity matrix is always denoted by $I$, whose size is understood from context. 
Finally, for a matrix $T\in\C^{n\times n}$ with $\rho(T)<1$, we define
\[
s(T) \coloneqq \sup_{z\in\C, |z|\ge 1}\norm{\left(I-T/z\right)^{-1}}
\]
which is further studied in Appendix \ref{app:lems}.

	\section{Multi-step one-shot inversion methods}
\label{sec:intro-k-shot}

We focus on (discretized) linear inverse problems, which correspond to a \emph{direct (or forward) problem} of the form: find $u \equiv u(\sigma)$ such that 
\begin{equation}
	\label{pbdirect}
	u=Bu+M\sigma+F
\end{equation}
where $u\in\R^{n_u}$, $\sigma\in\R^{n_\sigma}$, $B\in\R^{n_u \times n_u}$, $M\in\R^{n_u \times n_\sigma}$ and $F\in\R^{n_u}$.  
Here $I-B$ is the invertible matrix of the direct problem, obtained after discretization, with parameter $\sigma$. Note that in the non-linear case $B$ would be a function of $\sigma$. Equation~\eqref{pbdirect} is also called \emph{state equation} and $u$ is called \emph{state}. 
Given $\sigma$, we can solve for $u$ by a fixed point iteration
\begin{equation}
	\label{iterpbdirect}
	u_{\ell+1}=Bu_\ell+M\sigma+F, \quad \ell=0,1,\dots, 
\end{equation}
which converges for any initial guess $u_0$ if and only if the spectral radius $\rho(B)$ is strictly less than $1$ (see e.g. \cite[Theorem 2.1.1]{greenbaum97}). 
Hence we assume $\rho(B)<1$. 
Now, we measure $f=Hu(\sigma)$, where $H\in\R^{n_f\times n_u}$, and we are interested in the \emph{linear inverse problem} of finding $\sigma$ from $f$. 
In order to guarantee the uniqueness of the inverse problem, we assume that $H(I-B)^{-1}M$ is injective. In summary, we set
\begin{equation}
	\label{direct-inv-prb}
	\begin{array}{lc}
		\mbox{direct problem:} & u=Bu+M\sigma+F,\\
		\mbox{inverse problem:} & \mbox{measure }f=Hu(\sigma),\mbox{ find }\sigma
	\end{array}
\end{equation}
with the assumptions:
\begin{equation}
	\label{hypo}
	\rho(B)<1, \quad H(I-B)^{-1}M \mbox{ is injective}.
\end{equation}	
To solve the inverse problem we write its least squares formulation: given $\sigma^\text{ex}$ the exact solution of the inverse problem and $f \coloneqq Hu(\sigma^\text{ex})$, 
$$\sigma^\text{ex} = \mbox{argmin}_{\sigma\in\R^{n_\sigma}} J(\sigma) \quad \mbox{ where } J(\sigma) \coloneqq \frac{1}{2}\norm{Hu(\sigma)-f}^2.$$ 
Using the classical Lagrangian technique with real scalar products, we introduce the \emph{adjoint state} $p \equiv p(\sigma)$, which is the solution of 
$$p=B^*p+H^*(Hu-f)$$
and allows us to compute the gradient of the cost functional
$$\nabla J(\sigma)=M^*p(\sigma).$$  
The classical gradient descent algorithm then reads 
\begin{equation}
	\label{usualgd}
	\mbox{\textbf{usual gradient descent:}}\quad
	\begin{cases}
		\sigma^{n+1}=\sigma^n-\tau M^*p^n, \\
		u^n=Bu^n+M\sigma^n+F,\\
		p^n=B^*p^n+H^*(Hu^n-f),
	\end{cases}
\end{equation}
where $\tau>0$ is the descent step size, and the state and adjoint state equations are solved exactly by a direct solver.  
Here $\sigma^{n+1}=\sigma^n-\tau \nabla J(\sigma_{n})$; if instead we update $\sigma^{n+1}=\sigma^n-\tau \nabla J(\sigma_{n-1})$, we obtain the 
\begin{equation}
	\label{shifted-gd}
	\mbox{\textbf{shifted gradient descent:}}\quad
	\begin{cases}
		\sigma^{n+1}=\sigma^n-\tau M^*p^n, \\
		u^{n+1}=Bu^{n+1}+M\sigma^n+F,\\
		p^{n+1}=B^*p^{n+1}+H^*(Hu^{n+1}-f). 
	\end{cases}
\end{equation}  
Both algorithms converge for sufficiently small $\tau$ (see e.g.~Appendix~\ref{app:convGD}): 
for any initial guess, \eqref{usualgd} converges  if 
\begin{equation}
	\label{best-tau-usualgd}
	\tau<\frac{2}{\norm{H(I-B)^{-1}M}^2},
\end{equation}
and \eqref{shifted-gd} converges if 
\begin{equation}
	\label{best-tau-shifted-gd}
	\tau<\frac{1}{\norm{H(I-B)^{-1}M}^2}.
\end{equation}
Here, we are interested in methods where the direct and adjoint problems are rather solved iteratively as in \eqref{iterpbdirect}, and where we iterate at the same time on the forward problem solution and the inverse problem unknown: such methods are called \emph{one-shot methods}. More precisely, we are interested in two variants of \emph{multi-step one-shot methods}, defined as follows. Let $n$ be the index of the (outer) iteration on $\sigma$, the solution to the inverse problem. 
We update $\sigma^{n+1}=\sigma^n-\tau M^*p^n$ as in gradient descent methods, but the state and adjoint state equations are now solved by a fixed point iteration method, using just \emph{$k$ inner iterations}, and \emph{coupled}:
$$\begin{cases}
	u^{n+1}_{\ell+1}=Bu^{n+1}_\ell+M\sigma+F,\\
	p^{n+1}_{\ell+1}=B^*p^{n+1}_\ell+H^*(Hu^{n+1}_\ell-f),
\end{cases}
\quad\ell=0,1,\dots,k,
\quad\begin{cases}
	u^{n+1}=u^{n+1}_k,\\ p^{n+1}=p^{n+1}_k
\end{cases}$$
where $\sigma$ depends on the considered variant ($\sigma=\sigma^{n+1}$ or, for the shifted methods, $\sigma=\sigma^n$). As initial guess we naturally choose $u^{n+1}_0=u^n$ and $p^{n+1}_0=p^n$, the information from the previous (outer) step. 
In summary, we have two multi-step one-shot algorithms
\begin{equation}
	\label{alg:k-shot n+1}
	k\mbox{\textbf{-step one-shot:}}\quad\begin{cases}
		\sigma^{n+1}=\sigma^n-\tau M^*p^n, \\
		u^{n+1}_0 = u^n, p^{n+1}_{0} = p^n,\\
		\quad\left| \begin{array}{l}
			u^{n+1}_{\ell+1}=Bu^{n+1}_\ell+M {\sigma^{n+1}}+F,\\
			p^{n+1}_{\ell+1}=B^*p^{n+1}_\ell+H^*(Hu^{n+1}_\ell -f), 
		\end{array}\right. \\
		u^{n+1} = u^{n+1}_k, p^{n+1} = p^{n+1}_k 
	\end{cases}
\end{equation}
and
\begin{equation}
	\label{alg:k-shot n}
	\mbox{\textbf{shifted} }k\mbox{\textbf{-step one-shot:}}\quad\begin{cases}
		\sigma^{n+1}=\sigma^n-\tau M^*p^n, \\
		u^{n+1}_0 = u^n, p^{n+1}_{0} = p^n,\\
		\quad\left| \begin{array}{l}
			u^{n+1}_{\ell+1}=Bu^{n+1}_\ell+M {\sigma^{n}}+F,\\
			p^{n+1}_{\ell+1}=B^*p^{n+1}_\ell+H^*(Hu^{n+1}_\ell -f), 
		\end{array}\right. \\
		u^{n+1} = u^{n+1}_k, p^{n+1} = p^{n+1}_k,
	\end{cases}
\end{equation}	
and in particular, when $k=1$, we obtain the following two algorithms
\begin{equation}
	\label{alg:1-shot n+1}
	\mbox{\textbf{one-step one-shot:}}\quad\begin{cases}
		\sigma^{n+1}=\sigma^n-\tau M^*p^n, \\
		u^{n+1} = Bu^n+M\sigma^{n+1}+F\\
		p^{n+1} = B^*p^n+H^*(Hu^n-f)
	\end{cases}
\end{equation}
and
\begin{equation}
	\label{alg:1-shot n}
	\mbox{\textbf{shifted} }\mbox{\textbf{one-step one-shot:}}\quad\begin{cases}
		\sigma^{n+1}=\sigma^n-\tau M^*p^n, \\
		u^{n+1} = Bu^n+M\sigma^n+F\\
		p^{n+1} = B^*p^n+H^*(Hu^n-f).
	\end{cases}
\end{equation}	
The only difference for the shifted versions lies in the fact that $\sigma^{n}$ is used in \eqref{alg:k-shot n} and \eqref{alg:1-shot n}, instead of $\sigma^{n+1}$ in \eqref{alg:k-shot n+1} and \eqref{alg:1-shot n+1}, so that in \eqref{alg:k-shot n+1} and \eqref{alg:1-shot n+1} we need to wait for $\sigma$ before updating $u$ and $p$, while in \eqref{alg:k-shot n} and \eqref{alg:1-shot n} we can update $\sigma, u, p$ at the same time. Also note that when $k\rightarrow\infty$, the $k$-step one-shot method \eqref{alg:k-shot n+1} formally converges to the usual gradient descent \eqref{usualgd}, while the shifted $k$-step one-shot method \eqref{alg:k-shot n} formally converges to the shifted gradient descent \eqref{shifted-gd}.

We first analyze the one-step one-shot methods ($k=1$) in Section~\ref{sec:one-step} and then the multi-step one-shot methods ($k \ge 2$) in Section~\ref{sec:multi-step}.

	\section{Convergence of one-step one-shot methods ($k=1$)}
\label{sec:one-step}

\subsection{Block iteration matrices and eigenvalue equations}
\label{sec:eigen-eq-k=1}

To analyze the convergence of these methods, first we express $(\sigma^{n+1},u^{n+1},p^{n+1})$ in terms of $(\sigma^n,u^n,p^n)$, by inserting the expression for $\sigma^{n+1}$ into the iteration for $u^{n+1}$ in \eqref{alg:1-shot n+1}, so that system \eqref{alg:1-shot n+1} is rewritten as
\begin{equation}
	\label{1-shot expl n+1}
	\begin{cases}
		\sigma^{n+1}=\sigma^n-\tau M^*p^n\\
		u^{n+1}=Bu^n+M\sigma^n-\tau MM^*p^n+F\\
		p^{n+1}=B^*p^n+H^*Hu^n-H^*f.\\
	\end{cases}
\end{equation}
System \eqref{alg:1-shot n} is already in the form we need. 
In what follows we first study the shifted $1$-step one-shot method, then the $1$-step one-shot method. 

Now, we consider the errors $(\sigma^n-\sigma^\text{ex},u^n-u(\sigma^\text{ex}),p^n-p(\sigma^\text{ex}))$ with respect to the exact solution  at the $n$-th iteration, and, by abuse of notation, we designate them by $(\sigma^n,u^n,p^n)$. We obtain that the errors satisfy: for the shifted algorithm \eqref{alg:1-shot n}  
\begin{equation}
	\label{1-shot-err expl n}
	\begin{cases}
		\sigma^{n+1}=\sigma^n-\tau M^*p^n\\
		u^{n+1}=Bu^n+M\sigma^n\\
		p^{n+1}=B^*p^n+H^*Hu^n\\
	\end{cases}
\end{equation}
and for algorithm \eqref{1-shot expl n+1}
\begin{equation}
	\label{1-shot-err expl n+1}
	\begin{cases}
		\sigma^{n+1}=\sigma^n-\tau M^*p^n\\
		u^{n+1}=Bu^n+M\sigma^n-\tau MM^*p^n\\
		p^{n+1}=B^*p^n+H^*Hu^n,\\
	\end{cases}
\end{equation}
or equivalently, by putting in evidence the block iteration matrices 
\begin{equation}
	\label{1-shot-itermat n}
	\begin{bmatrix}
		p^{n+1}\\ u^{n+1}\\ \sigma^{n+1}
	\end{bmatrix}=\begin{bmatrix}
		B^* & H^*H & 0 \\
		0 & B & M \\
		-\tau M^* & 0 & I
	\end{bmatrix}
	\begin{bmatrix}
		p^{n}\\ u^{n}\\ \sigma^{n}
	\end{bmatrix}
\end{equation}
and 
\begin{equation}
	\label{1-shot-itermat n+1}
	\begin{bmatrix}
		p^{n+1}\\ u^{n+1}\\ \sigma^{n+1}
	\end{bmatrix}=\begin{bmatrix}
		B^* & H^*H & 0\\
		-\tau MM^* & B & M \\
		-\tau M^* & 0 & I
	\end{bmatrix}
	\begin{bmatrix}
		p^{n}\\ u^{n}\\ \sigma^{n}
	\end{bmatrix}.
\end{equation}
Now recall that a fixed point iteration converges if and only if the spectral radius of its iteration matrix is strictly less than $1$. Therefore in the following propositions we establish eigenvalue equations for the iteration matrix of the two methods. 

\begin{proposition}[Eigenvalue equation for the shifted $1$-step one-shot method]\label{prop:eq-eigen shift-1-shot} 
	Assume that $\lambda\in\C$ is an eigenvalue of the iteration matrix in \eqref{1-shot-itermat n}.
	\begin{enumerate}[label=(\roman*)]
		\item If $\lambda\in\C$, $\lambda\notin\mathrm{Spec}(B)$, then $\exists \, y\in\C^{n_\sigma}, y\neq 0$ such that
		\begin{equation}\label{ori-eq-eigen-1-shot n}
			(\lambda-1)\norm{y}^2+\tau\scalar{M^*(\lambda I-B^*)^{-1}H^*H(\lambda I-B)^{-1}My,y}=0.
		\end{equation}
		\item $\lambda=1$ is not an eigenvalue of the iteration matrix.
	\end{enumerate}
\end{proposition}
\begin{remark}
	Since $\rho(B)$ is strictly less than $1$, so is $\rho(B^*)$.
\end{remark}		
\begin{proof}
	Since $\lambda\in\C$ is an eigenvalue of the iteration matrix in \eqref{1-shot-itermat n}, there exists a non-zero vector $(\tilde{p},\tilde{u},y)\in\C^{n_u+n_u+n_\sigma}$ such that 
	\begin{equation}
		\label{eq:eigenvec-1-shot n}
		\begin{cases}
			\lambda y = y-\tau M^*\tilde{p} \\
			\lambda\tilde{u}= B\tilde{u}+My \\
			\lambda\tilde{p}=B^*\tilde{p}+H^*H\tilde{u}.
		\end{cases}
	\end{equation}
	By the second equation in \eqref{eq:eigenvec-1-shot n} $\tilde{u}=(\lambda I-B)^{-1}My$, so together with the third equation 
	$$
	\tilde{p}=(\lambda I-B^*)^{-1}H^*H\tilde{u}=(\lambda I-B^*)^{-1}H^*H(\lambda I-B)^{-1}My, 
	$$
	and by inserting this result into the first equation we obtain
	\begin{equation}\label{eq:pre-eq-eigen-1-shot n}
		(\lambda-1)y=-\tau M^*(\lambda I-B^*)^{-1}H^*H(\lambda I-B)^{-1}My,
	\end{equation}
	that gives \eqref{ori-eq-eigen-1-shot n} by taking the scalar product with $y$. We also see that if $y=0$ then the above formulas for $\tilde{u},\tilde{p}$ immediately give $\tilde{u}=\tilde{p}=0$, that is a contradiction.
	
	\noindent (ii) Assume that $\lambda=1$ is an eigenvalue of the iteration matrix, then \eqref{eq:pre-eq-eigen-1-shot n} gives us
	$$M^*(I-B^*)^{-1}H^*H(I-B)^{-1}My=0,$$
	but this cannot happen for $y \ne 0$ due to the injectivity of  $H(I-B)^{-1}M$.
\end{proof}

\begin{proposition}[Eigenvalue equation for the $1$-step one-shot method]
	\label{prop:eq-eigen 1-shot} 
	Assume that $\lambda\in\C$ is an eigenvalue of the iteration matrix in \eqref{1-shot-itermat n+1}.
	\begin{enumerate}[label=(\roman*)]
		\item If $\lambda\in\C$, $\lambda\notin\mathrm{Spec}(B)$ then $\exists \, y\in\C^{n_\sigma}, y\neq 0$ such that:
		\begin{equation}\label{ori-eq-eigen-1-shot n+1}
			(\lambda-1)\norm{y}^2+\tau\lambda\scalar{M^*(\lambda I-B^*)^{-1}H^*H(\lambda I-B)^{-1}My,y}=0.
		\end{equation}
		\item $\lambda=1$ is not an eigenvalue of the iteration matrix.
	\end{enumerate}
\end{proposition}

\begin{proof}
	Since $\lambda\in\C$ is an eigenvalue of the iteration matrix in \eqref{1-shot-itermat n+1}, there exists a non-zero vector $(\tilde{p},\tilde{u},y)\in\C^{n_u+n_u+n_\sigma}$ such that 
	\begin{equation}
		\label{eq:eigenvec-1-shot n+1}
		\begin{cases}
			\lambda y = y-\tau M^*\tilde{p} \\
			\lambda\tilde{u}= B\tilde{u}+My-\tau MM^*\tilde{p} \\
			\lambda\tilde{p}=B^*\tilde{p}+H^*H\tilde{u}.
		\end{cases}
	\end{equation}
	By the third equation in \eqref{eq:eigenvec-1-shot n+1} $\tilde{p}=(\lambda I-B^*)^{-1}H^*H\tilde{u}$, and  inserting this result into the second equation we obtain
	$$
	\lambda\tilde{u}=B\tilde{u}+My-\tau MM^*(\lambda I-B^*)^{-1}H^*H\tilde{u},
	$$
	or equivalently, 
	$$[I+\tau MM^*A](\lambda I-B)\tilde{u}=My$$
	where $A=(\lambda I-B^*)^{-1}H^*H(\lambda I-B)^{-1}$. Since $\tau>0$, $I+\tau MM^*A$ is a positive definite matrix. Therefore
	$$\tilde{u}=(\lambda I-B)^{-1}[I+\tau MM^*A]^{-1}My$$
	and
	$$\tilde{p}=(\lambda I-B^*)^{-1}H^*H\tilde{u}=A[I+\tau MM^*A]^{-1}My.$$
	By inserting this result into the first equation in \eqref{eq:eigenvec-1-shot n+1} we obtain
	$$(\lambda-1)y=-\tau M^*A[I+\tau MM^*A]^{-1}My.$$
	Thanks to the fact that $[I+\tau MM^*A]^{-1}$ and $MM^*A$ commute, we have
	$$(\lambda-1)My=-\tau MM^*A[I+\tau MM^*A]^{-1}My=-\tau[I+\tau MM^*A]^{-1}MM^*AMy$$
	then
	$$(\lambda-1)[I+\tau MM^*A]My=-\tau MM^*AMy,$$
	that leads to
	$$(\lambda-1)My+\tau\lambda MM^*AMy=0.$$ 
	Since $H(I-B)^{-1}M$ is injective, so is $M$. Therefore
	\begin{equation}\label{eq:pre-eq-eigen-1-shot n+1}
		(\lambda-1)y+\tau\lambda M^*AMy=0,
	\end{equation}
	that gives \eqref{ori-eq-eigen-1-shot n+1} by taking scalar product with $y$. We also see that if $y=0$ then the above formulas for $\tilde{u},\tilde{p}$ immediately give $\tilde{u}=\tilde{p}=0$, that is a contradiction.
	
	\noindent (ii) Assume that $\lambda=1$ is an eigenvalue of the iteration matrix, then \eqref{eq:pre-eq-eigen-1-shot n+1} gives us
	$$M^*(I-B^*)^{-1}H^*H(I-B)^{-1}My=0,$$
	but this cannot happen for $y \ne 0$ due to the injectivity of  $H(I-B)^{-1}M$.
\end{proof}

In the following sections we will show that, for sufficiently small $\tau$, equations \eqref{ori-eq-eigen-1-shot n} and \eqref{ori-eq-eigen-1-shot n+1} admit no solution $|\lambda|\ge 1$, thus algorithms \eqref{alg:1-shot n} and \eqref{alg:1-shot n+1} converge. When $\lambda\neq 0$, it is convenient to rewrite \eqref{ori-eq-eigen-1-shot n} and \eqref{ori-eq-eigen-1-shot n+1} respectively as
\begin{equation}
	\label{eq-eigen-1-shot n}
	\lambda^2(\lambda-1)\norm{y}^2+\tau\scalar{M^*\left(I-B^*/\lambda\right)^{-1}H^*H\left(I-B/\lambda\right)^{-1}My,y}=0
\end{equation}
and
\begin{equation}
	\label{eq-eigen-1-shot n+1}
	\lambda(\lambda-1)\norm{y}^2+\tau\scalar{M^*\left(I-B^*/\lambda\right)^{-1}H^*H\left(I-B/\lambda\right)^{-1}My,y}=0.
\end{equation}
For the analysis we use auxiliary results proved in Appendix~\ref{app:lems}.

\medskip
First, we study separately the very particular case where $B=0$. 

\begin{proposition}[shifted $1$-step one-shot method]\label{tau-1-shot-B=0 n}
	When $B=0$, the eigenvalue equation \eqref{eq-eigen-1-shot n} admits no solution $\lambda\in\C, |\lambda|\ge 1$ if
	$\tau<\frac{-1+\sqrt{5}}{2\norm{H}^2\norm{M}^2}$.
\end{proposition}
\begin{proof}
	When $B=0$, equation \eqref{eq-eigen-1-shot n} becomes
	$\lambda^2(\lambda-1)\norm{y}^2+\tau\norm{HMy}^2=0$ which is equivalent to $\lambda^3-\lambda^2+\frac{\norm{HMy}^2}{\norm{y}^2}\tau=0$. Then, the conclusion can be obtained by Lemma \ref{marden-ord-3}.
\end{proof}

\begin{proposition}[$1$-step one-shot method]\label{tau-1-shot-B=0 n+1}
	When $B=0$, the eigenvalue equation \eqref{eq-eigen-1-shot n+1} admits no solution $\lambda\in\C, |\lambda|\le 1$ if 
	$\tau<\frac{1}{\norm{H}^2\norm{M}^2}$.
\end{proposition}
\begin{proof}
	When $B=0$, equation \eqref{eq-eigen-1-shot n+1} becomes
	$\lambda(\lambda-1)\norm{y}^2+\tau\norm{HMy}^2=0$ which yields $\lambda^3-\lambda^2+\frac{\norm{HMy}^2}{\norm{y}^2}\tau\lambda=0$. Then, the conclusion can be obtained by Lemma \ref{marden-ord-3}.
\end{proof}

\subsection{Real eigenvalues}

We now find conditions on the descent step $\tau$ such that the real eigenvalues stay inside the unit disk. Recall that we have already proved that $\lambda=1$ is not an eigenvalue for both methods.
\begin{proposition}[shifted $1$-step one-shot method]\label{tau n k=1 real}
	Equation \eqref{eq-eigen-1-shot n} 
	\begin{enumerate}[label=(\roman*)]
		\item admits no solution $\lambda\in\R, \lambda>1$ for all $\tau>0$;
		\item admits no solution $\lambda\in\R, \lambda\le -1$ if we take $$\tau<\frac{2}{\norm{H}^2\norm{M}^2s(B)^2},$$
		where $s(B)$ is defined in Lemma~\ref{inv(I-T/z)}; 
		moreover if $0<\norm{B}<1$, we can take
		\begin{equation}\label{eq:tau n k=1 real}
			\tau<\frac{\chi_0(1,\norm{B})}{\norm{H}^2\norm{M}^2}, \quad \text{where } \; \chi_0(1,b) = 2(1-b)^2
		\end{equation}
		(here in the notation $\chi_0(1,b)$, $1$ refers to $k=1$).
	\end{enumerate}
\end{proposition}
\begin{proof}
	When $\lambda\in\R\backslash\{0\}$ equation \eqref{eq-eigen-1-shot n} becomes $$\lambda^2(\lambda-1)\norm{y}^2+\tau\norm{H(I-B/\lambda)^{-1}My}^2=0.$$
	The left-hand side of the above equation is strictly positive for any $\tau>0$ if $\lambda>1$; it is strictly negative for $\tau$ satisfying the inequality in (ii) if $\lambda\le -1$, noting that $\lambda \mapsto \lambda^2(\lambda-1)$ is increasing for $\lambda\le-1$.
\end{proof}

\begin{proposition}[$1$-step one-shot method]\label{tau n+1 k=1 real}
	Equation \eqref{eq-eigen-1-shot n+1} admits no solution $\lambda\in\R, \lambda\neq 1, |\lambda|\ge 1$ for all $\tau>0$.
\end{proposition}
\begin{proof}
	When $\lambda\in\R\backslash\{0\}$ equation \eqref{eq-eigen-1-shot n+1} becomes $$\lambda(\lambda-1)\norm{y}^2+\tau\norm{H(I-B/\lambda)^{-1}My}^2=0.$$
	If $\lambda\in\R,\lambda\neq 1,|\lambda|\ge 1$ then $\lambda(\lambda-1)>0$, thus the left-hand side of the above equation is strictly positive for any $\tau>0$.
\end{proof}

\subsection{Complex eigenvalues}
\label{subsec:complex-1-step}

We now look for conditions on the descent step $\tau$ such that also the complex eigenvalues stay inside the unit disk. We first deal with the shifted $1$-step one-shot method. 

\begin{proposition}[shifted $1$-step one-shot method]\label{tau n k=1 cpl}
	If $B\neq 0$, $\exists\tau>0$ sufficiently small such that equation \eqref{eq-eigen-1-shot n} admits no solution $\lambda\in\C\backslash\R, |\lambda|\ge 1$. 
	In particular, if $0<\norm{B}<1$, given any $\delta_0>0$ and $0<\theta_0\le\frac{\pi}{6}$, take
	\[
	\tau < \frac{\min \{ \chi_1(1,\norm{B}),\; \chi_2(1,\norm{B}),\; \chi_3(1,\norm{B}),\; \chi_4(1,\norm{B}) \} }{\norm{H}^2\norm{M}^2},
	\]
	where 
	$$\chi_1(1,b)=\frac{(1-b)^4}{4b^2},\quad\chi_2(1,b)=\cfrac{2\sin\frac{\theta_0}{2}(1-b)^2}{(1+b)^2},$$
	$$\chi_3(1,b)=\cfrac{\delta_0\cos^2\frac{5\theta_0}{2}}{2\left(1+2\delta_0\sin\frac{5\theta_0}{2}+\delta_0^2\right)}\cdot\cfrac{(1-b)^4}{b^2},\quad\chi_4(1,b)=\left[\sin\left(\frac{\pi}{2}-3\theta_0\right)+\cos2\theta_0\right](1-b)^2$$
	(here in the notation $\chi_i(1,b), i=1,\dots,4$, $1$ refers to $k=1$).
\end{proposition}
\begin{proof}
	\textbf{Step 1. Rewrite equation \eqref{eq-eigen-1-shot n} so that we can study its real and imaginary parts.}
	
	Let $\lambda=R(\cos\theta+\ic\sin\theta)$ in polar form where $R=|\lambda| \ge 1$ and $\theta\in(-\pi,\pi)$. Write ${1}/{\lambda}=r(\cos\phi+\ic\sin\phi)$ in polar form where $r={1}/{|\lambda|}={1}/{R}\le 1$ and $\phi=-\theta\in(-\pi,\pi)$. By Lemma \ref{decompq}, we have
	$$\left(I-\cfrac{B}{\lambda}\right)^{-1}=P(\lambda)+\ic Q(\lambda),\quad \left(I-\cfrac{B^*}{\lambda}\right)^{-1}=P(\lambda)^*+\ic Q(\lambda)^*$$
	where $P(\lambda)$ and $Q(\lambda)$ are $\C^{n_u\times n_u}$-valued functions, and, by omitting the dependence on $\lambda$, 
	\begin{equation}\label{eq:pB}
		\norm{P}\le p \coloneqq \left\{\begin{array}{cl}
			(1+\norm{B})s(B)^2 &\text{ for general }B\neq 0,\\
			\cfrac{1}{1-\norm{B}} &\text{ when }\norm{B}<1;
		\end{array}\right.
	\end{equation}
	\begin{equation}\label{eq:q1B}
		\norm{Q}\le q_1 \coloneqq \left\{\begin{array}{cl}
			\norm{B}s(B)^2 &\text{ for general }B\neq 0,\\
			\cfrac{\norm{B}}{1-\norm{B}} &\text{ when }0<\norm{B}<1;
		\end{array}\right.
	\end{equation}
	\begin{equation}\label{eq:q2B}
		\norm{Q}\le |\sin\theta|q_2,\quad q_2 \coloneqq \left\{\begin{array}{cl}
			\norm{B}s(B)^2 &\text{ for general }B\neq 0,\\
			\cfrac{\norm{B}}{(1-\norm{B})^2} &\text{ when }0<\norm{B}<1.
		\end{array}\right.
	\end{equation}
	Now we rewrite \eqref{eq-eigen-1-shot n} as
	\begin{equation}
		\label{ready-split-k=1}
		\lambda^2(\lambda-1)\norm{y}^2+\tau G(P^*+\ic Q^*,P+\ic Q)=0
	\end{equation}
	where 
	$$G(X,Y)=\scalar{M^*XH^*HYMy,y}\in\C, \quad X,Y\in\C^{n_u\times n_u}.$$
	$G$ satisfies the following properties:
	\begin{itemize}
		\item $\forall X,Y_1,Y_2\in\C^{n_u\times n_u}, \forall z_1,z_2\in\C$: \quad
		$G(X, z_1Y_1+z_2Y_2)=z_1G(X,Y_1)+z_2G(X,Y_2).$
		\item $\forall X_1,X_2,Y\in\C^{n_u\times n_u}, \forall z_1,z_2\in\C$: \quad
		$G(z_1X_1+z_2X_2,Y)=z_1G(X_1,Y)+z_2G(X_2,Y).$ 
		\item $\forall X\in\C^{n_u\times n_u}$: \quad
		$0\le G(X^*,X)=\norm{HXMy}^2\le (\norm{H}\norm{M}\norm{X})^2 \norm{y}^2.$
		\item $\forall X,Y\in\C^{n_u\times n_u}$: \quad $G(X,Y)+G(Y^*,X^*)\in\R$, indeed 
		$$\begin{array}{ll}
			G(X,Y)&=\scalar{M^*XH^*HYMy,y}=\scalar{y,M^*Y^*H^*HX^*My}\\
			&=\scalar{M^*Y^*H^*HX^*My,y}^*=G(Y^*,X^*)^*.
		\end{array}$$ 
	\end{itemize}
	With these properties of $G$, we expand \eqref{ready-split-k=1} and take its real and imaginary parts, so we respectively obtain:
	\begin{equation}\label{re-eq k=1}
		\Re(\lambda^3-\lambda^2)\norm{y}^2+\tau [G(P^*,P)- G(Q^*,Q)]=0
	\end{equation}
	and
	\begin{equation}\label{im-eq k=1}
		\Im(\lambda^3-\lambda^2)\norm{y}^2+\tau [G(P^*,Q)+ G(Q^*,P)]=0
	\end{equation}
	
	\noindent\textbf{Step 2. Find a suitable combination of equations \eqref{re-eq k=1} and \eqref{im-eq k=1}, choose $\tau$ so that we obtain a new equation with a left-hand side which is strictly positive/negative.}
	
	Let $\gamma=\gamma(\lambda)\in\R$, defined by cases as in Lemma~\ref{gamma123 n}. Multiplying equation \eqref{im-eq k=1} with $\gamma$ then summing it with equation \eqref{re-eq k=1}, we obtain:
	$$
	[\Re(\lambda^3-\lambda^2)+\gamma\Im(\lambda^3-\lambda^2)]\norm{y}^2+\tau[G(P^*,P)-G(Q^*,Q)+\gamma G(P^*,Q)+\gamma G(Q^*,P)]=0,
	$$
	or equivalently,
	\begin{equation}\label{im-gamma-re k=1}
		[\Re(\lambda^3-\lambda^2)+\gamma\Im(\lambda^3-\lambda^2)]\norm{y}^2+\tau G(P^*+\gamma Q^*,P+\gamma Q)-(1+\gamma^2)\tau G(Q^*,Q)=0.
	\end{equation}
	Now we consider four cases of $\lambda$ as in Lemma~\ref{gamma123 n}:
	\begin{itemize}
		\item\textit{Case 1.} $\Re(\lambda^3-\lambda^2)\ge 0$;
		\item\textit{Case 2.} $\Re(\lambda^3-\lambda^2)<0$ and $\theta\in [\theta_0,\pi-\theta_0]\cup[-\pi+\theta_0,-\theta_0]$ for fixed $0<\theta_0\le\frac{\pi}{6}$;
		\item\textit{Case 3.} $\Re(\lambda^3-\lambda^2)<0$ and $\theta \in (-\theta_0,\theta_0)$ for fixed $0<\theta_0\le \frac{\pi}{6}$;
		\item\textit{Case 4.} $\Re(\lambda^3-\lambda^2)<0$ and $\theta \in (\pi-\theta_0,\pi)\cup(-\pi,-\pi+\theta_0)$ for fixed $0<\theta_0\le \frac{\pi}{6}$. 
	\end{itemize}
	The four cases will be treated in the following four lemmas (Lemmas \ref{k=1 n case1}--\ref{k=1 n case4}), which together give the statement of this proposition.
\end{proof}
\begin{lemma}[Case 1]
	\label{k=1 n case1} Equation \eqref{eq-eigen-1-shot n} admits no solutions $\lambda$ in Case 1 if we take
	$$\tau<\frac{1}{4\norm{H}^2\norm{M}^2\norm{B}^2s(B)^4}.$$
	Moreover, if $0<\norm{B}<1$, we can take
	$$\tau<\frac{(1-\norm{B})^4}{4\norm{H}^2\norm{M}^2\norm{B}^2}.$$
\end{lemma}
\begin{proof}
	Writing \eqref{im-gamma-re k=1} for $\gamma=\gamma_1$ as in Lemma \ref{gamma123 n} (i) (in particular $\gamma_1^2=1$), we have
	\begin{equation}\label{case1 k=1}
		[\Re(\lambda^3-\lambda^2)+\gamma_1\Im(\lambda^3-\lambda^2)]\norm{y}^2+\tau G(P^*+\gamma_1 Q^*,P+\gamma_1 Q)-2\tau G(Q^*,Q)=0.
	\end{equation}
	By the properties of $G$ we have
	$$G(P^*+\gamma_1 Q^*,P+\gamma_1 Q) \ge 0$$
	and
	$$G(Q^*,Q)\le(\norm{H}\norm{M}\norm{Q})^2\norm{y}^2\le (\norm{H}\norm{M}|\sin\theta|q_2)^2\norm{y}^2,$$
	therefore the left-hand side of \eqref{case1 k=1} will be strictly positive if $\tau$ satisfies
	$$\tau<\frac{\Re(\lambda^3-\lambda^2)+\gamma_1\Im(\lambda^3-\lambda^2)}{2\left(\norm{H}\norm{M}|\sin\theta|q_2\right)^2}.$$
	Since $\Re(\lambda^3-\lambda^2)+\gamma_1\Im(\lambda^3-\lambda^2)\ge 2|\sin(\theta/2)|$ by Lemma \ref{gamma123 n} (i), it is enough to choose
	$$\tau<\cfrac{1}{4\left|\sin\frac{\theta}{2}\right|\cos^2\frac{\theta}{2}\norm{H}^2\norm{M}^2q_2^2}.$$
	Since $\left|\sin\frac{\theta}{2}\right|\cos^2\frac{\theta}{2}\le 1$, it is sufficient to choose 
	$\tau<\frac{1}{4\norm{H}^2\norm{M}^2q_2^2}$ and we use definition \eqref{eq:q2B} of $q_2$.
\end{proof}

\begin{lemma}[Case 2]
	\label{k=1 n case2} Equation \eqref{eq-eigen-1-shot n} admits no solutions $\lambda$ in Case 2 if we take
	$$\tau<\cfrac{2\sin\frac{\theta_0}{2}}{\norm{H}^2\norm{M}^2(1+2\norm{B})^2s(B)^4}.$$
	Moreover, if $0<\norm{B}<1$, we can take
	$$\tau<\cfrac{2\sin\frac{\theta_0}{2}(1-\norm{B})^2}{\norm{H}^2\norm{M}^2(1+\norm{B})^2}.$$
\end{lemma}
\begin{proof}
	Writing \eqref{im-gamma-re k=1} for $\gamma=\gamma_2$ as in Lemma \ref{gamma123 n} (ii) (in particular $\gamma_2^2=1$), we have
	\begin{equation}\label{case2 k=1}
		[\Re(\lambda^3-\lambda^2)+\gamma_2\Im(\lambda^3-\lambda^2)]\norm{y}^2+\tau G(P^*+\gamma_2 Q^*,P+\gamma_2 Q)-2\tau G(Q^*,Q)=0.
	\end{equation}
	By the properties of $G$
	$$G(Q^*,Q)\ge 0,\quad G(P^*+\gamma_2 Q^*,P+\gamma_2 Q)\le(\norm{H}\norm{M}\norm{P+\gamma_2 Q})^2\norm{y}^2$$
	and the estimate  
	$\norm{P+\gamma_2 Q}\le \norm{P}+|\gamma_2|\norm{Q}=\norm{P}+\norm{Q}\le p+q_1,$
	the left-hand side of \eqref{case2 k=1} will be strictly negative if $\tau$ satisfies:
	$$\tau<\frac{-\Re(\lambda^3-\lambda^2)-\gamma_2\Im(\lambda^3-\lambda^2)}{\left[\norm{H}\norm{M}(p+q_1)\right]^2}.$$
	Thanks to Lemma \ref{gamma123 n} (ii), it is sufficient to choose
	$$\tau<\cfrac{2\sin\frac{\theta_0}{2}}{\norm{H}^2\norm{M}^2(p+q_1)^2}$$
	and we use definitions \eqref{eq:pB} and \eqref{eq:q1B} of $p$ and $q_1$.
\end{proof}

\begin{lemma}[Case 3]
	\label{k=1 n case3} Let $\delta_0>0$ be fixed. Equation \eqref{eq-eigen-1-shot n} admits no solutions $\lambda$ in Case 3 if we take
	$$\tau<\cfrac{\delta_0\cos^2\frac{5\theta_0}{2}}{2\left(1+2\delta_0\sin\frac{5\theta_0}{2}+\delta_0^2\right)}\cdot\cfrac{1}{\norm{H}^2\norm{M}^2\norm{B}^2s(B)^4}.$$
	Moreover, if $0<\norm{B}<1$, we can take
	$$\tau<\cfrac{\delta_0\cos^2\frac{5\theta_0}{2}}{2\left(1+2\delta_0\sin\frac{5\theta_0}{2}+\delta_0^2\right)}\cdot\cfrac{(1-\norm{B})^4}{\norm{H}^2\norm{M}^2\norm{B}^2}.$$
\end{lemma}
\begin{proof}
	Writing \eqref{im-gamma-re k=1} for $\gamma=\gamma_3$ as in Lemma \ref{gamma123 n} (iii), we have
	\begin{equation}\label{case3 k=1}
		[\Re(\lambda^3-\lambda^2)+\gamma_3\Im(\lambda^3-\lambda^2)]\norm{y}^2+\tau G(P^*+\gamma_3 Q^*,P+\gamma_3 Q)-(1+\gamma_3^2)\tau G(Q^*,Q)=0.
	\end{equation}
	By the properties of $G$
	$$G(P^*+\gamma_3 Q^*,P+\gamma_3 Q) \ge 0,\quad G(Q^*,Q)\le(\norm{H}\norm{M}\norm{Q})^2\norm{y}^2$$
	and by the estimate $\norm{Q}\le|\sin\theta|q_2$, the left-hand side of \eqref{case3 k=1} will be strictly positive if $\tau$ satisfies:
	$$\tau<\frac{\Re(\lambda^2-\lambda)+\gamma_3\Im(\lambda^2-\lambda)}{(1+\gamma_3^2)\left(\norm{H}\norm{M}|\sin\theta|q_2\right)^2}.$$
	Since by Lemma \ref{gamma123 n} (iii) $\Re(\lambda^3-\lambda^2)+\gamma_3\Im(\lambda^3-\lambda^2)>2\delta_0\left|\sin\frac{\theta}{2}\right|$, it is sufficient to choose 
	$$\tau<\cfrac{\delta_0}{2(1+\gamma_3^2)\norm{H}^2\norm{M}^2q_2^2}=\cfrac{1}{2\norm{H}^2\norm{M}^2q_2^2}\cdot\cfrac{\delta_0\cos^2\frac{5\theta_0}{2}}{1+2\delta_0\sin\frac{5\theta_0}{2}+\delta_0^2},$$ 
	where we have used the definition of $\gamma_3$. To conclude we use definition \eqref{eq:q2B} of $q_2$.
\end{proof}

\begin{lemma}[Case 4]
	\label{k=1 n case4} Equation \eqref{eq-eigen-1-shot n} admits no solutions $\lambda$ in Case 4 if we take
	$$\tau<\cfrac{\sin\left(\frac{\pi}{2}-3\theta_0\right)+\cos2\theta_0}{\norm{H}^2\norm{M}^2(1+\norm{B})^2s(B)^2}.$$
	Moreover, if $0<\norm{B}<1$, we can take
	$$\tau<\left[\sin\left(\frac{\pi}{2}-3\theta_0\right)+\cos2\theta_0\right]\cfrac{(1-\norm{B})^2}{\norm{H}^2\norm{M}^2}.$$
\end{lemma}
\begin{proof}
	Here it is enough to consider \eqref{re-eq k=1}. By the properties of $G$
	$$G(Q^*,Q)\ge 0,\quad G(P^*,P)\le(\norm{H}\norm{M}p)^2\norm{y}^2$$
	we see that	the left-hand side of \eqref{re-eq k=1} will be strictly negative if $\tau$ satisfies
	$$\tau<\frac{-\Re(\lambda^3-\lambda^2)}{\left(\norm{H}\norm{M}p\right)^2}.$$
	Thanks to Lemma \ref{gamma123 n} (iv), it is sufficient to choose
	$$\tau<\cfrac{\sin\left(\frac{\pi}{2}-3\theta_0\right)+\cos2\theta_0}{\norm{H}^2\norm{M}^2p^2},$$
	and definition \eqref{eq:pB} of $p$ leads to the conclusion.
\end{proof}

Similarly, with the help of Lemma \ref{gamma123 n+1}, we prove for the $1$-step one-shot method the analogue of Proposition~\ref{tau n k=1 cpl}. In particular, note that here just three cases of $\lambda$ need to be considered, because the analogue of the fourth one is excluded by Lemma \ref{gamma123 n+1} (iv). 

\begin{proposition}[$1$-step one-shot method]\label{tau n+1 k=1 cpl}
	If $B\neq 0$, $\exists\tau>0$ sufficiently small such that equation \eqref{eq-eigen-1-shot n+1} admits no solution $\lambda\in\C\backslash\R, |\lambda|\ge 1$. 
	In particular, if $0<\norm{B}<1$, given any $\delta_0>0$ and $0<\theta_0\le\frac{\pi}{4}$, take
	\[
	\tau < \frac{\min \{ \psi_1(1,\norm{B}),\; \psi_2(1,\norm{B}),\; \psi_3(1,\norm{B}) \} }{\norm{H}^2\norm{M}^2},
	\]
	where 
	$$\psi_1(1,b)=\frac{(1-b)^4}{4b^2},\quad\psi_2(1,b)=\cfrac{2\sin\frac{\theta_0}{2}(1-b)^2}{(1+b)^2},\quad\psi_3(1,b)=\cfrac{\delta_0\cos^2\frac{3\theta_0}{2}(1-b)^4}{2\left(1+2\delta_0\sin\frac{3\theta_0}{2}+\delta_0^2\right)b^2}$$
	(here in the notation $\psi_i(1,b), i=1,2,3$, $1$ refers to $k=1$).
\end{proposition}

\subsection{Final result ($k=1$)}

Considering Proposition~\ref{tau-1-shot-B=0 n}, and taking the minimum between the bound \eqref{eq:tau n k=1 real} in Proposition~\ref{tau n k=1 real} for real eigenvalues and the bound in Proposition~\ref{tau n k=1 cpl} for complex eigenvalues, we obtain a sufficient condition on the descent step $\tau$ to ensure convergence of the shifted $1$-step one-shot method. 

\begin{theorem}[Convergence of shifted $1$-step one-shot] \label{th:tau n k=1 all} Under assumption \eqref{hypo}, the shifted $1$-step one-shot method \eqref{alg:1-shot n} converges for sufficiently small $\tau$.  In particular, for $\norm{B}<1$, it is enough to take 
	$$\tau<\frac{\chi(1,\norm{B})}{\norm{H}^2\norm{M}^2},$$
	where $\chi(1,\norm{B})$ is an explicit function of $\norm{B}$ (in this notation $1$ refers to $k=1$).
\end{theorem}
\begin{remark}
	Set $b=\norm{B}$. For $0<b<1$, a practical (but not optimal) bound for $\tau$ is
	$$\tau< \frac{1}{\norm{H}^2\norm{M}^2} \cdot \min\left\{\frac{1}{2}\cdot\frac{(1-b)^2}{(1+b)^2}, \; \frac{1-\sin\frac{5\pi}{12}}{4}\cdot\frac{(1-b)^4}{b^2}\right\}.$$
	Indeed, using the notation in Proposition~\ref{tau n k=1 real} and \ref{tau n k=1 cpl}, it is easy to show that $\chi_2(1,b)\le\chi_0(1,b)$ and $\chi_3(1,b)\le\chi_1(1,b)$. By studying $\chi_3(1,b)$ and noting that $\delta_0^2+1\ge 2\delta_0$, we see that we should take $\delta_0=1$. Finally, we can take for instance $\theta_0=\frac{\pi}{6}$, then compare $\chi_2(1,b)$, $\chi_3(1,b)$ and $\chi_4(1,b)$.
\end{remark}

Putting together Propositions~\ref{tau-1-shot-B=0 n+1}, \ref{tau n+1 k=1 real}, \ref{tau n+1 k=1 cpl}, we obtain a sufficient condition on the descent step $\tau$ to ensure convergence of the $1$-step one-shot method. 

\begin{theorem}[Convergence of $1$-step one-shot]
	\label{th:tau n+1 k=1 all}
	Under assumption \eqref{hypo}, the $1$-step one-shot method \eqref{alg:1-shot n+1} converges for sufficiently small $\tau$. In particular, for $\norm{B}<1$, it is enough to take 
	$$\tau<\frac{\psi(1,\norm{B})}{\norm{H}^2\norm{M}^2},$$
	where $\psi(1,\norm{B})$ is an explicit function of $\norm{B}$ (in this notation $1$ refers to $k=1$).
\end{theorem}	

\begin{remark} 	Similarly as above, for $0<b<1$, a practical (but not optimal) bound for $\tau$ is
	$$\tau< \frac{1}{\norm{H}^2\norm{M}^2} \cdot \min\left\{2\sin\frac{\pi}{8}\cdot\frac{(1-b)^2}{(1+b)^2}, \;\frac{1-\sin\frac{3\pi}{8}}{4}\cdot\frac{(1-b)^4}{b^2}\right\}.$$
\end{remark}

	\section{Convergence of multi-step one-shot methods ($k\ge 2$)}
\label{sec:multi-step}

We now tackle the multi-step case, that is the $k$-step one-shot methods with $k\ge 2$. 

\subsection{Block iteration matrices and eigenvalue equations}
\label{sec:eigen-eq}

Once again, to analyze the convergence of these methods, first we express $(\sigma^{n+1},u^{n+1},p^{n+1})$ in terms of $(\sigma^n,u^n,p^n)$, by rewriting the recursions for $u$ and $p$: systems \eqref{alg:k-shot n+1} and \eqref{alg:k-shot n} are respectively rewritten as
\begin{equation}
	\label{k-shot expl n+1}
	\begin{cases}
		\sigma^{n+1}=\sigma^n-\tau M^*p^n\\
		u^{n+1}=B^ku^n+T_kM\sigma^n-\tau T_kMM^*p^n+T_kF\\
		p^{n+1}=[(B^*)^k-\tau X_kMM^*]p^n+U_ku^n+X_kM\sigma^n+X_kF-T_k^*H^*f\\
	\end{cases}
\end{equation}
and
\begin{equation}
	\label{k-shot expl n}
	\begin{cases}
		\sigma^{n+1}=\sigma^n-\tau M^*p^n\\
		u^{n+1}=B^ku^n+T_kM\sigma^n+T_kF\\
		p^{n+1}=(B^*)^kp^n+U_ku^n+X_kM\sigma^n+X_kF-T_k^*H^*f\\
	\end{cases}
\end{equation}
where
\begin{equation}\label{eq:T}
	T_k=I+B+...+B^{k-1}=(I-B)^{-1}(I-B^k), \quad k\ge 1,
\end{equation}
$$U_k=(B^*)^{k-1}H^*H+(B^*)^{k-2}H^*HB+...+H^*HB^{k-1}, \quad k\ge 1,$$
\begin{equation}\label{eq:X}
	X_k=\left\{\begin{array}{cl}
		(B^*)^{k-2}H^*HT_1+(B^*)^{k-3}H^*HT_2+...+H^*HT_{k-1} & \text{if }k\ge 2,\\
		0 & \text{if }k=1. 
	\end{array}\right.
\end{equation}
Note that \eqref{k-shot expl n+1} ($k$-step one-shot) can be obtained from \eqref{k-shot expl n} (shifted $k$-step one-shot) by replacing $\sigma^n$ with $\sigma^{n+1}=\sigma^n-\tau M^*p^n$ in the equations for $u$ and $p$, which yields two extra terms in \eqref{k-shot expl n+1}. In what follows we first study the shifted $k$-step one-shot method then the $k$-step one-shot method. The following lemma gathers some useful properties of $T_k, U_k$ and $X_k$.
\begin{lemma}\label{lem:propTUX} 
	\begin{enumerate}[label=(\roman*)]
		\item The matrices $U_k$ and $X_k$ can be rewritten as 
		\begin{equation*} 
			\begin{split}
				& U_k=\sum_{i+j=k-1} (B^*)^iH^*HB^j \quad \text{for } k\ge 1, \\
				& X_k=\sum_{l=0}^{k-2}\sum_{i+j=l} (B^*)^iH^*HB^j=\sum_{l=1}^{k-1}U_l \quad \text{for } k\ge 2.
			\end{split}
		\end{equation*}
		\item The matrices $U_k$ and $X_k$ are self-adjoint: $U_k^*=U_k$, $X_k^*=X_k$.
		\item We have the relation 
		\begin{equation}\label{uxt}
			U_kT_k-X_kB^k+X_k=T_k^*H^*HT_k, \quad\forall k\ge 1.
		\end{equation}
	\end{enumerate}
\end{lemma}	
\begin{proof}
	(i) is easy to check by the definitions. (ii) follows from (i). 

	\noindent (iii) For $k=1$, we have $U_1=H^*H$, $T_1=I$ and $X_1=0$, hence the identity is verified. For $k\ge 2$, note that $X_{k+1}=B^*X_k+H^*HT_k$, then by (ii) $X_{k+1}=X_{k+1}^*=X_kB+T_k^*H^*H$. 
	On the other hand, from (i) we get that $X_{k+1}=X_k+U_k$. Thus,
	$$X_k+U_k=X_kB+T_k^*H^*H,\quad\mbox{ or equivalently, }\quad U_k=X_k(B-I)+T_k^*H^*H.$$
	Finally,
	$$U_kT_k=X_k(B-I)T_k+T_k^*H^*HT_k=X_k(B^k-I)+T_k^*H^*HT_k.$$
\end{proof}

Now, we consider the errors $(\sigma^n-\sigma^\text{ex},u^n-u(\sigma^\text{ex}),p^n-p(\sigma^\text{ex}))$ with respect to the exact solution  at the $n$-th iteration, and, by abuse of notation, we designate them by $(\sigma^n,u^n,p^n)$. We obtain that the errors satisfy: for the shifted algorithm \eqref{k-shot expl n}  
\begin{equation}
	\label{k-shot-err expl n}
	\begin{cases}
		\sigma^{n+1}=\sigma^n-\tau M^*p^n\\
		u^{n+1}=B^ku^n+T_kM\sigma^n\\
		p^{n+1}=(B^*)^kp^n+U_ku^n+X_kM\sigma^n\\
	\end{cases}
\end{equation}
and for algorithm \eqref{k-shot expl n+1}
\begin{equation}
	\label{k-shot-err expl n+1}
	\begin{cases}
		\sigma^{n+1}=\sigma^n-\tau M^*p^n\\
		u^{n+1}=B^ku^n+T_kM\sigma^n-\tau T_kMM^*p^n\\
		p^{n+1}=[(B^*)^k-\tau X_kMM^*]p^n+U_ku^n+X_kM\sigma^n,\\
	\end{cases}
\end{equation}
or equivalently, by putting in evidence the block iteration matrices 
\begin{equation}
	\label{itermat n}
	\begin{bmatrix}
		p^{n+1}\\ u^{n+1}\\ \sigma^{n+1}
	\end{bmatrix}=\begin{bmatrix}
		(B^*)^k & U_k & X_kM \\
		0 & B^k & T_kM \\
		-\tau M^* & 0 & I
	\end{bmatrix}
	\begin{bmatrix}
		p^{n}\\ u^{n}\\ \sigma^{n}
	\end{bmatrix}
\end{equation}
and 
\begin{equation}
	\label{itermat n+1}
	\begin{bmatrix}
		p^{n+1}\\ u^{n+1}\\ \sigma^{n+1}
	\end{bmatrix}=\begin{bmatrix}
		(B^*)^k-\tau X_kMM^* & U_k & X_kM \\
		-\tau T_kMM^* & B^k & T_kM \\
		-\tau M^* & 0 & I
	\end{bmatrix}
	\begin{bmatrix}
		p^{n}\\ u^{n}\\ \sigma^{n}
	\end{bmatrix}.
\end{equation}
Now recall that a fixed point iteration converges if and only if the spectral radius of its iteration matrix is strictly less than $1$. Therefore in the following propositions we establish eigenvalue equations for the iteration matrix of the two methods. 

\begin{proposition}[Eigenvalue equation for the shifted $k$-step one-shot method]\label{prop:eq-eigen shift-k-shot} 
	Assume that $\lambda\in\C$ is an eigenvalue of the iteration matrix in \eqref{itermat n}.
	\begin{enumerate}[label=(\roman*)]
		\item If $\lambda\in\C$, $\lambda\notin\mathrm{Spec}(B^k)$, then $\exists \, y\in\C^{n_\sigma}, y\neq 0$ such that
		\begin{equation}\label{ori-eq-eigen n}
			(\lambda-1)\norm{y}^2+\tau\scalar{M^*[\lambda I-(B^*)^k]^{-1}[(\lambda-1)X_k+T_k^*H^*HT_k](\lambda I-B^k)^{-1}My,y}=0.
		\end{equation}
		\item $\lambda=1$ is not an eigenvalue of the iteration matrix.
	\end{enumerate}
\end{proposition}

\begin{proposition}[Eigenvalue equation for the $k$-step one-shot method]
	\label{prop:eq-eigen k-shot} 
	Assume that $\lambda\in\C$ is an eigenvalue of the iteration matrix in \eqref{itermat n+1}.
	\begin{enumerate}[label=(\roman*)]
		\item If $\lambda\in\C$, $\lambda\notin\mathrm{Spec}(B^k)$ then $\exists \, y\in\C^{n_\sigma}, y\neq 0$ such that:
		\begin{equation}\label{ori-eq-eigen n+1}
			(\lambda-1)\norm{y}^2+\tau\lambda\scalar{M^*[\lambda I-(B^*)^k]^{-1}[(\lambda-1)X_k+T_k^*H^*HT_k](\lambda I-B^k)^{-1}My,y}=0.
		\end{equation}
		\item $\lambda=1$ is not an eigenvalue of the iteration matrix.
	\end{enumerate}
\end{proposition}
\begin{remark}
	Since $\rho(B)$ is strictly less than $1$, so are $\rho(B^*), \rho(B^k)$ and $\rho((B^*)^k)$.
\end{remark}		
\noindent The proofs for Propositions~\ref{prop:eq-eigen shift-k-shot} and  \ref{prop:eq-eigen k-shot} are respectively similar to the ones of Propositions~\ref{prop:eq-eigen shift-1-shot} and \ref{prop:eq-eigen 1-shot}, the slight difference is that in the calculation we use \eqref{uxt} to simplify some terms.

In the following sections we will show that, for sufficiently small $\tau$, equations \eqref{ori-eq-eigen n} and \eqref{ori-eq-eigen n+1} admit no solution $|\lambda|\ge 1$, thus algorithms \eqref{alg:k-shot n} and \eqref{alg:k-shot n+1} converge. When $\lambda\neq 0$, it is convenient to rewrite \eqref{ori-eq-eigen n} and \eqref{ori-eq-eigen n+1} respectively as
\begin{equation}\label{eq-eigen n}
	\lambda^2(\lambda-1)\norm{y}^2+\tau\scalar{M^*\left[I-(B^*)^k/\lambda\right]^{-1}[(\lambda-1)X_k+T_k^*H^*HT_k]\left(I-B^k/\lambda\right)^{-1}My,y}=0
\end{equation}
and
\begin{equation}\label{eq-eigen n+1}
	\lambda(\lambda-1)\norm{y}^2+\tau\scalar{M^*\left[I-(B^*)^k/\lambda\right]^{-1}[(\lambda-1)X_k+T_k^*H^*HT_k]\left(I-B^k/\lambda\right)^{-1}My,y}=0
\end{equation}

The scalar case where $n_u, n_\sigma, n_f =1$ is analyzed in Appendix~\ref{app:1D}.

\begin{remark}\label{rmk:B=0,k>1}
	Note that when $B=0$ and $k \ge 2$, the shifted $k$-step one-shot and $k$-step one-shot are respectively equivalent to the shifted and usual gradient descent methods, therefore we retrieve the same bounds \eqref{best-tau-shifted-gd}--\eqref{best-tau-usualgd} for the descent step $\tau$ as for those methods. 
\end{remark}

For the analysis we use auxiliary results proved in Appendix~\ref{app:lems}, and the following bounds for $s(B^k), T_k, X_k$.
\begin{lemma}\label{lem:boundsSTXk}
	If $\norm{B}<1$, 
	\[
	s(B^k)\le\frac{1}{1-\norm{B}^k}, \quad
	\norm{T_k} \le \frac{1-\norm{B}^k}{1-\norm{B}}, \quad
	\norm{X_k} \le \frac{\norm{H}^2(1-k\norm{B}^{k-1}+(k-1)\norm{B}^k)}{(1-\norm{B})^2}.
	\]
\end{lemma}
\begin{proof}
	The bound for $s(B^k)$ is proved using Lemma~\ref{inv(I-T/z)} and  $\norm{B^k} \le \norm{B}^k$. Next, from \eqref{eq:T} we have
	$$\norm{T_k}\le 1+\norm{B}+...+\norm{B}^{k-1}=\frac{1-\norm{B}^k}{1-\norm{B}}.$$
	From \eqref{eq:X}, 
	if $k\ge 2$ we have
	$$\begin{array}{cl}
		\norm{X_k}&\le\norm{H}^2 \bigl(\norm{B}^{k-2}+\norm{B}^{k-3}(1+\norm{B})+...+(1+\norm{B}+...+\norm{B}^{k-2})\bigr)\\
		&=\displaystyle\norm{H}^2(1+2\norm{B}+...+(k-1)\norm{B}^{k-2})=\frac{\norm{H}^2(1-k\norm{B}^{k-1}+(k-1)\norm{B}^k)}{(1-\norm{B})^2}.
	\end{array}
	$$
\end{proof}

\subsection{Real eigenvalues}
We first find conditions on the descent step $\tau$ such that the real eigenvalues stay inside the unit disk. Recall that we have already proved that $\lambda=1$ is not an eigenvalue for any $k$.
\begin{proposition}[shifted $k$-step one-shot method]\label{tau n k>1 real}
	When $k\ge 2$, $\exists\tau>0$ sufficiently small such that equation \eqref{eq-eigen n} admits no solution $\lambda\in\R, \lambda\neq 1, |\lambda|\ge 1$. More precisely, take
	\begin{itemize}
		\item $\tau<\frac{2}{\norm{M}^2\left(\norm{H}^2\norm{T_k}^2+2\norm{X_k}\right)s(B^k)^2}$
		if the denominator of the right-hand side is not $0$;
		\item any $\tau>0$ otherwise.
	\end{itemize}
	Moreover, if $\norm{B}<1$, we can take
	$$\tau<\frac{(1-\norm{B})^2}{\norm{H}^2\norm{M}^2}\cdot\frac{2(1-\norm{B}^k)^2}{(1-\norm{B}^k)^2+2(1-k\norm{B}^{k-1}+(k-1)\norm{B}^k)}.$$
\end{proposition}
\begin{proof}
	When $\lambda\in\R$ equation \eqref{eq-eigen n} is rewritten as
	$$
	\begin{array}{ll}
		\lambda^2(\lambda-1)\norm{y}^2+\tau\norm{HT_k\left(I-\frac{B^k}{\lambda}\right)^{-1}My}^2&\\
		+\tau(\lambda-1)\scalar{M^*\left[I-\frac{(B^*)^k}{\lambda}\right]^{-1}X_k\left(I-\frac{B^k}{\lambda}\right)^{-1}My,y}&=0.
	\end{array}
	$$
	We show that if $\lambda>1$ (or respectively $\lambda\le-1$) we can choose $\tau$ so that the left-hand side of the above equation is strictly positive (or respectively negative). Indeed, if $\lambda>1$, we choose $\tau$ such that
	\[
	\lambda^2\norm{y}^2-\tau\left|\scalar{M^*\left[I-\frac{(B^*)^k}{\lambda}\right]^{-1}X_k\left(I-\frac{B^k}{\lambda}\right)^{-1}My,y}\right|>0
	\]
	and this can be done by taking $\tau$ such that $$[\norm{X_k}\norm{M}^2s(B^k)^2]\tau<1.$$ 
	If $\lambda\le-1$, we choose $\tau$ such that
	$$\begin{array}{ll}
		\lambda^2(\lambda-1)\norm{y}^2+\tau\norm{HT_k\left(I-\frac{B^k}{\lambda}\right)^{-1}My}^2&\\
		+\tau(1-\lambda)\left|\scalar{M^*\left[I-\frac{(B^*)^k}{\lambda}\right]^{-1}X_k\left(I-\frac{B^k}{\lambda}\right)^{-1}My,y}\right|&<0
	\end{array}$$
	and this can be done by taking $\tau$ such that
	$$\left[\frac{\norm{H}^2\norm{T_k}^2\norm{M}^2s(B^k)^2}{2}+\norm{X_k}\norm{M}^2s(B^k)^2\right]\tau<1,$$
	so we obtain the first conclusion. Finally, the second conclusion in the case $\norm{B}<1$ can be obtained by Lemma \ref{lem:boundsSTXk}.
\end{proof}

\begin{proposition}[$k$-step one-shot method]\label{tau n+1 k>1 real}
	When $k\ge 2$, $\exists\tau>0$ sufficiently small such that equation \eqref{eq-eigen n+1} admits no solution $\lambda\in\R, \lambda\neq 1, |\lambda|\ge 1$. More precisely, take
	\begin{itemize}
		\item $\tau<\cfrac{1}{\norm{X_k}\norm{M}^2s(B^k)^2}$
		if the denominator of the right-hand side is not $0$;
		\item any $\tau>0$ otherwise.
	\end{itemize}
	Moreover, if $\norm{B}<1$, we can take
	$$\tau<\frac{(1-\norm{B})^2}{\norm{H}^2\norm{M}^2}\cdot\frac{(1-\norm{B}^k)^2}{1-k\norm{B}^{k-1}+(k-1)\norm{B}^k}.$$
\end{proposition}
\begin{proof}
	When $\lambda\in\R$ equation \eqref{eq-eigen n+1} is rewritten as
	$$
	\begin{array}{ll}
		\lambda(\lambda-1)\norm{y}^2+\tau\norm{HT_k\left(I-\frac{B^k}{\lambda}\right)^{-1}My}^2&\\
		+\tau(\lambda-1)\scalar{M^*\left[I-\frac{(B^*)^k}{\lambda}\right]^{-1}X_k\left(I-\frac{B^k}{\lambda}\right)^{-1}My,y}&=0.
	\end{array}
	$$
	We show that we can choose $\tau$ so that the left-hand side of the above equation is strictly positive. Indeed, if $\lambda>1$, we choose $\tau$ such that
	\[
	\lambda\norm{y}^2-\tau\left|\scalar{M^*\left[I-\frac{(B^*)^k}{\lambda}\right]^{-1}X_k\left(I-\frac{B^k}{\lambda}\right)^{-1}My,y}\right|>0
	\]
	and this can be done by taking $\tau$ such that
	$$\norm{X_k}\norm{M}^2s(B^k)^2\tau<1.$$
	If $\lambda\le-1$, we choose $\tau$ such that
	\[
	\lambda\norm{y}^2+\tau\left|\scalar{M^*\left[I-\frac{(B^*)^k}{\lambda}\right]^{-1}X_k\left(I-\frac{B^k}{\lambda}\right)^{-1}My,y}\right|<0
	\]
	and this is also done by taking $\tau$ such that
	$$\norm{X_k}\norm{M}^2s(B^k)^2\tau<1.$$
	so we obtain the first conclusion. Finally, the conclusion in the case $\norm{B}<1$ can be obtained by Lemma \ref{lem:boundsSTXk}.
\end{proof}

\subsection{Complex eigenvalues}
\label{subsec:complex-k-step}

We now look for conditions on the descent step $\tau$ such that also the complex eigenvalues stay inside the unit disk. We first deal with the shifted $k$-step one-shot method. 

\begin{proposition}[shifted $k$-step one-shot method]\label{tau n k>1 cpl}
	When $k\ge 2$, $\exists\tau>0$ sufficiently small such that equation \eqref{eq-eigen n} admits no solution $\lambda\in\C\backslash\R$, $|\lambda|\ge 1$. 
	In particular, if $\norm{B}<1$, given any $\delta_0>0$ and $0<\theta_0<\frac{\pi}{6}$, take
	\[
	\tau < \frac{\min \{ \chi_1(k,\norm{B}),\; \chi_2(k,\norm{B}),\; \chi_3(k,\norm{B}),\; \chi_4(k,\norm{B}) \} }{\norm{H}^2\norm{M}^2}
	\]
	where 
	\[
	\chi_1(k,b)=\frac{(1-b)^2(1-b^k)^2}{4b^{2k} + \sqrt{2}(1-k b^{k-1}+(k-1)b^k)(1+b^k)^2}
	\]
	\[
	\chi_2(k,b)=\frac{(1-b)^2(1-b^k)^2}{\bigl[\frac{1}{2\sin(\theta_0/2)}(1-b^k)^2+\sqrt{2}(1-k b^{k-1}+(k-1)b^k)\bigr](1+b^k)^2}
	\]
	\[
	\chi_3(k,b)=\frac{(1-b)^2(1-b^k)^2}{ \frac{2c\sin(\theta_0/2)}{\delta_0}b^{2k} +
		(1-kb^{k-1}+(k-1)b^k) \Bigl[ \frac{\sqrt{c}}{\delta_0}(1+b^{2k}) + 2\max\Bigl(\frac{\sqrt{c}}{\delta_0},\frac{\sqrt{c}}{\cos3\theta_0}\Bigr)b^k \Bigr] }
	\]
	\[
	\chi_4(k,b)=\frac{\left[\sin\left(\frac{\pi}{2}-3\theta_0\right)+\cos2\theta_0\right](1-b)^2(1-b^k)^2}{(1-b^k)^2+2(1-kb^{k-1}+(k-1)b^k)(1+b^k)^2}
	\]
	and $c=\frac{1+2\delta_0\sin\frac{5\theta_0}{2}+\delta_0^2}{\cos^2\frac{5\theta_0}{2}}$. 
\end{proposition}

	\begin{proof}
		\textbf{Step 1. Rewrite equation \eqref{eq-eigen n} so that we can study its real and imaginary parts.}
		
		Let $\lambda=R(\cos\theta+\ic\sin\theta)$ in polar form where $R=|\lambda| \ge 1$ and $\theta\in(-\pi,\pi)$. Write ${1}/{\lambda}=r(\cos\phi+\ic\sin\phi)$ in polar form where $r={1}/{|\lambda|}={1}/{R}\le 1$ and $\phi=-\theta\in(-\pi,\pi)$. 
		By Lemma \ref{decompq} applied to $T=B^k$, we have
		$$\left(I-\cfrac{B^k}{\lambda}\right)^{-1}=P(\lambda)+\ic Q(\lambda),\quad \left(I-\cfrac{(B^*)^k}{\lambda}\right)^{-1}=P(\lambda)^*+\ic Q(\lambda)^*$$
		where $P(\lambda)$ and $Q(\lambda)$ are $\C^{n_u\times n_u}$-valued functions, and, by omitting the dependence on $\lambda$, 
		\begin{equation}\label{eq:pBk}
			\norm{P}\le p \coloneqq \left\{\begin{array}{cl}
				(1+\norm{B^k})s(B^k)^2 &\text{ for general }B,\\
				\cfrac{1}{1-\norm{B}^k} &\text{ when }\norm{B}<1;
			\end{array}\right.
		\end{equation}
		\begin{equation}\label{eq:q1Bk}
			\norm{Q}\le q_1 \coloneqq \left\{\begin{array}{cl}
				\norm{B^k}s(B^k)^2 &\text{ for general }B,\\
				\cfrac{\norm{B}^k}{1-\norm{B}^k} &\text{ when }\norm{B}<1;
			\end{array}\right.
		\end{equation}
		\begin{equation}\label{eq:q2Bk}
			\norm{Q}\le q_2|\sin\theta|,\quad q_2 \coloneqq \left\{\begin{array}{cl}
				\norm{B^k}s(B^k)^2 &\text{ for general }B,\\
				\cfrac{\norm{B}^k}{(1-\norm{B}^k)^2} &\text{ when }\norm{B}<1.
			\end{array}\right.
		\end{equation}

		\noindent Now we rewrite \eqref{eq-eigen n} as
		\begin{equation}
			\label{ready-split n k>1}
			\lambda^2(\lambda-1)\norm{y}^2+\tau G(P^*+\ic Q^*,P+\ic Q)+\tau(\lambda-1)L(P^*+\ic Q^*,P+\ic Q)=0.
		\end{equation}
		where 
		$$G(X,Y)=\scalar{M^*XT_k^*H^*HT_kYMy,y},\quad L(X,Y)=\scalar{M^*XX_kYMy,y}$$
		for $X,Y\in\C^{n_u\times n_u}$.	$G$ satisfies the following properties:
		\begin{itemize}
			\item $\forall X,Y_1,Y_2\in\C^{n_u\times n_u}, \forall z_1,z_2\in\C$: \quad
			$G(X, z_1Y_1+z_2Y_2)=z_1G(X,Y_1)+z_2G(X,Y_2).$
			\item $\forall X_1,X_2,Y\in\C^{n_u\times n_u}, \forall z_1,z_2\in\C$: \quad
			$G(z_1X_1+z_2X_2,Y)=z_1G(X_1,Y)+z_2G(X_2,Y).$
			\item $\forall X\in\C^{n_u\times n_u}$: \quad
			$G(X^*,X)\in \R$.
			\item $\forall X,Y\in\C^{n_u\times n_u}$: \quad $G(X,Y)+G(Y^*,X^*)\in\R$, indeed
			$$\begin{array}{ll}
				G(X,Y)&=\scalar{M^*XT_k^*H^*HT_kYMy,y}=\scalar{y,M^*Y^*T_k^*H^*HT_kX^*My}\\
				&=\scalar{M^*Y^*T_k^*H^*HT_kX^*My,y}^*=G(Y^*,X^*)^*.
			\end{array}$$ 
		\end{itemize}
		Similarly, $L$ has the same properties as $G$ (note that $X_k^*=X_k$ by Lemma \ref{lem:propTUX}). With these properties of $G$ and $L$, we expand \eqref{ready-split n k>1} and take its real and imaginary parts, so we respectively obtain:
		\begin{equation}\label{re-eq k>1}
			\Re(\lambda^3-\lambda^2)\norm{y}^2+\tau G_1+\tau [\Re(\lambda-1)L_1-\Im(\lambda-1)L_2]=0
		\end{equation}
		and
		\begin{equation}\label{im-eq k>1}
			\Im(\lambda^3-\lambda^2)\norm{y}^2+\tau G_2+\tau [\Im(\lambda-1)L_1+\Re(\lambda-1)L_2]=0
		\end{equation}
		where
		$$G_1=G(P^*,P)-G(Q^*,Q),\quad G_2=G(P^*,Q)+G(Q^*,P),$$
		$$L_1=L(P^*,P)-L(Q^*,Q),\quad L_2=L(P^*,Q)+L(Q^*,P).$$
		
		\noindent\textbf{Step 2. Find a suitable combination of equations \eqref{re-eq k>1} and \eqref{im-eq k>1}, choose $\tau$ so that we obtain a new equation with a left-hand side which is strictly positive/negative.}
		
		Let $\gamma=\gamma(\lambda)\in\R$, defined by cases as in Lemma~\ref{gamma123 n}. Multiplying equation \eqref{im-eq k>1} with $\gamma$ then summing it with equation \eqref{re-eq k>1}, we obtain:
		\begin{equation}\label{im-gamma-re k>1}
			\begin{array}{ll}
				[\Re(\lambda^3-\lambda^2)+\gamma\Im(\lambda^3-\lambda^2)]\norm{y}^2+\tau G(P^*+\gamma Q^*,P+\gamma Q)-(1+\gamma^2)\tau G(Q^*,Q)&\\
				+\tau\left([\Re(\lambda-1)+\gamma\Im(\lambda-1)]L_1+[\gamma\Re(\lambda-1)-\Im(\lambda-1)]L_2\right)  &=0.
			\end{array}
		\end{equation}
		Now we prepare some useful estimates.
		\begin{itemize}
			\item $\forall X\in\C^{n_u\times n_u}$: \quad
			$0\le G(X^*,X)=\norm{HT_kXMy}^2\le (\norm{H}\norm{T_k}\norm{M}\norm{X})^2 \norm{y}^2.$
			
			Since $\norm{Q}\le q_1$ and $\norm{Q}\le q_2|\sin\theta|$, we have
			$$G(Q^*,Q)\le (\norm{H}\norm{T_k}\norm{M}q_1)^2 \norm{y}^2
			\text{ and }
			G(Q^*,Q)\le (\norm{H}\norm{T_k}\norm{M}q_2\sin|\theta|)^2 \norm{y}^2.$$
			\item By Cauchy-Schwarz inequality we have
			$$|\Re(\lambda-1)+\gamma\Im(\lambda-1)|\le \sqrt{1+\gamma^2}|\lambda-1|; \quad
			|\gamma\Re(\lambda-1)-\Im(\lambda-1)|\le \sqrt{1+\gamma^2}|\lambda-1|.$$
			\item $\forall X,Y\in\C^{n_u\times n_u}$: \quad
			$|L(X,Y)|=|\scalar{M^*XX_kYMy,y}|\le\norm{X_k}\norm{M}^2\norm{X}\norm{Y}\norm{y}^2.$
			Hence
			$$\begin{array}{ll}
				|L_1|&=|L(P^*,P)-L(Q^*,Q)|
				\le |L(P^*,P)|+|L(Q^*,Q)|\\
				&\le \norm{X_k}\norm{M}^2(\norm{P}^2+\norm{Q}^2)\norm{y}^2
				\le \norm{X_k}\norm{M}^2(p^2+q_1^2)\norm{y}^2,
			\end{array}$$
			$$\begin{array}{ll}
				|L_2|&=|L(P^*,Q)+L(Q^*,P)|
				\le |L(P^*,Q)|+|L(Q^*,P)|\\
				&\le 2\norm{X_k}\norm{M}^2\norm{P}\norm{Q}\norm{y}^2
				\le 2\norm{X_k}\norm{M}^2pq_1\norm{y}^2,
			\end{array}$$
			and then
			$$\begin{array}{ll}
				&|[\Re(\lambda-1)+\gamma\Im(\lambda-1)]L_1+[\gamma\Re(\lambda-1)-\Im(\lambda-1)]L_2|\\
				\le&|\Re(\lambda-1)+\gamma\Im(\lambda-1)||L_1|+|\gamma\Re(\lambda-1)-\Im(\lambda-1)||L_2|\\
				\le&\sqrt{1+\gamma^2}|\lambda-1|\norm{X_k}\norm{M}^2(p^2+q_1^2+2pq_1)\norm{y}^2\\
				=&\sqrt{1+\gamma^2}|\lambda-1|\norm{X_k}\norm{M}^2(p+q_1)^2\norm{y}^2.
			\end{array}
			$$
		\end{itemize}
		Now we consider four cases of $\lambda$ as in Lemma~\ref{gamma123 n}:
		\begin{itemize}
			\item\textit{Case 1.} $\Re(\lambda^3-\lambda^2)\ge 0$;
			\item\textit{Case 2.} $\Re(\lambda^3-\lambda^2)<0$ and $\theta\in [\theta_0,\pi-\theta_0]\cup[-\pi+\theta_0,-\theta_0]$ for fixed $0<\theta_0<\frac{\pi}{6}$;
			\item\textit{Case 3.} $\Re(\lambda^3-\lambda^2)<0$ and $\theta \in (-\theta_0,\theta_0)$ for fixed $0<\theta_0<\frac{\pi}{6}$;
			\item\textit{Case 4.} $\Re(\lambda^3-\lambda^2)<0$ and $\theta \in (\pi-\theta_0,\pi)\cup(-\pi,-\pi+\theta_0)$ for fixed $0<\theta_0<\frac{\pi}{6}$. 
		\end{itemize}
		The four cases will be treated in the following four lemmas (Lemmas \ref{k>1 n case1}--\ref{k>1 n case4}), which together give the statement of this proposition.
	\end{proof}
	
	\begin{lemma}[Case 1]
		\label{k>1 n case1} For $k\ge 2$, equation \eqref{eq-eigen n} admits no solutions $\lambda$ in Case 1 if we take
		\begin{itemize}
			\item $\tau<\cfrac{s(B^k)^{-4}}{4\norm{H}^2\norm{M}^2\norm{T_k}^2\norm{B^k}^2+\sqrt{2}\norm{M}^2\norm{X_k}(1+2\norm{B^k})^2}$ if the denominator of the right-hand side is not $0$;
			\item any $\tau>0$ otherwise.
		\end{itemize}
		Moreover, if $\norm{B}<1$, we can take
		$$\tau<\frac{(1-\norm{B})^2}{\norm{H}^2\norm{M}^2}\cdot\frac{(1-\norm{B}^k)^2}{4\norm{B}^{2k}+\sqrt{2}(1-k\norm{B}^{k-1}+(k-1)\norm{B}^k)(1+\norm{B}^k)^2}.$$
	\end{lemma}
	\begin{proof}
		Writing \eqref{im-gamma-re k>1} for $\gamma=\gamma_1$ as in Lemma \ref{gamma123 n} (i) (in particular $\gamma_1^2=1$), we have
		\begin{equation}\label{case1 k>1}
			\begin{array}{ll}
				[\Re(\lambda^3-\lambda^2)+\gamma_1\Im(\lambda^3-\lambda^2)]\norm{y}^2+\tau G(P^*+\gamma_1 Q^*,P+\gamma_1 Q)-2\tau G(Q^*,Q)&\\
				+\tau\left([\Re(\lambda-1)+\gamma_1\Im(\lambda-1)]L_1+[\gamma_1\Re(\lambda-1)-\Im(\lambda-1)]L_2\right)  &=0.
			\end{array}
		\end{equation}
		Since $G(P^*+\gamma_1 Q^*,P+\gamma_1 Q)\ge 0$, and by estimating
		\[
		G(Q^*,Q)\le(\norm{H}\norm{T_k}\norm{M}q_2\sin|\theta|)^2 \norm{y}^2,
		\]
		\begin{multline*}
			[\Re(\lambda-1)+\gamma_1\Im(\lambda-1)]L_1+[\gamma_1\Re(\lambda-1)-\Im(\lambda-1)]L_2\\
			\ge-\sqrt{2}|\lambda-1|\norm{X_k}\norm{M}^2(p+q_1)^2\norm{y}^2,
		\end{multline*}
		by Lemma \ref{gamma123 n} (i) the left-hand side of \eqref{case1 k>1} will be strictly positive if $\tau$ satisfies:
		$$\left(2\left(\norm{H}\norm{T_k}\norm{M}q_2\right)^2\frac{|\sin\theta|^2}{|\lambda-1|}+\sqrt{2}\norm{X_k}\norm{M}^2(p+q_1)^2\right)\tau<1.$$
		Since $\frac{|\sin\theta|^2}{|\lambda-1|}\le\frac{|\sin\theta|^2}{2|\sin(\theta/2)|}=2\left| \sin\frac{\theta}{2}\right|\cos^2\frac{\theta}{2}\le 2$, we have the first part of the conclusion using definitions \eqref{eq:pBk}, \eqref{eq:q1Bk}, \eqref{eq:q2Bk}  of $p, q_1, q_2$.
		Finally, the conclusion in the case $\norm{B}<1$ can be obtained by Lemma \ref{lem:boundsSTXk}.
	\end{proof}
	
	\begin{lemma}[Case 2]
		\label{k>1 n case2} For $k\ge 2$, equation \eqref{eq-eigen n} admits no solutions $\lambda$ in Case 2 if we take
		\begin{itemize}
			\item $\tau<\cfrac{s(B^k)^{-4}}{\left(\frac{1}{2\sin(\theta_0/2)}\norm{H}^2\norm{M}^2\norm{T_k}^2+\sqrt{2}\norm{M}^2\norm{X_k}\right)(1+2\norm{B^k})^2}$ if the denominator of the right-hand side is not $0$;
			\item any $\tau$ otherwise.
		\end{itemize}
		Moreover, if $\norm{B}<1$, we can take
		$$\tau<\frac{(1-\norm{B})^2}{\norm{H}^2\norm{M}^2}\cdot\frac{(1-\norm{B^k})^2}{\left[ \frac{1}{2\sin(\theta_0/2)}(1-\norm{B}^k)^2+\sqrt{2}(1-k\norm{B}^{k-1}+(k-1)\norm{B}^k)\right](1+\norm{B}^k)^2}.$$
	\end{lemma}
	\begin{proof}
		Writing \eqref{im-gamma-re k>1} for $\gamma=\gamma_2$ as in Lemma \ref{gamma123 n} (ii) (in particular $\gamma_2^2=1$), we have
		\begin{equation}\label{case2 k>1}
			\begin{array}{ll}
				[\Re(\lambda^3-\lambda^2)+\gamma_2\Im(\lambda^3-\lambda^2)]\norm{y}^2+\tau G(P^*+\gamma_2 Q^*,P+\gamma_2 Q)-2\tau G(Q^*,Q)&\\
				+\tau\left([\Re(\lambda-1)+\gamma_2\Im(\lambda-1)]L_1+[\gamma_2\Re(\lambda-1)-\Im(\lambda-1)]L_2\right)  &=0.
			\end{array}
		\end{equation}
		Since $G(Q^*,Q)\ge 0$, and by estimating $\norm{P+\gamma_2 Q}\le \norm{P}+|\gamma_2|\norm{Q}=\norm{P}+\norm{Q}\le p+q_1$, so that 
		\[
		G(P^*+\gamma_2 Q^*,P+\gamma_2 Q)\le[\norm{H}\norm{T_k}\norm{M}(p+q_1)]^2 \norm{y}^2, 
		\]
		and 
		\begin{multline*}
			[\Re(\lambda-1)+\gamma_2\Im(\lambda-1)]L_1+[\gamma_2\Re(\lambda-1)-\Im(\lambda-1)]L_2\\
			\le\sqrt{2}|\lambda-1|\norm{X_k}\norm{M}^2(p+q_1)^2\norm{y}^2,
		\end{multline*}
		by Lemma \ref{gamma123 n} (ii), the left-hand side of \eqref{case2 k>1} will be strictly negative if $\tau$ satisfies:
		$$\left(\left[\norm{H}\norm{T_k}\norm{M}(p+q_1)\right]^2\frac{1}{|\lambda-1|}+\sqrt{2}\norm{X_k}\norm{M}^2(p+q_1)^2\right)\tau<1.$$
		Since $\frac{1}{|\lambda-1|}\le\frac{1}{2\sin(\theta_0/2)}$, we have the first part of the conclusion using definitions \eqref{eq:pBk}, \eqref{eq:q1Bk} of $p, q_1$.
		Finally, the conclusion in the case $\norm{B}<1$ can be obtained by Lemma \ref{lem:boundsSTXk}.
	\end{proof}	
	
	\begin{lemma}[Case 3]
		\label{k>1 n case3} Let $\delta_0>0$ be fixed and $c \coloneqq \frac{1+2\delta_0\sin\frac{5\theta_0}{2}+\delta_0^2}{\cos^2\frac{5\theta_0}{2}}$. For $k\ge 2$, equation \eqref{eq-eigen n} admits no solutions $\lambda$ in Case 3 if we take
		\begin{itemize}
			\item $
			\tau<s(B^k)^{-4}\bigg/\left[\frac{2c\sin\frac{\theta_0}{2} }{\delta_0}\norm{H}^2\norm{M}^2\norm{T_k}^2\norm{B^k}^2+\frac{\sqrt{c}}{\delta_0}\norm{M}^2\norm{X_k}(1+2\norm{B^k}+2\norm{B^k}^2)\right.$
			
			$\left.+2\max\left(\frac{\sqrt{c}}{\delta_0},\frac{\sqrt{c}}{\cos3\theta_0}\right)\norm{M}^2\norm{X_k}(\norm{B^k}+\norm{B^k}^2) \right]
			$ if the denominator of the right-hand side is not $0$;
			\item  any $\tau>0$ otherwise.
		\end{itemize}
		Moreover, if $\norm{B}<1$, we can take
		$$\begin{array}{ll}
			\tau<&\frac{(1-\norm{B})^2}{\norm{H}^2\norm{M}^2}(1-\norm{B}^k)^2\left[\frac{2c\sin\frac{\theta_0}{2}}{\delta_0}\norm{B^k}^{2k}+\frac{\sqrt{c}}{\delta_0}(1-k\norm{B}^{k-1}+(k-1)\norm{B}^k)(1+\norm{B}^{2k})\right.\\
			&\left.
			+2\max\left(\frac{\sqrt{c}}{\delta_0},\frac{\sqrt{c}}{\cos3\theta_0}\right)(1-k\norm{B}^{k-1}+(k-1)\norm{B}^k)\norm{B}^k\right]^{-1}.\\
		\end{array}
		$$
	\end{lemma}
	\begin{proof}
		Writing \eqref{im-gamma-re k>1} for $\gamma=\gamma_3$ as in Lemma \ref{gamma123 n} (iii), we have
		\begin{equation}\label{case3 k>1}
			\begin{array}{ll}
				[\Re(\lambda^3-\lambda^2)+\gamma_3\Im(\lambda^3-\lambda^2)]\norm{y}^2+\tau G(P^*+\gamma_3 Q^*,P+\gamma_3 Q)-(1+\gamma_3^2)\tau G(Q^*,Q)&\\
				+\tau\left([\Re(\lambda-1)+\gamma_3\Im(\lambda-1)]L_1+[\gamma_3\Re(\lambda-1)-\Im(\lambda-1)]L_2\right)  &=0.
			\end{array}
		\end{equation}
		Since $G(P^*+\gamma_3 Q^*,P+\gamma_3 Q)\ge 0$, the left-hand side of \eqref{case3 k>1} will be strictly positive if $\tau$ satisfies:
		$$\begin{array}{ll}
			\tau<&\cfrac{1}{\norm{y}^2}\left[(1+\gamma_3^2)\cfrac{G(Q^*,Q)}{\Re(\lambda^3-\lambda^2)+\gamma_3\Im(\lambda^3-\lambda^2)}\right.\\
			&\left.+|L_1|\cfrac{|\Re(\lambda-1)+\gamma_3\Im(\lambda-1)|}{\Re(\lambda^3-\lambda^2)+\gamma_3\Im(\lambda^3-\lambda^2)}
			+|L_2|\cfrac{|\gamma_3\Re(\lambda-1)-\Im(\lambda-1)|}{\Re(\lambda^3-\lambda^2)+\gamma_3\Im(\lambda^3-\lambda^2)}\right]^{-1}.
		\end{array}
		$$
		By estimating
		\begin{itemize}
			\item $G(Q^*,Q)\le(\norm{H}\norm{T_k}\norm{M}q_2|\sin\theta|)^2 \norm{y}^2$
			\item $|L_1|\le \norm{X_k}\norm{M}^2(p^2+q_1^2)\norm{y}^2$;
			\item $|L_2|\le  2\norm{X_k}\norm{M}^2pq_1\norm{y}^2$
		\end{itemize}
		and using Lemma \ref{gamma123 n} (iii), it suffices to choose
		$$\begin{array}{ll}
			\left[(1+\gamma_3^2)\left(\norm{H}\norm{T_k}\norm{M}q_2\right)^2\frac{2|\sin\frac{\theta}{2}|\cos^2\frac{\theta}{2}}{\delta_0}+\norm{X_k}\norm{M}^2(p^2+q_1^2)\frac{\sqrt{1+\gamma_3^2}}{\delta_0}\right.&\\
			\left.
			+2\norm{X_k}\norm{M}^2pq_1\max\left(\frac{\sqrt{1+\gamma_3^2}}{\delta_0},\frac{\sqrt{1+\gamma_3^2}}{\cos3\theta_0}\right)\right]\tau&<1.
		\end{array}
		$$	
		Noting that $c=1+\gamma_3^2$, the final result is obtained by definitions \eqref{eq:pBk}, \eqref{eq:q1Bk}, \eqref{eq:q2Bk} of $p,q_1,q_2$. 
		Finally, the conclusion in the case $0<\norm{B}<1$ can be obtained by Lemma \ref{lem:boundsSTXk}.
	\end{proof}
	
	\begin{lemma}[Case 4]
		\label{k>1 n case4} For $k\ge 2$, equation \eqref{eq-eigen n} admits no solutions $\lambda$ in Case 4 if we take
		\begin{itemize}
			\item $\tau<\cfrac{\left[\sin\left(\frac{\pi}{2}-3\theta_0\right)+\cos2\theta_0\right]s(B^k)^{-4}}{\norm{H}^2\norm{M}^2\norm{T_k}^2(1+\norm{B^k})^2+2\norm{M}^2\norm{X_k}(1+2\norm{B^k})^2}$ if the denominator of the right-hand side is not $0$;
			\item any $\tau>0$ otherwise.
		\end{itemize}
		Moreover, if $\norm{B}<1$, we can take
		$$\tau<\frac{(1-\norm{B})^2}{\norm{H}^2\norm{M}^2}\cdot\frac{\left[\sin\left(\frac{\pi}{2}-3\theta_0\right)+\cos2\theta_0\right](1-\norm{B}^k)^2}{ (1-\norm{B}^k)^2+2(1-k\norm{B}^{k-1}+(k-1)\norm{B}^k)(1+\norm{B}^k)^2}.$$
	\end{lemma}
	\begin{proof}
		Here it is enough to consider \eqref{re-eq k>1}. By the properties of $G$
		$$G(Q^*,Q)\ge 0,\quad G(P^*,P)\le(\norm{H}\norm{T_k}\norm{M}p)^2\norm{y}^2$$
		and Lemma \ref{gamma123 n} (iv), we see that	the left-hand side of \eqref{re-eq k>1} will be strictly negative if $\tau$ satisfies:
		$$\begin{array}{ll}
			\left[\left(\norm{H}\norm{T_k}\norm{M}p\right)^2\frac{1}{\sin\left(\frac{\pi}{2}-3\theta_0\right)+\cos 2\theta_0}
			+ \norm{X_k}\norm{M}^2(p+q_1)^2\frac{2}{\sin\left(\frac{\pi}{2}-3\theta_0\right)+\cos2\theta_0}\right]\tau&<1.
		\end{array}	
		$$		
		Definitions \eqref{eq:pBk}, \eqref{eq:q1Bk} of $p, q_1$ lead to the final result. 
		Finally, the conclusion in the case $0<\norm{B}<1$ can be obtained by Lemma \ref{lem:boundsSTXk}.
	\end{proof}
	
	Similarly, with the help of Lemma \ref{gamma123 n+1}, we prove for the $k$-step one-shot method the analogue of Proposition~\ref{tau n k>1 cpl}. In particular, note that here just three cases of $\lambda$ need to be considered, because the analogue of the fourth one is excluded by Lemma \ref{gamma123 n+1} (iv).
	
	\begin{proposition}[$k$-step one-shot method]\label{tau n+1 k>1 cpl}
		$\exists\tau>0$ sufficiently small such that equation \eqref{eq-eigen n+1} admits no solution $\lambda\in\C\backslash\R$, $|\lambda|\ge 1$.  
		In particular, if $\norm{B}<1$, given any $\delta_0>0$ and $0<\theta_0<\frac{\pi}{4}$, take
		\[
		\tau < \frac{\min\{\psi_1(k,b), \; \psi_2(k,b), \; \psi_3(k,b)\}}{\norm{H}^2\norm{M}^2}
		\]
		where
		\[
		\psi_1(k,b)=\frac{(1-b)^2(1-b^k)^2}{4b^{2k}+\sqrt{2}(1-kb^{k-1}+(k-1)b^k)(1+b^k)^2}
		\]
		\[
		\psi_2(k,b)=\frac{(1-b)^2(1-b^k)^2}{\Bigl[ \frac{1}{2\sin(\theta_0/2)}(1-b^k)^2+\sqrt{2}(1-kb^{k-1}+(k-1)b^k)\Bigr](1+b^k)^2}
		\]
		\[
		\psi_3(k,b)= \frac{(1-b)^2(1-b^k)^2}{ \frac{2c\sin(\theta_0/2)}{\delta_0}b^{2k} + (1-kb^{k-1}+(k-1)b^k) \Bigl[ \frac{\sqrt{c}}{\delta_0}(1+b^{2k}) +2\max\Bigl(\frac{\sqrt{c}}{\delta_0},\frac{\sqrt{c}}{\cos2\theta_0}\Bigr)b^k \Bigr] }
		\]
		and $c=\frac{1+2\delta_0\sin\frac{3\theta_0}{2}+\delta_0^2}{\cos^2\frac{3\theta_0}{2}}$. 
	\end{proposition}	
	
	\subsection{Final result ($k\ge 2$)}
	Considering Remark~\ref{rmk:B=0,k>1}, and taking the minimum between the bound in Proposition~\ref{tau n k>1 real} for real eigenvalues and the bound in Proposition~\ref{tau n k>1 cpl} for complex eigenvalues, we finally obtain a sufficient condition on the descent step $\tau$ to ensure convergence of the shifted multi-step one-shot method. 
	\begin{theorem}
		[Convergence of shifted $k$-step one-shot, $k\ge 2$]
		\label{th:tau n k>1 all}
		Under assumption \eqref{hypo}, the shifted $k$-step one-shot method, $k\ge 2$, converges for sufficiently small $\tau$.  In particular, for $\norm{B}<1$, it is enough to take
		$$\tau<\frac{\chi(k,\norm{B})}{\norm{H}^2\norm{M}^2},$$ 
		where $\chi(k,\norm{B})$ is an explicit function of $k$ and $\norm{B}$. 
	\end{theorem}	
	
	Similarly, by combining Remark \ref{rmk:B=0,k>1}, Propositions \ref{tau n+1 k>1 real} and \ref{tau n+1 k>1 cpl}, we obtain a sufficient condition on the descent step $\tau$ to ensure convergence of the multi-step one-shot method. 		
	\begin{theorem}[Convergence of $k$-step one-shot, $k\ge 2$]
		\label{th:tau n+1 k>1 all}
		Under assumption \eqref{hypo}, the $k$-step one-shot method, $k\ge 2$, converges for sufficiently small $\tau$. In particular, for $\norm{B}<1$, it is enough to take
		$$\tau<\frac{\psi(k,\norm{B})}{\norm{H}^2\norm{M}^2},$$
		where $\psi(k,\norm{B})$ is an explicit function of $k$ and $\norm{B}$. 
	\end{theorem}

	\section{Inverse problem with complex forward problem and real parameter}
\label{sec:complex_extension}

In this section we show that a linear inverse problem with associated complex forward problem and real parameter can be transformed into a linear inverse problem which matches with the real model at the beginning of Section~\ref{sec:intro-k-shot}, so that the previous theory applies.   
More precisely, here we study the state equation
$$u=Bu+M\sigma+F$$
where $u\in\C^{n_u}$, $\sigma\in\R^{n_\sigma}$, $B\in\C^{n_u\times n_u}, M\in\C^{n_u\times n_\sigma}$. We measure $Hu(\sigma)=f$ where $H\in\C^{n_f\times n_u}$ and we want to recover $\sigma$ from $f$. Using the method of least squares, we consider the cost functional
$$J(\sigma) \coloneqq \frac{1}{2}\norm{Hu(\sigma)-f}^2,$$
then by the Lagrangian technique with
$$\mathcal{L}(u,v,\sigma)=\frac{1}{2}\norm{Hu-f}^2+\Re\scalar{Bu+m\sigma+F-u,v},$$
we can define the adjoint state $p=p(\sigma)$ such that
$$p=B^*p+H^*(Hu(\sigma)-f),$$
which allows us to compute
$$\nabla J(\sigma)=\Re(M^*p).$$
By separating the real and imaginary parts of all vectors and matrices
$u=u_1+\ic u_2$, $p=p_1+\ic p_2$, $B=B_1+\ic B_2$, $M=M_1+\ic M_2$, $F=F_1+\ic F_2$, $H=H_1+\ic H_2$, $f=f_1+\ic f_2$, we can transform this inverse problem with complex forward problem into the inverse problem with real forward problem introduced at the beginning of Section \ref{sec:intro-k-shot}. Indeed, note that $B^*=B_1^*-\ic B_2^*$, $M^*=M_1^*-\ic M_2^*$, $H^*=H_1^*-\ic H_2^*$, so we have
$$\begin{cases}
	u_1+\ic u_2=(B_1+\ic B_2)(u_1+\ic u_2)+(M_1+\ic M_2)\sigma+(F_1+\ic F_2)\\
	p_1+\ic p_2=(B_1^*-\ic B_2^*)(p_1+\ic p_2)+ (H_1^*-\ic H_2^*)[(H_1+\ic H_2)(u_1+\ic u_2)-(f_1+\ic f_2)]\\
	\nabla J(\sigma)=\Re[(M_1^*-\ic M_2^*)(p_1+\ic p_2)],
\end{cases}$$
which implies
$$\begin{cases}
	u_1=B_1u_1-B_2u_2+M_1\sigma+F_1 \\
	u_2=B_2u_1+B_1u_2+M_2\sigma+F_2 \\
	p_1=B_1^*p_1+B_2^*p_2+(H_1^*H_1+H_2^*H_2)u_1-(H_2^*H_1-H_1^*H_2)u_2-(H_1^*f_1+H_2^*f_2) \\
	p_2=-B_2^*p_1+B_1^*p_2+(H_2^*H_1-H_1^*H_2)u_1+(H_1^*H_1+H_2^*H_2)u_2-(-H_2^*f_1+H_1^*f_2) \\
	\nabla J(\sigma)=M_1^*p_1+M_2^*p_2.
\end{cases}$$
By setting
$$
\tilde{u}=\begin{bmatrix}
	u_1 \\ u_2
\end{bmatrix},
\tilde{p}=\begin{bmatrix}
	p_1 \\ p_2
\end{bmatrix},
\tilde{B}=\begin{bmatrix}
	B_1 & -B_2 \\ B_2 & B_1
\end{bmatrix}, 
\tilde{M}=\begin{bmatrix}
	M_1 \\ M_2 
\end{bmatrix}, 
\tilde{F}=\begin{bmatrix}
	F_1 \\ F_2 
\end{bmatrix}, 
\tilde{H}=\begin{bmatrix}
	H_1 & -H_2\\ H_2 & H_1
\end{bmatrix},
\tilde{f}=\begin{bmatrix}
	f_1 \\ f_2
\end{bmatrix}
$$
we have 
$$\begin{cases}
	\tilde{u}=\tilde{B}\tilde{u}+\tilde{M}\sigma+\tilde{F}\\
	\tilde{p}=\tilde{B}^*\tilde{p}+\tilde{H}^*(\tilde{H}\tilde{u}-\tilde{f})\\
	\nabla J(\sigma)=\tilde{M}^*\tilde{p},
\end{cases}$$
that has the same structure as the inverse problem at the beginning of Section \ref{sec:intro-k-shot}.

Finally we finish this section by two lemmas that match the assumptions of the inverse problem with complex state variable with the assumptions of the transformed inverse problem with real state variable.

\begin{lemma}
	$\mathrm{Spec}(\tilde{B})=\mathrm{Spec}(B)\cup\overline{\mathrm{Spec(B)}}$.
\end{lemma}

\begin{proof}
	By writing 
	\begin{equation}\label{eq:B1B2}
		\tilde{B}=\begin{bmatrix}
			B_1 & -B_2 \\
			B_2 & B_1
		\end{bmatrix}
		=\underbrace{
			\begin{bmatrix}
				I & I \\
				\ic I & -\ic I
		\end{bmatrix}}_{C^{-1}}
		\begin{bmatrix}
			\overline{B} & 0 \\
			0 & B 
		\end{bmatrix}
		\underbrace{
			\begin{bmatrix}
				\frac{1}{2}I & -\frac{\ic}{2}I \\
				\frac{1}{2}I & \frac{i}{2}I
		\end{bmatrix}}_{C},
	\end{equation}
	we find that $\det(\tilde{B}-\lambda I)=\det(\overline{B}-\lambda I)\det(B-\lambda I)$. The conclusion is then deduced thanks to the fact that $\mathrm{Spec}(\overline{B})=\overline{\mathrm{Spec(B)}}$.
\end{proof}

\begin{lemma}
	Assume that $\rho(B)<1$, and $H(I-B)^{-1}M$ is injective. Then $\rho(\tilde{B})<1$, and $\tilde{H}(\tilde{I}-\tilde{B})^{-1}\tilde{M}$ is injective where $\tilde{I}\in\R^{2n_u\times 2n_u}$ is the identity matrix.
\end{lemma}
\begin{proof}
	The previous lemma says that $\rho(\tilde{B})=\rho(B)<1$. Therefore $(\tilde{I}-\tilde{B})^{-1}$ is well-defined and thanks to \eqref{eq:B1B2},
	$$\begin{array}{ll}
		(\tilde{I}-\tilde{B})^{-1}&=\underbrace{
			\begin{bmatrix}
				I & I \\
				\ic I & -\ic I
		\end{bmatrix}}_{C^{-1}}\begin{bmatrix}
			(I-\overline{B})^{-1} & 0 \\
			0 & (I-B)^{-1} 
		\end{bmatrix}
		\underbrace{\begin{bmatrix}
				\frac{1}{2}I & -\frac{\ic}{2}I \\
				\frac{1}{2}I & \frac{i}{2}I
		\end{bmatrix}}_{C} \\
		&=\frac{1}{2}\begin{bmatrix}
			(I-\overline{B})^{-1}+(I-B)^{-1}  & -\ic(I-\overline{B})^{-1}+\ic (I-B)^{-1}  \\
			\ic(I-\overline{B})^{-1}-\ic(I-B)^{-1}  & (I-\overline{B})^{-1}+(I-B)^{-1} 
		\end{bmatrix}.
	\end{array}$$
	Now we have
	$$\begin{array}{ll}
		\tilde{H}(\tilde{I}-\tilde{B})^{-1}\tilde{M}
		&=\frac{1}{2}\begin{bmatrix}
			H_1 & -H_2\\ H_2 & H_1
		\end{bmatrix}
		\begin{bmatrix}
			(I-\overline{B})^{-1}+(I-B)^{-1}  & -\ic(I-\overline{B})^{-1}+\ic (I-B)^{-1}  \\
			\ic(I-\overline{B})^{-1}-\ic(I-B)^{-1}  & (I-\overline{B})^{-1}+(I-B)^{-1} 
		\end{bmatrix}
		\begin{bmatrix}
			M_1 \\ M_2
		\end{bmatrix} \\
		&=\frac{1}{2}\begin{bmatrix}
			\overline{H}(I-\overline{B})^{-1}+H(I-B)^{-1}  & -\ic\overline{H}(I-\overline{B})^{-1}+\ic H(I-B)^{-1}  \\
			\ic\overline{H}(I-\overline{B})^{-1}-\ic H(I-B)^{-1}  & \overline{H}(I-\overline{B})^{-1}+H(I-B)^{-1} 
		\end{bmatrix} 
		\begin{bmatrix}
			M_1 \\ M_2
		\end{bmatrix} \\
		&=\frac{1}{2}\begin{bmatrix}
			\overline{H}(I-\overline{B})^{-1}\overline{M}+H(I-B)^{-1}M \\
			\ic \overline{H}(I-\overline{B})^{-1}\overline{M}-\ic H(I-B)^{-1}M
		\end{bmatrix}.
	\end{array} 
	$$ 
	Now assume that there exists $x\in\C^{n_\sigma}$ such that $H(\tilde{I}-\tilde{B})^{-1}\tilde{M}x=0$, then 
	$$
	\begin{cases}
		[\overline{H}(I-\overline{B})^{-1}\overline{M}+H(I-B)^{-1}M]x=0\\
		[\ic \overline{H}(I-\overline{B})^{-1}\overline{M}-\ic H(I-B)^{-1}M]x=0
	\end{cases}
	$$ 
	or, equivalently,
	$$
	\begin{cases}
		[H(I-\overline{B})^{-1}\overline{M}+H(I-B)^{-1}M]x=0\\
		[-\overline{H}(I-\overline{B})^{-1}\overline{M}+H(I-B)^{-1}M]x=0.
	\end{cases}
	$$ 
	By summing up these two equations we deduce that $H(I-B)^{-1}Mx=0$, then $x=0$ thanks to the injectivity of $H(I-B)^{-1}M$.
\end{proof}

	\section{Numerical experiments}
\label{sec:num-exp}

Let us introduce a toy model to illustrate numerically the performance of the different methods. Given $\Omega\subset\R^n$ an open bounded Lipschitz domain, we consider the direct problem for the linearized scattered field $u \in\Hso^2(\Omega)$ given by the Helmholtz equation
\begin{equation}
	\label{u-toymodel}
	\left\{
	\begin{array}{ll}
		\dive(\sigmatz\nabla u)+\tilde{k}^2u=\dive(\sigma\nabla u_0), &\text{in }\Omega, \\
		u=0, &\text{on }\pa\Omega,
	\end{array}
	\right.
\end{equation}
where the incident field $u_0:\Omega\to\R$ satisfies
\begin{equation}
	\label{u0-toymodel}
	\left\{
	\begin{array}{ll}
		\dive(\sigmatz\nabla u_0)+\tilde{k}^2u=0, &\text{in }\Omega, \\
		u_0=f, &\text{on }\pa \Omega
	\end{array}
	\right.
\end{equation}
with the datum $f:\partial\Omega\to\R$. Here $\sigma:\overline{\Omega}\to\R$ such that $\sigma\big|_{\pa\Omega}=0$; $\sigmatz=\sigma_0+\delta\sigmar$ is a given function with $\delta\ge 0$ and random $\sigmar$. More precisely, given $\sigmatz$ and $f$, we solve for $u_0=u_0(f)$ in \eqref{u0-toymodel}, then insert $u_0$ into \eqref{u-toymodel} to solve for $u=u(\sigma)$. The variational formulations for $u$ and $u_0$ are respectively
\begin{equation}
	\label{vf-u}
	\int_{\Omega}\sigmatz\nabla u\cdot\nabla v-\int_\Omega \tilde{k}^2uv=\int_{\Omega}\sigma\nabla u_0\cdot\nabla v, \quad\forall v\in\Hso^1_0(\Omega)\quad\text{and }u=0 \text{ on }\pa\Omega,
\end{equation}
\begin{equation}
	\label{vf-u0}
	\int_{\Omega}\sigmatz\nabla u_{0}\cdot\nabla v-\int_\Omega \tilde{k}^2uv=0, \quad\forall v\in\Hso^1_0(\Omega)\quad\text{and }u_0=f\text{ on }\pa\Omega.
\end{equation}
We are interested in the inverse problem of finding $\sigma$ from the measurement $Hu(\sigma)$ where $Hu \coloneqq \sigmatz\frac{\pa u}{\pa\nu}\big|_{\pa\Omega}$. To solve this inverse problem we use the method of least squares. Denoting by $\sigma^\text{ex}$ the exact $\sigma$ and $g=\sigmatz\frac{\pa u(\sigma^\text{ex})}{\pa\nu}\big|_{\pa\Omega}$ the corresponding measurement, we consider the cost functional $J(\sigma)=\frac{1}{2}\norm{Hu(\sigma)-g}^2_{\mathcal{L}^2(\pa\Omega)}=\frac{1}{2}\int_{\pa\Omega}(\sigmatz\frac{\pa u(\sigma)}{\pa\nu}-g)^2$. 
The Lagrangian technique allows us to compute the gradient  $\nabla_\sigma J(\sigma)=-\nabla u_0\cdot\nabla p(\sigma)$, where the adjoint state $p=p(\sigma)$ satisfies
\begin{equation}
	\label{vf-p}
	\int_{\Omega}\sigmatz\nabla p\cdot\nabla v-\int_\Omega\tilde{k}^2pv=0, \quad\forall v\in\Hso^1(\Omega)\quad\text{and }p=\left(\sigmatz\frac{\pa u(\sigma)}{\pa\nu}\bigg|_{\pa\Omega}-g\right)\text{ on }\pa\Omega.
\end{equation}

By discretizing $u$ by $\mathbb{P}^1$ finite elements on a mesh $\mathcal{T}_h^u$ of $\Omega$, and $\sigma$ by $\mathbb{P}^0$ finite elements on a coarser mesh $\mathcal{T}_h^\sigma$ of $\Omega$, the discretization of \eqref{vf-u} can be written as the linear system $A_1\vec u=A_2\vec\sigma$, where $\vec u\in\R^{n_u}$, $\vec \sigma\in\R^{n_\sigma}$. More precisely, $A_1$ and $A_2$ are respectively issued from the discretization of $\int_{\Omega}\sigmatz\nabla u\cdot\nabla v-\int_\Omega \tilde{k}^2uv$ and $\int_{\Omega}\sigma\nabla u_0\cdot\nabla v$, where the Dirichlet boundary conditions are imposed by the penalty method. To rewrite the system in the form \eqref{pbdirect}, we consider the naive splitting $A_1=A_{11}+\delta A_{12}$, where $A_{11}$ and $A_{12}$ are respectively issued from the discretization of $\int_{\Omega}\sigma_0\nabla u\cdot\nabla v-\int_\Omega \tilde{k}^2uv$ and $\int_{\Omega}\sigmar\nabla u\cdot\nabla v$. Then we get
$$\vec{u}=A_{11}^{-1}(-\delta A_{12}\vec{u}+A_2\vec{\sigma})\quad\text{ and } \vec{u}=0\text{ on }\pa\Omega$$
and
$$\vec{p}=A_{11}^{-1}\left(-\delta A_{12}\vec{p}\right)\quad\text{ and } \vec{p}=H\vec{u}-\vec{g}\text{ on }\pa\Omega$$ 
where $H\in\R^{n_f\times n_u}$ is the discretization of the above operator $H$ by abuse of notation. Choosing $\delta$ such that $\delta\norm{A_{11}^{-1}A_{12}}_2<1$,  we consider \eqref{direct-inv-prb} with $B=-\delta A_{11}^{-1}A_{12}$, $M=A_{11}^{-1}A_2$, $F=0$. The application of $A_{11}^{-1}$, which has the same size as matrix $A_1$, is done by a direct solver; more practical fixed point iterations will be investigated in the future.  

\begin{figure}[htbp]
	\centering
	\includegraphics[width=0.7\linewidth]{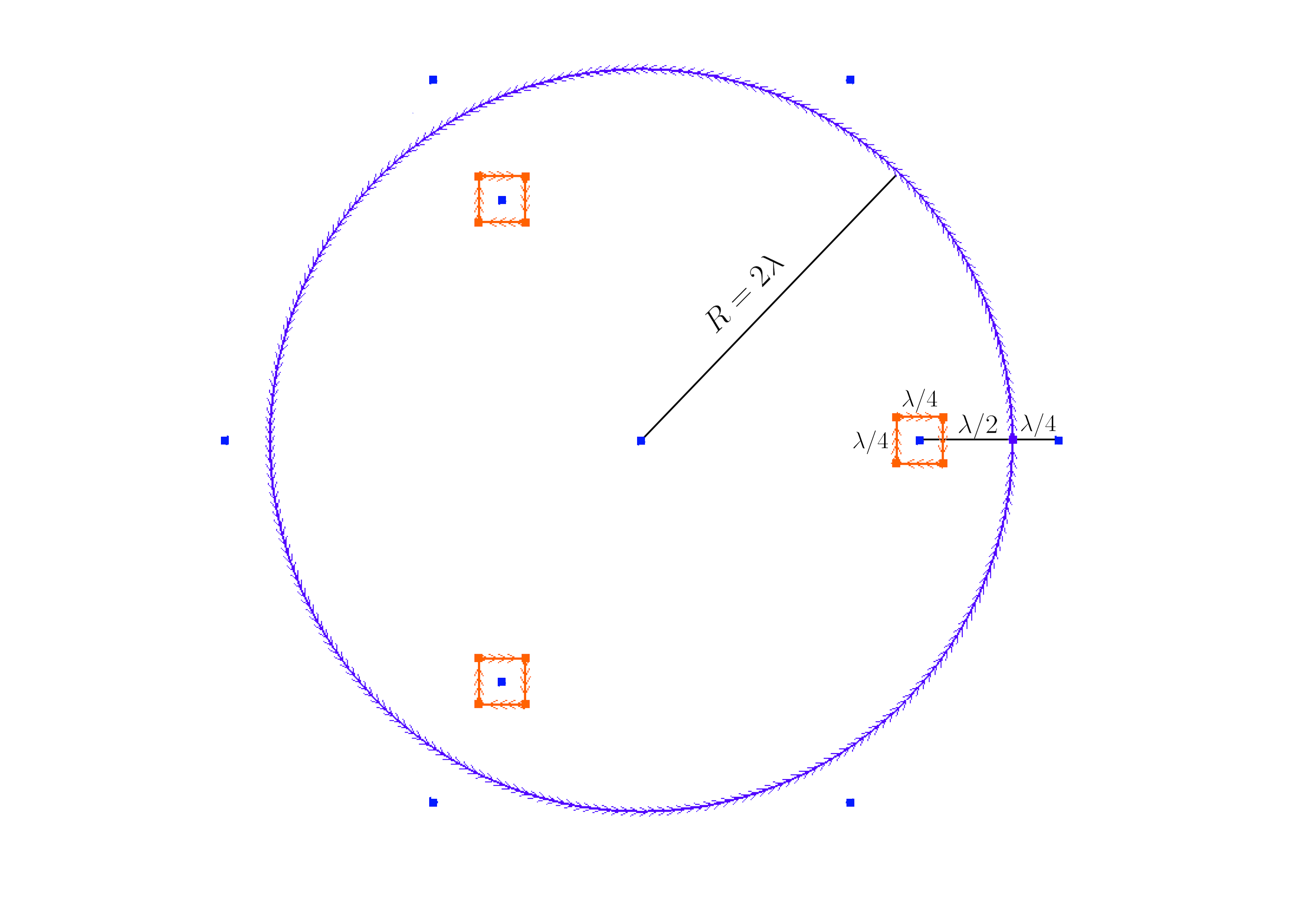}
	\caption{Domain with six source points for the numerical experiments. The unknown $\sigma$ is supported on the three squares.}
	\label{fig:domain}
\end{figure}

We then perform some numerical experiments in FreeFEM \cite{FreeFEM} with the following setting:
\begin{itemize}
	\item Wavenumber $\tilde{k}=2\pi$, $\sigma_0=1$, $ \delta=0.01$, $\sigmar$ is a random real function with range in the interval $[1,2]$.
	\item Wavelength $\lambda=\frac{2\pi}{\tilde{k}}\sqrt{\sigma_0}=1$, mesh size $h=\frac{\lambda}{20}=0.05$. The domain $\Omega$ is the disk shown in Figure \ref{fig:domain}, where the squares are the support of function $\sigma$. Here $n_u=5853$, $n_\sigma=6$. 
	
	\item We test with $6$ data $f$ given by zero-order Bessel function of the second kind centered at the points shown in Figure~\ref{fig:domain}, and the cost functional is the normalized sum of the contributions corresponding to different data.
	
	\item We take $\sigma^\text{ex}=10$ in every square and $0$ otherwise. The initial guess for the inverse problem is $12$ in every square and $0$ otherwise. 
	
	\item For the first iteration, we perform a line search to adapt the descent step $\tau$, using a direct solver for the forward and adjoint problems.  
	
	\item The stopping rule for the outer iteration is based on the relative value of the cost functional and on the relative norm of the gradient with a tolerance of $10^{-5}$.
\end{itemize}

\noindent Recall that $k$ is the number of inner iterations on the direct and adjoint problems. We are interested in two experiments. 

\begin{figure}[htbp]
	\centering
	\begin{subfigure}{0.49\columnwidth}
		\includegraphics[width=\linewidth]{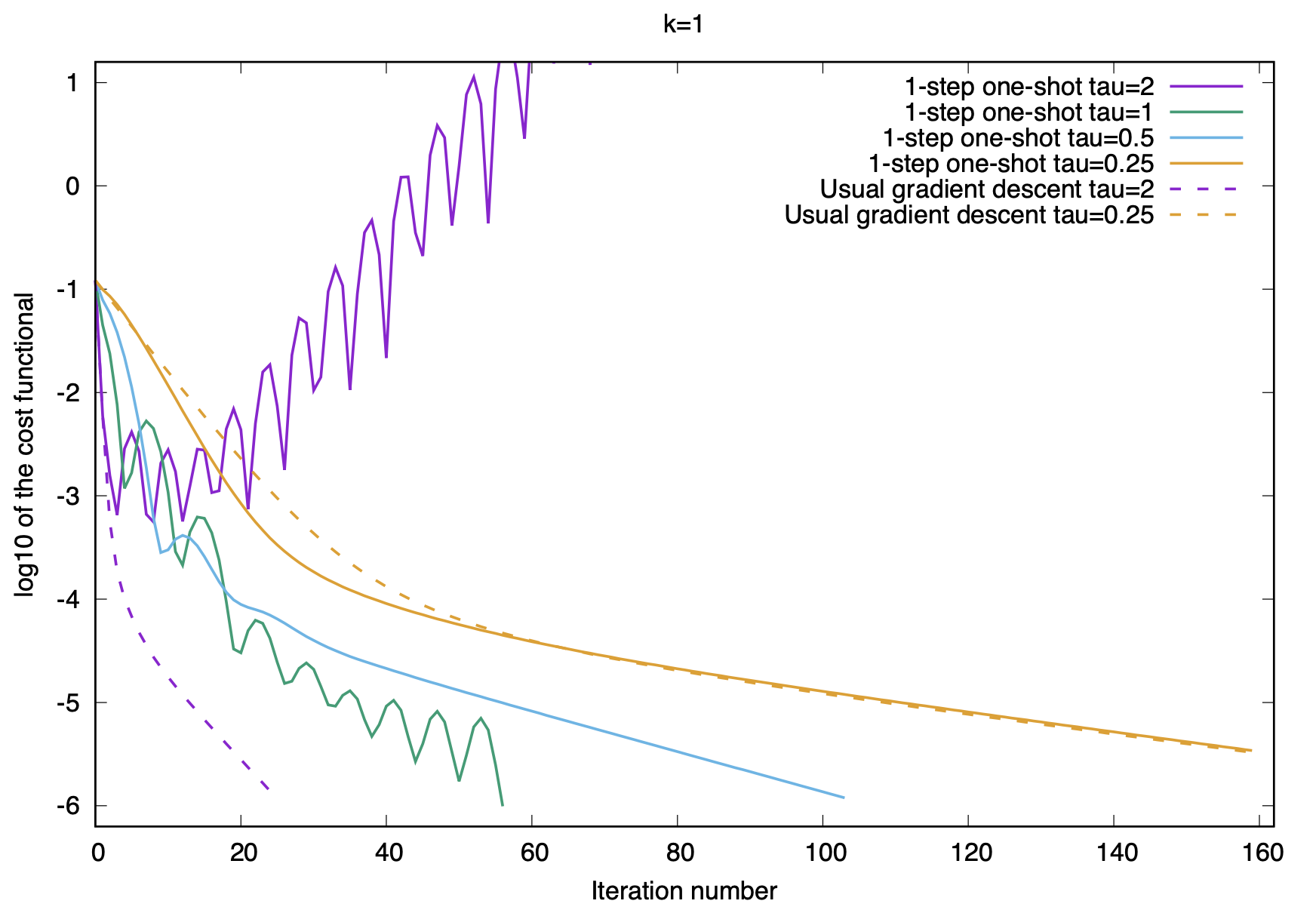}
		\caption{Convergence curves of usual gradient descent and $1$-step one-shot for different descent step $\tau$.}
		\label{fig:usual_fix_k_1}
	\end{subfigure} \hfill
	\begin{subfigure}{0.49\columnwidth}
		\includegraphics[width=\linewidth]{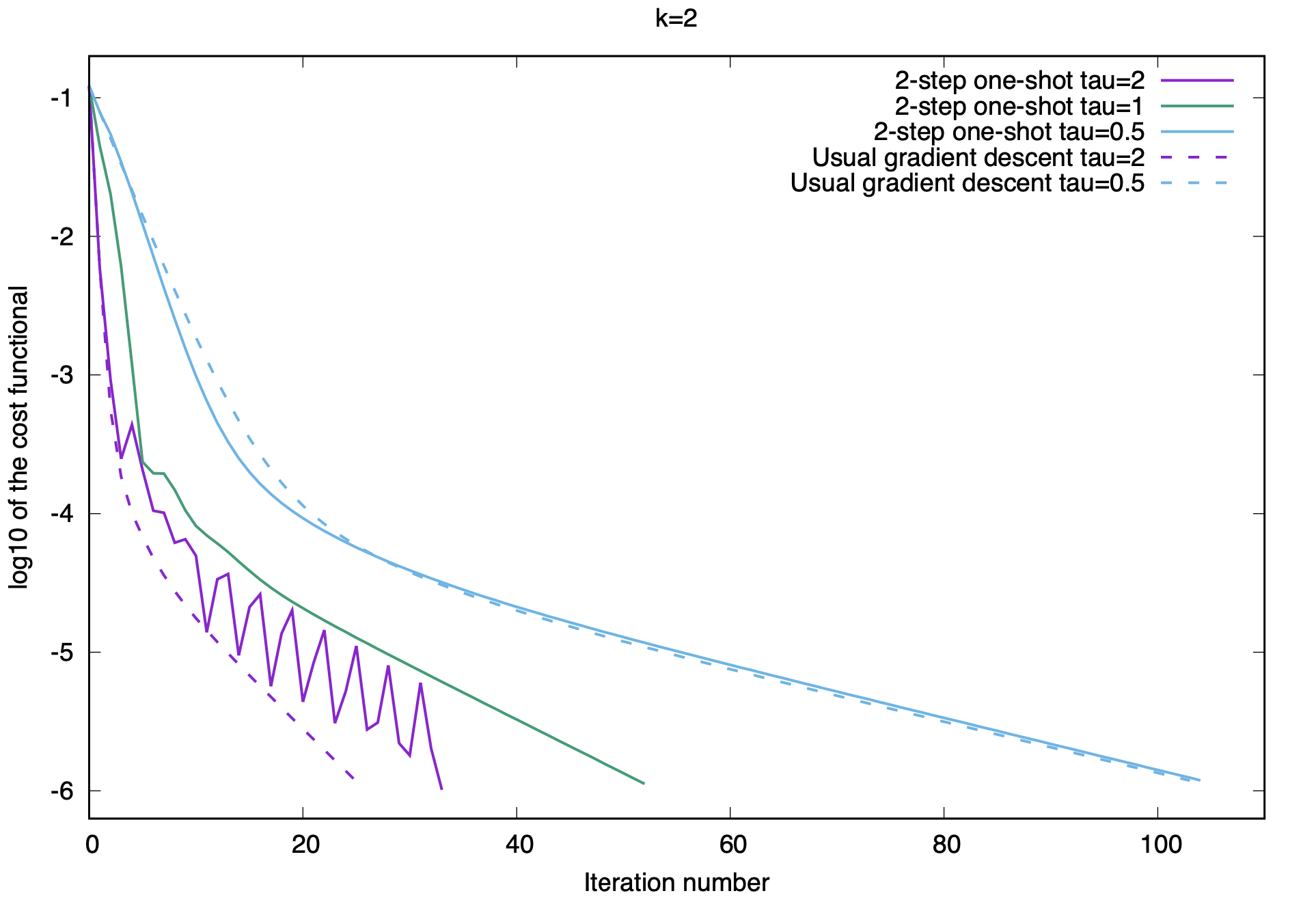}
		\caption{Convergence curves of usual gradient descent and $2$-step one-shot for different descent step $\tau$.}
		\label{fig:usual_fix_k_2}
	\end{subfigure}
	\begin{subfigure}{0.49\columnwidth}
		\includegraphics[width=\linewidth]{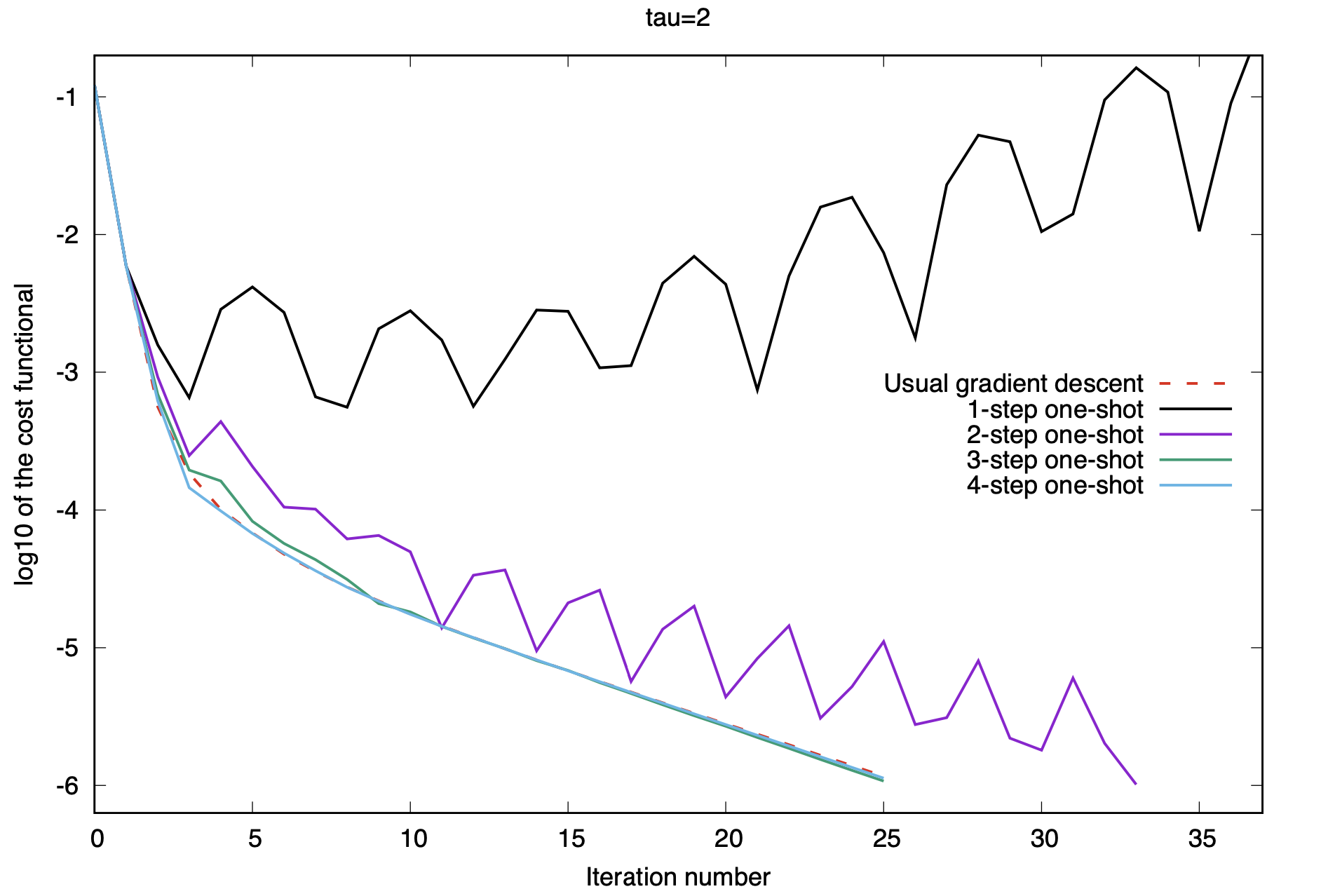}
		\caption{Convergence curves of usual gradient descent and $k$-step one-shot for different $k$ with $\tau=2$.}
		\label{fig:usual_fix_tau_2}
	\end{subfigure} \hfill
	\begin{subfigure}{0.49\columnwidth}
		\includegraphics[width=\linewidth]{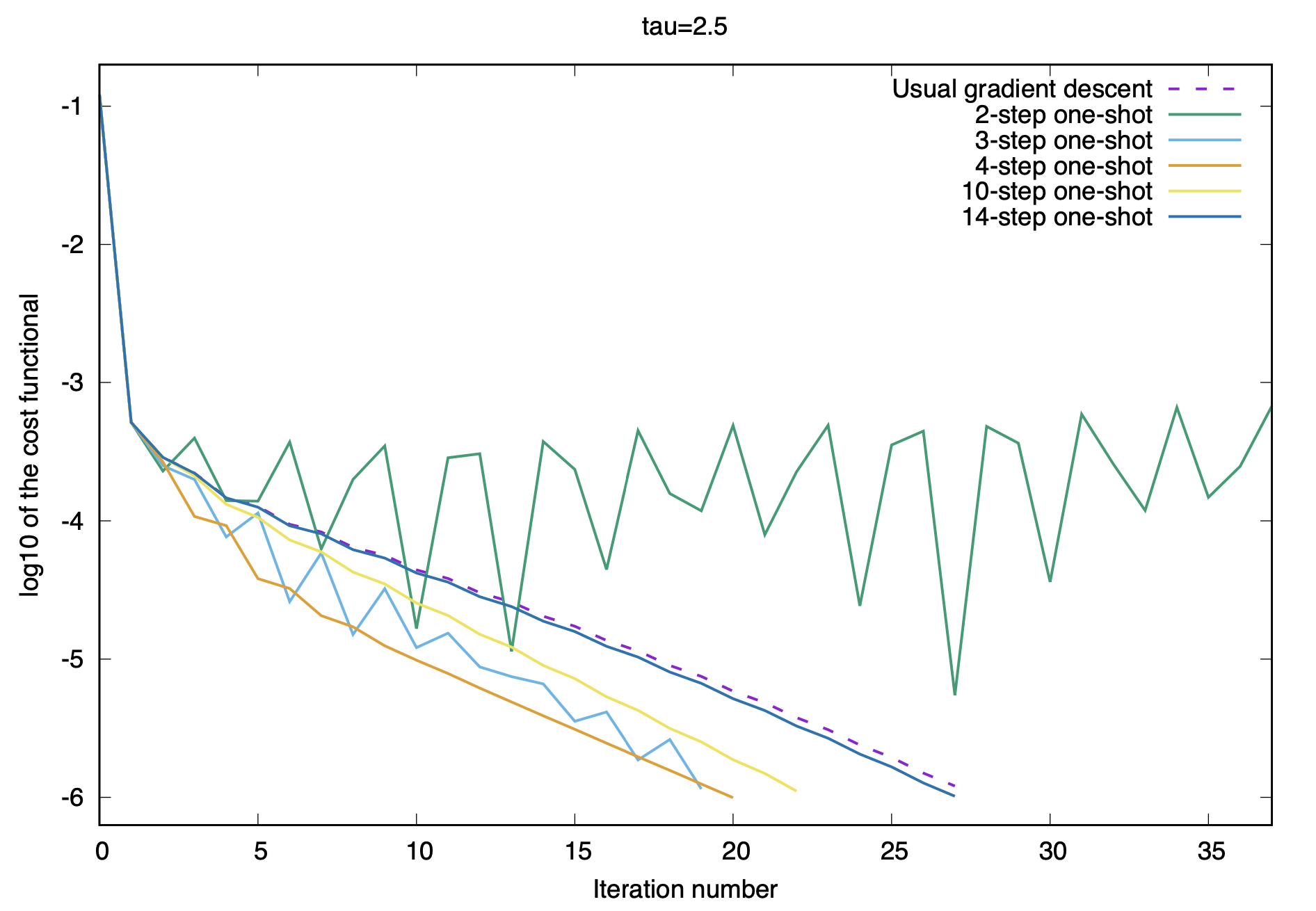}
		\caption{Convergence curves of usual gradient descent and $k$-step one-shot for different $k$ with $\tau=2.5$.}
		\label{fig:usual_fix_tau_2.5}
	\end{subfigure}
	\begin{subfigure}{0.49\columnwidth}
		\includegraphics[width=\linewidth]{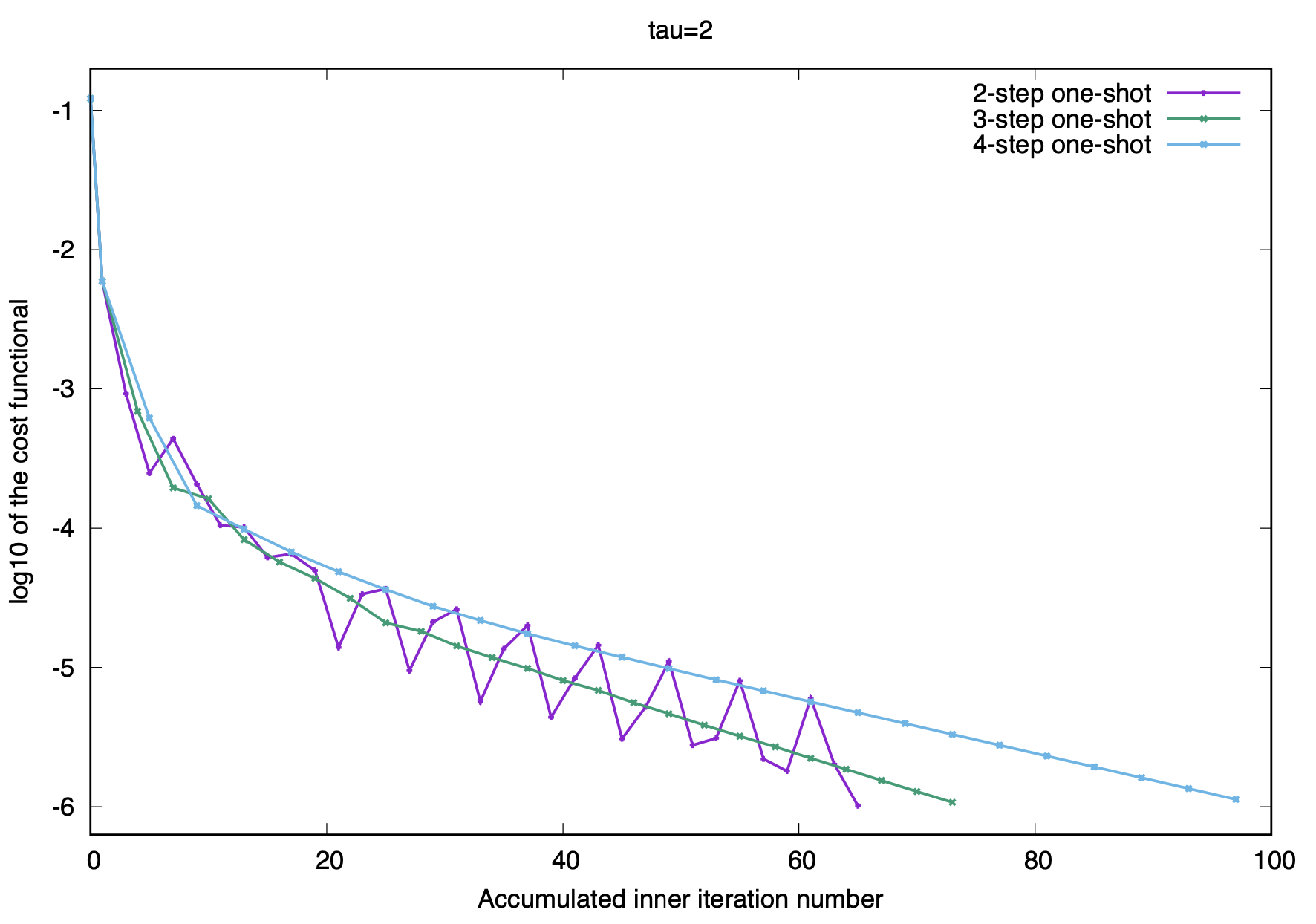}
		\caption{Convergence curves of $k$-step one-shot for different $k$ with $\tau=2$.}
		\label{fig:v2_usual_fix_tau_2}
	\end{subfigure} \hfill
	\begin{subfigure}{0.49\columnwidth}
		\includegraphics[width=\linewidth]{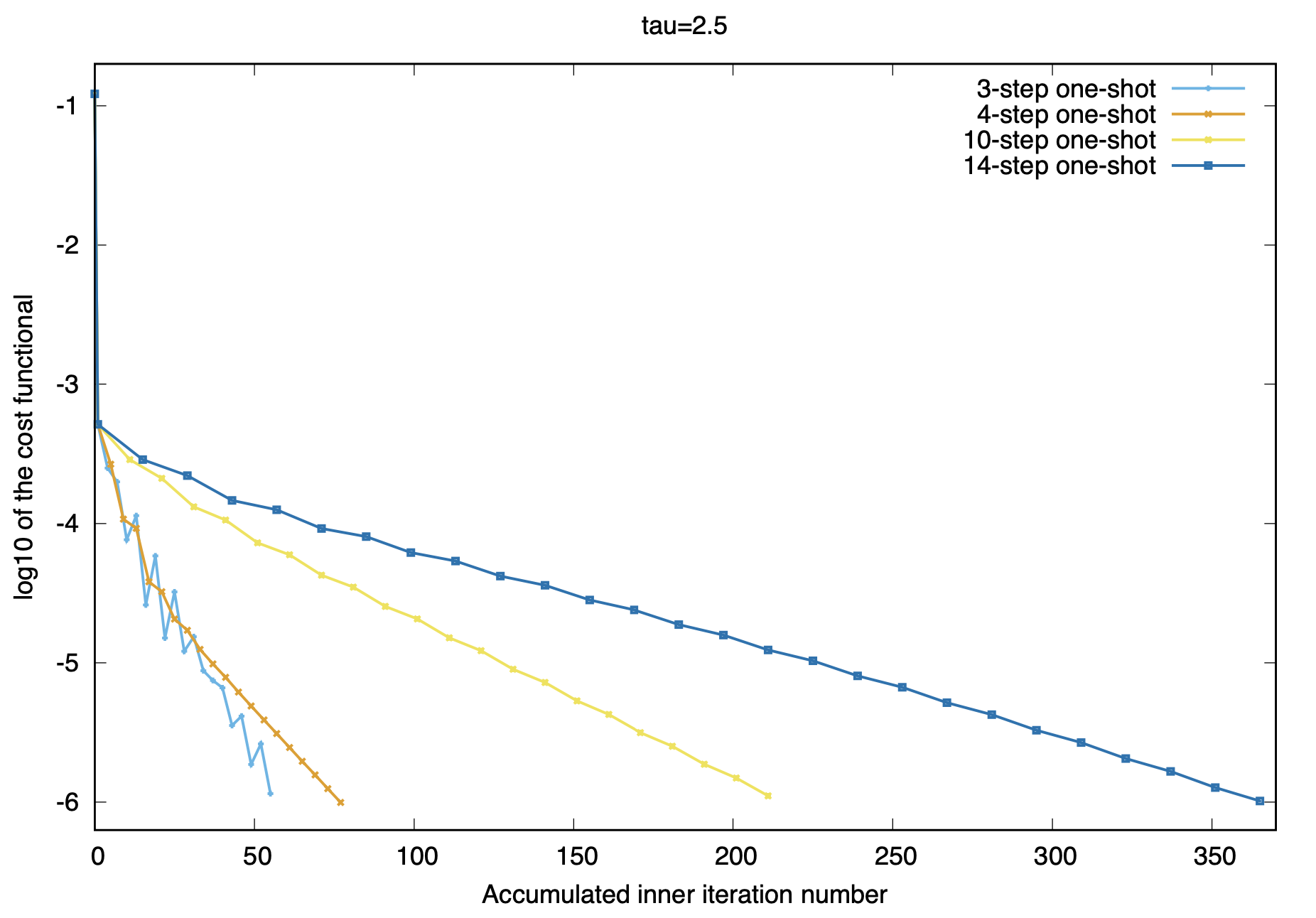}
		\caption{Convergence curves of $k$-step one-shot for different $k$ with $\tau=2.5$.}
		\label{fig:v2_usual_fix_tau_2.5}
	\end{subfigure}
	\caption{Convergence curves of usual gradient descent and $k$-step one-shot.}
	\label{fig:usual_fix_k,tau}
\end{figure}

In the first experiment, we study the dependence on the descent step $\tau$. In Figure \ref{fig:usual_fix_k_1} and \ref{fig:usual_fix_k_2} we respectively fix $k=1$ and $k=2$ and compare $k$-step one-shot methods with the usual gradient descent method. On the horizontal axis we indicate the (outer) iteration number $n$ in \eqref{usualgd} and \eqref{alg:k-shot n+1}. We can verify that for sufficiently small $\tau$, both one-shot methods converge. In particular, for $\tau=2$, while gradient descent and $2$-step one-shot converge, $1$-step one-shot diverges. Oscillations may appear on the convergence curve for certain values of $\tau$, but they gradually vanish when $\tau$ gets smaller. For sufficiently small $\tau$, the convergence curves of both one-shot methods are comparable to the one of gradient descent.

In the second experiment, we study the dependence on the number of inner iterations $k$, for fixed $\tau$. First (Figures~\ref{fig:usual_fix_tau_2}--\ref{fig:usual_fix_tau_2.5}), we investigate for which $k$ the convergence curve of $k$-step one-shot is comparable with the one of usual gradient descent. As in the previous pictures, on the horizontal axis we indicate the (outer) iteration number $n$ in \eqref{usualgd} and \eqref{alg:k-shot n+1}.
For $\tau=2$ (see Figure \ref{fig:usual_fix_tau_2}), we observe that for $k=3,4$ the convergence curves of $k$-step one-shot are close to the one of usual gradient descent. 
Note that with $3$ inner iterations the $\mathcal{L}^2$ error between $u^n$ and the exact solution to the forward problem ranges between $4.3\cdot10^{-6}$ and $0.0136$ for different $n$ in \eqref{alg:k-shot n+1}; in fact this error is rather significant at the beginning then it tends to reduce when we are closer to convergence for the parameter $\sigma$. 
Therefore incomplete inner iterations on the forward problem are enough to have good precision on the solution of the inverse problem.
In the very particular case $\tau =2.5$ (see Figure \ref{fig:usual_fix_tau_2.5}), we observe an interesting phenomenon: when $k=3,5,10$, with $k$-step one-shot the cost functional decreases even faster than with usual GD.
For bigger $k$, for example $k=14$, the convergence curve of one-shot is close to the one of usual gradient descent as expected. 
Next (Figures~\ref{fig:v2_usual_fix_tau_2}--\ref{fig:v2_usual_fix_tau_2.5}), since the overall cost of the $k$-step one-shot method increases with $k$, we indicate on the horizontal axis the accumulated inner iteration number, which sums up $k$ from an outer iteration to the next. More precisely, because at the first outer iteration we perform a step search by a direct solver, we set to $1$ the first accumulated inner iteration number; for the following outer iterations $n\ge 2$, the accumulated inner iteration number is set to $1+(n-1)k$. In Figures~\ref{fig:v2_usual_fix_tau_2}--\ref{fig:v2_usual_fix_tau_2.5} we replot the results for 
the converging $k$-step one-shot methods of Figures~\ref{fig:usual_fix_tau_2}--\ref{fig:usual_fix_tau_2.5} with respect to the accumulated inner iteration number. 
For $\tau=2$ (see Figure \ref{fig:v2_usual_fix_tau_2}), while $k=2$ presents some oscillations, quite interestingly it appears that $k=3$ gives a faster decrease of the cost functional with respect to $k=4$, at least after the first iterations. For $\tau=2.5$ (see Figure \ref{fig:v2_usual_fix_tau_2.5}) we observe that $k=3$ is enough for the decrease of the cost functional, but with some oscillations, and the considered higher $k$ appears again to give slower decrease. 
	
	\begin{figure}[htbp]
		\centering
		\begin{subfigure}{0.49\columnwidth}
			\includegraphics[width=\linewidth]{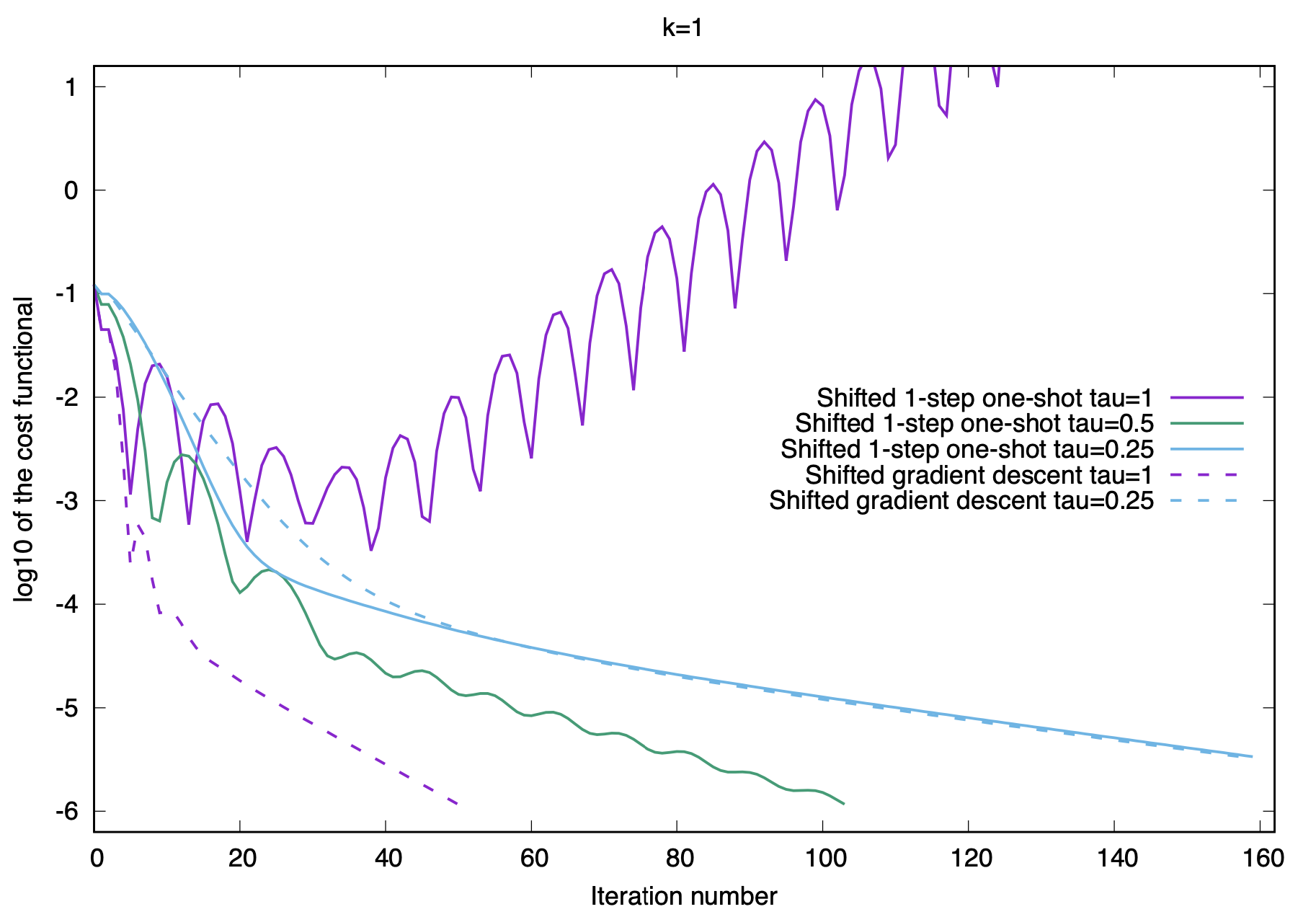}
			\caption{Convergence curves of shifted gradient descent and shifted $1$-step one-shot for different descent step $\tau$.}
			\label{fig:shift_fix_k_1}
		\end{subfigure} \hfill
		\begin{subfigure}{0.49\columnwidth}
			\includegraphics[width=\linewidth]{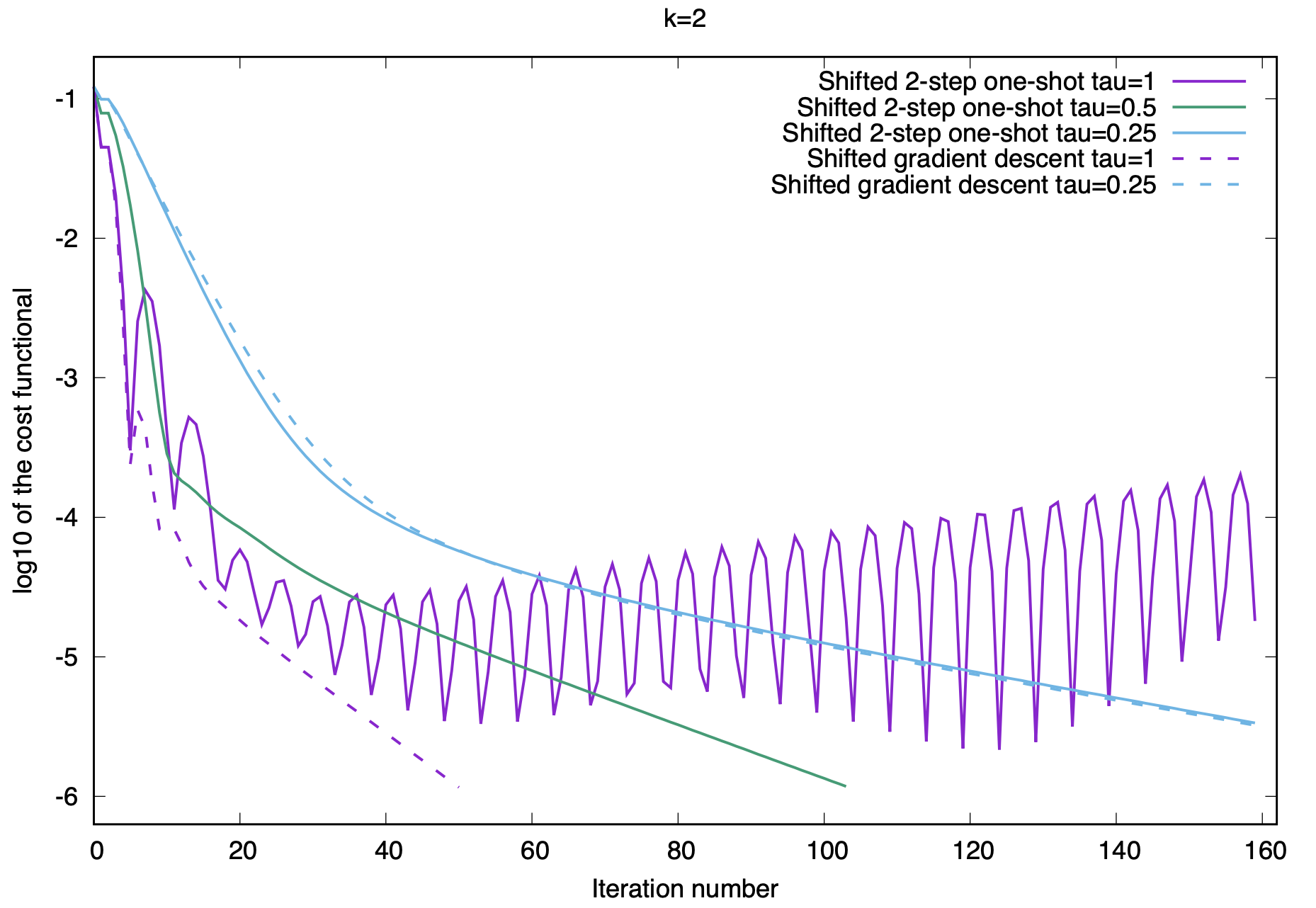}
			\caption{Convergence curves of shifted gradient descent and shifted $2$-step one-shot for different descent step $\tau$.}
			\label{fig:shift_fix_k_2}
		\end{subfigure}
		\begin{subfigure}{0.49\columnwidth}
			\includegraphics[width=\linewidth]{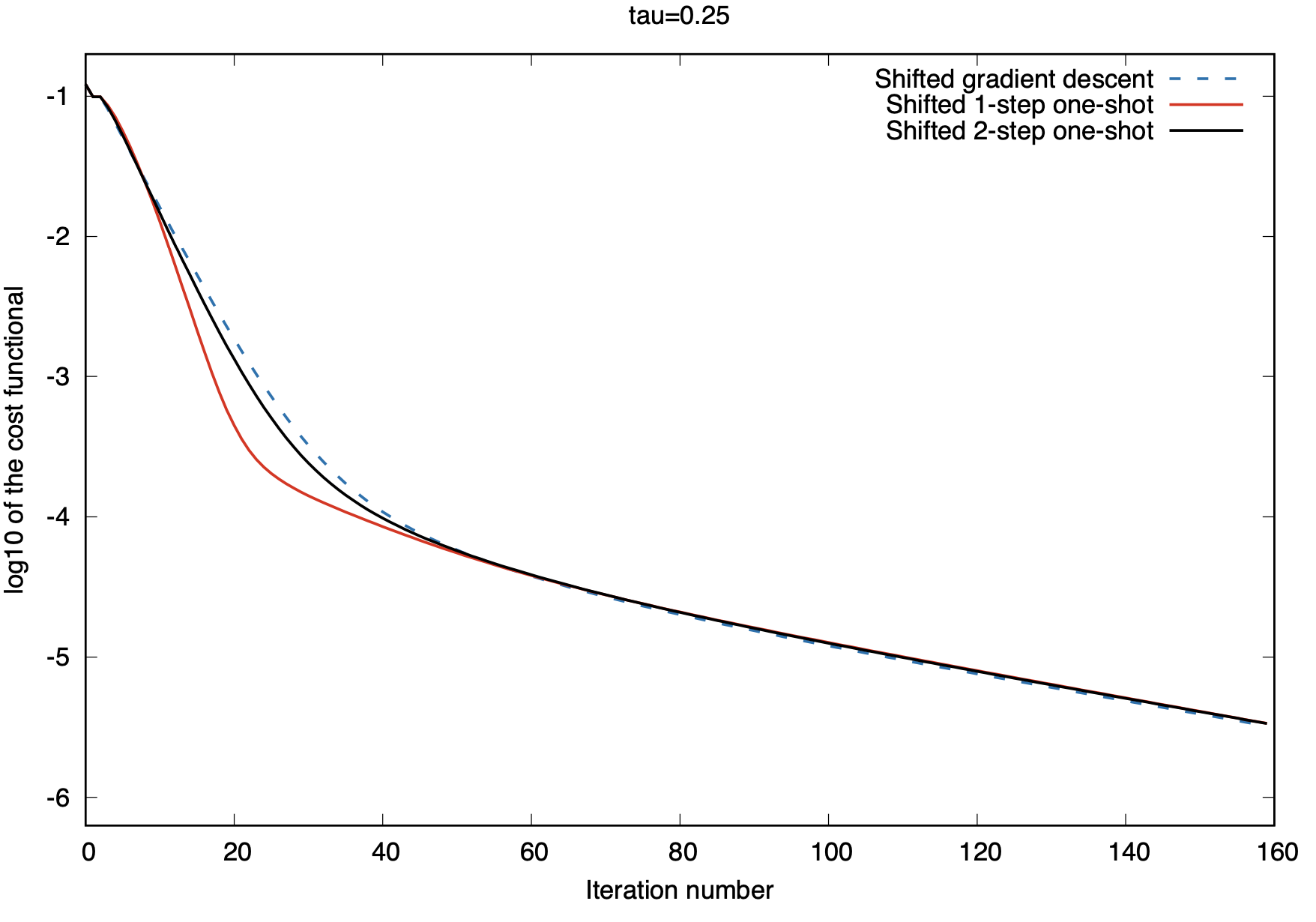}
			\caption{Convergence curves of shifted gradient descent and shifted $k$-step one-shot for different $k$ with $\tau=0.25$.}
			\label{fig:shift_fix_tau_0.25}
		\end{subfigure} \hfill
		\begin{subfigure}{0.49\columnwidth}
			\includegraphics[width=\linewidth]{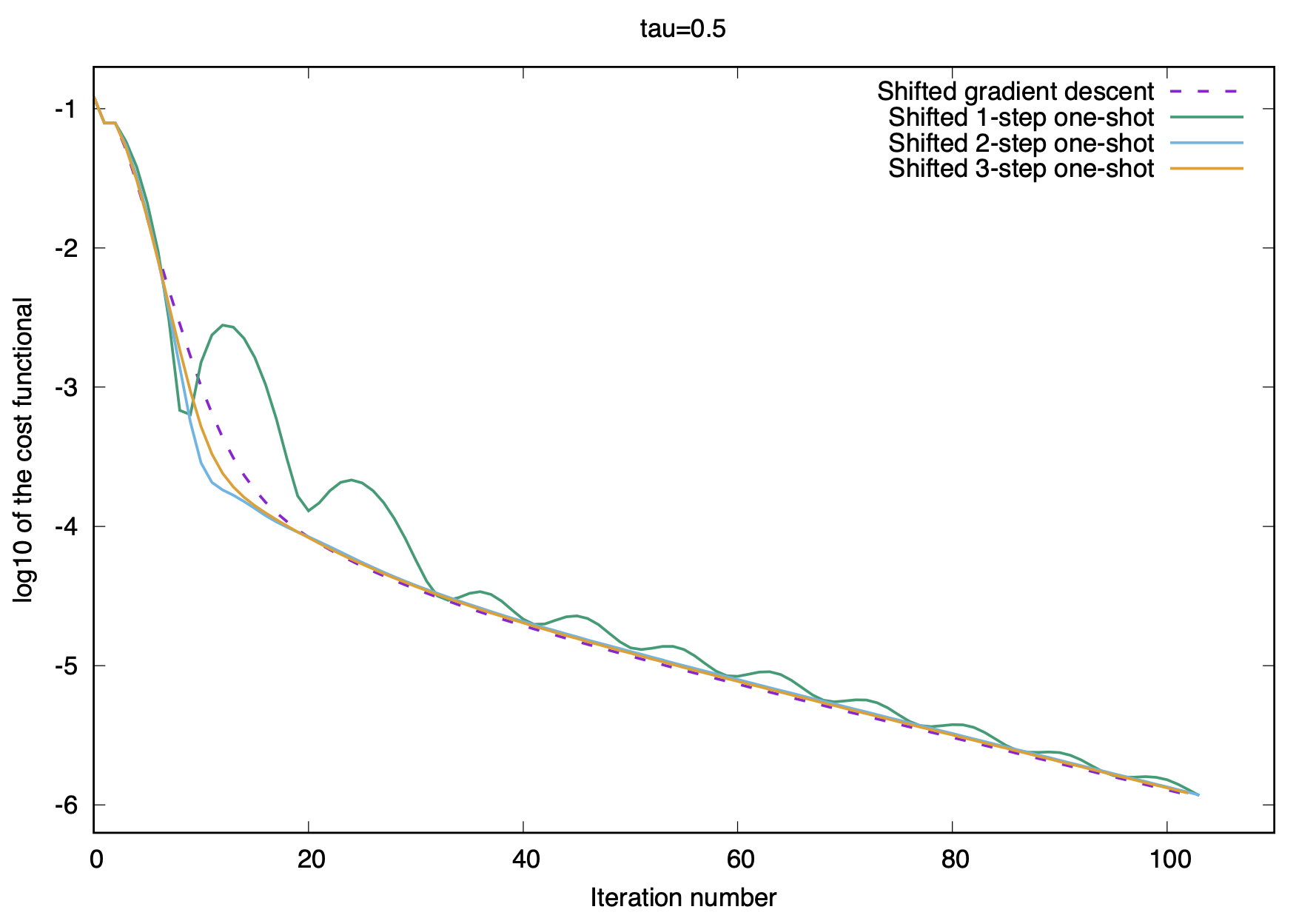}
			\caption{Convergence curves of shifted gradient descent and shifted $k$-step one-shot for different $k$ with $\tau=0.5$.}
			\label{fig:shift_fix_tau_0.5}
		\end{subfigure}
		\begin{subfigure}{0.49\columnwidth}
			\includegraphics[width=\linewidth]{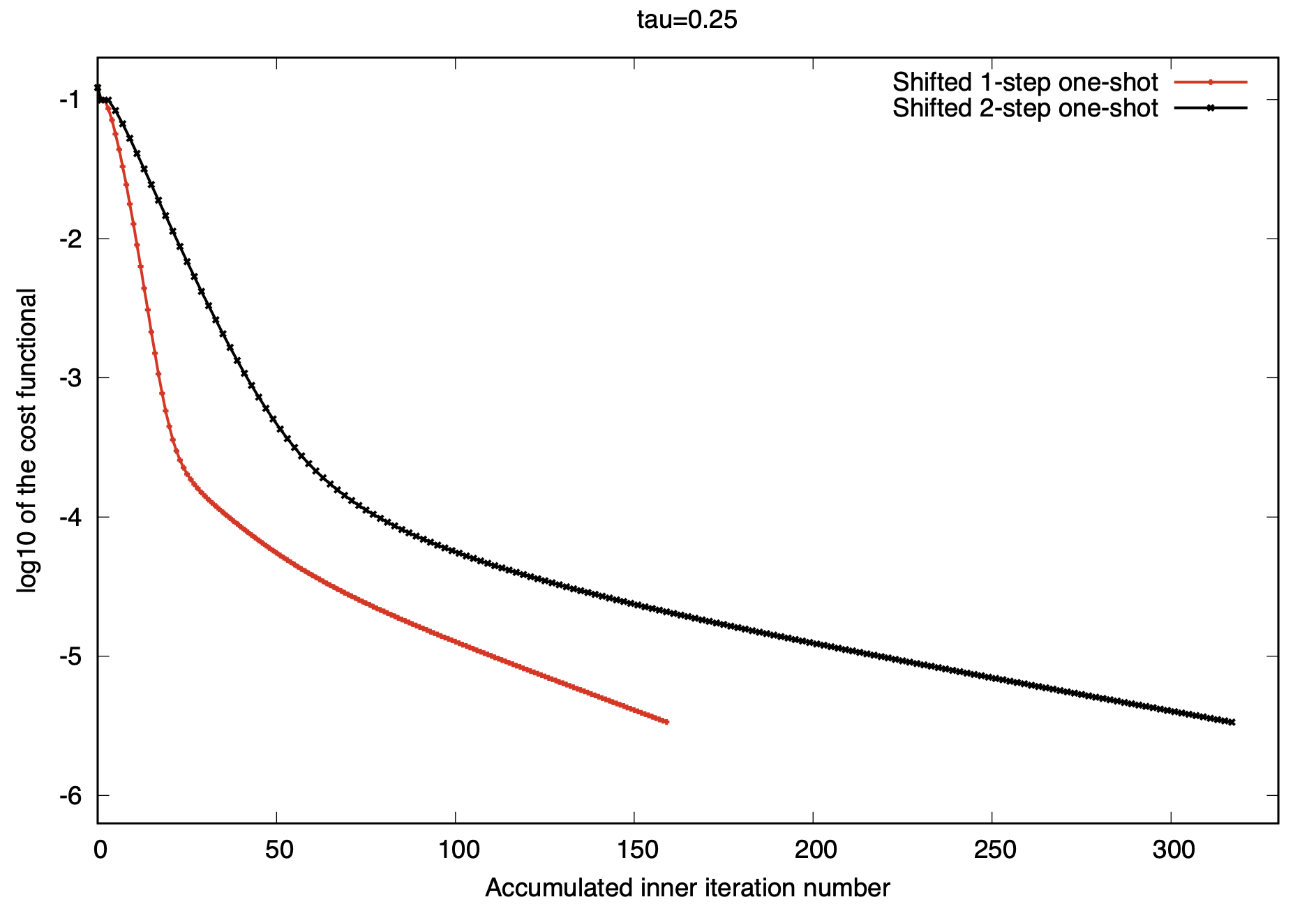}
			\caption{Convergence curves of shifted $k$-step one-shot for different $k$ with $\tau=0.25$.}
			\label{fig:v2_shift_fix_tau_0.25}
		\end{subfigure} \hfill
		\begin{subfigure}{0.49\columnwidth}
			\includegraphics[width=\linewidth]{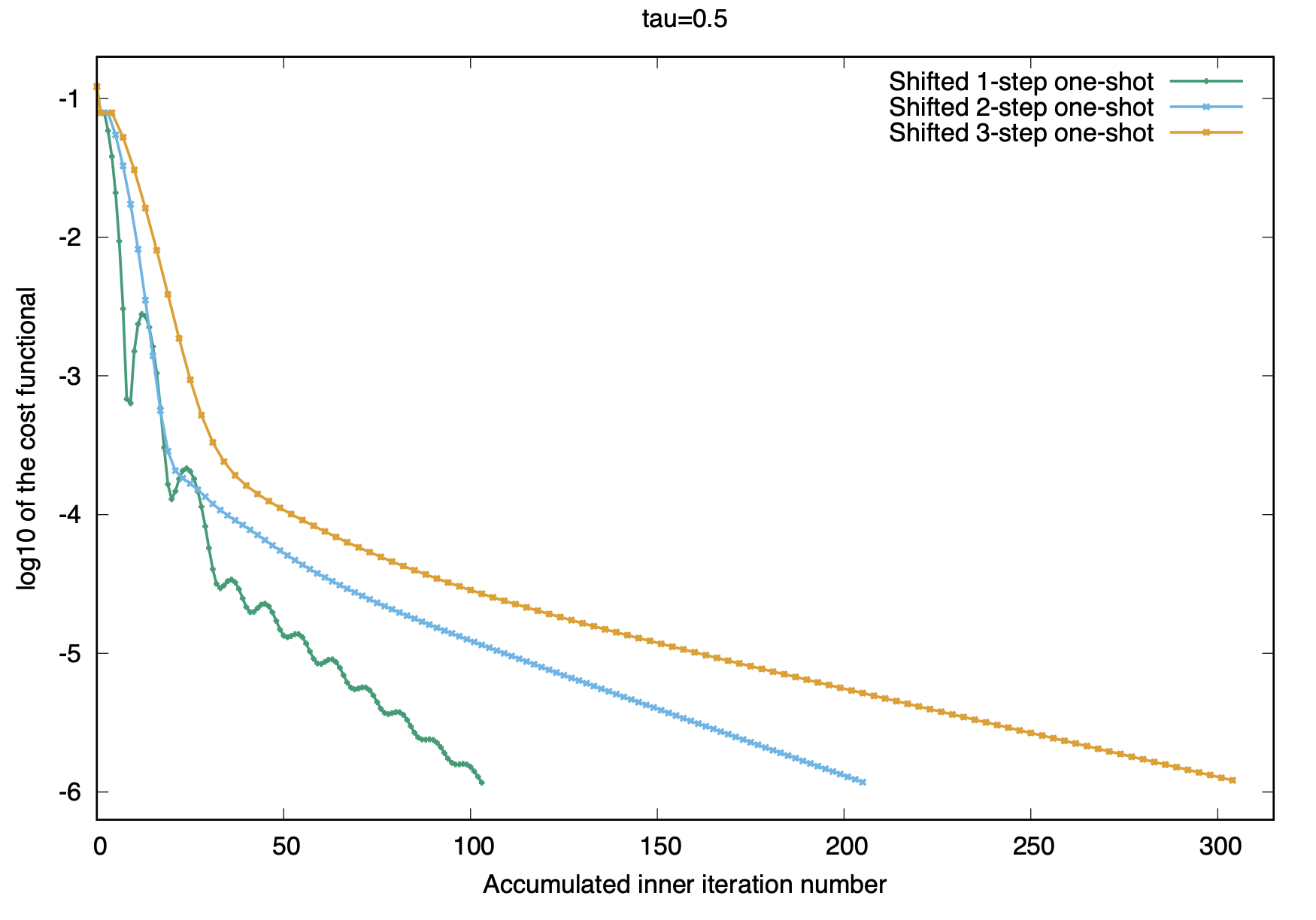}
			\caption{Convergence curves of shifted $k$-step one-shot for different $k$ with $\tau=0.5$.}
			\label{fig:v2_shift_fix_tau_0.5}
		\end{subfigure}
		\caption{Convergence curves of shifted gradient descent and shifted $k$-step one-shot.}
		\label{fig:shift_fix_k,tau}
	\end{figure}
	A similar behavior can be observed for the shifted methods in Figure \ref{fig:shift_fix_k,tau}.

	\begin{figure}[htbp]
		\centering
		\begin{subfigure}{0.49\columnwidth}
			\includegraphics[width=\linewidth]{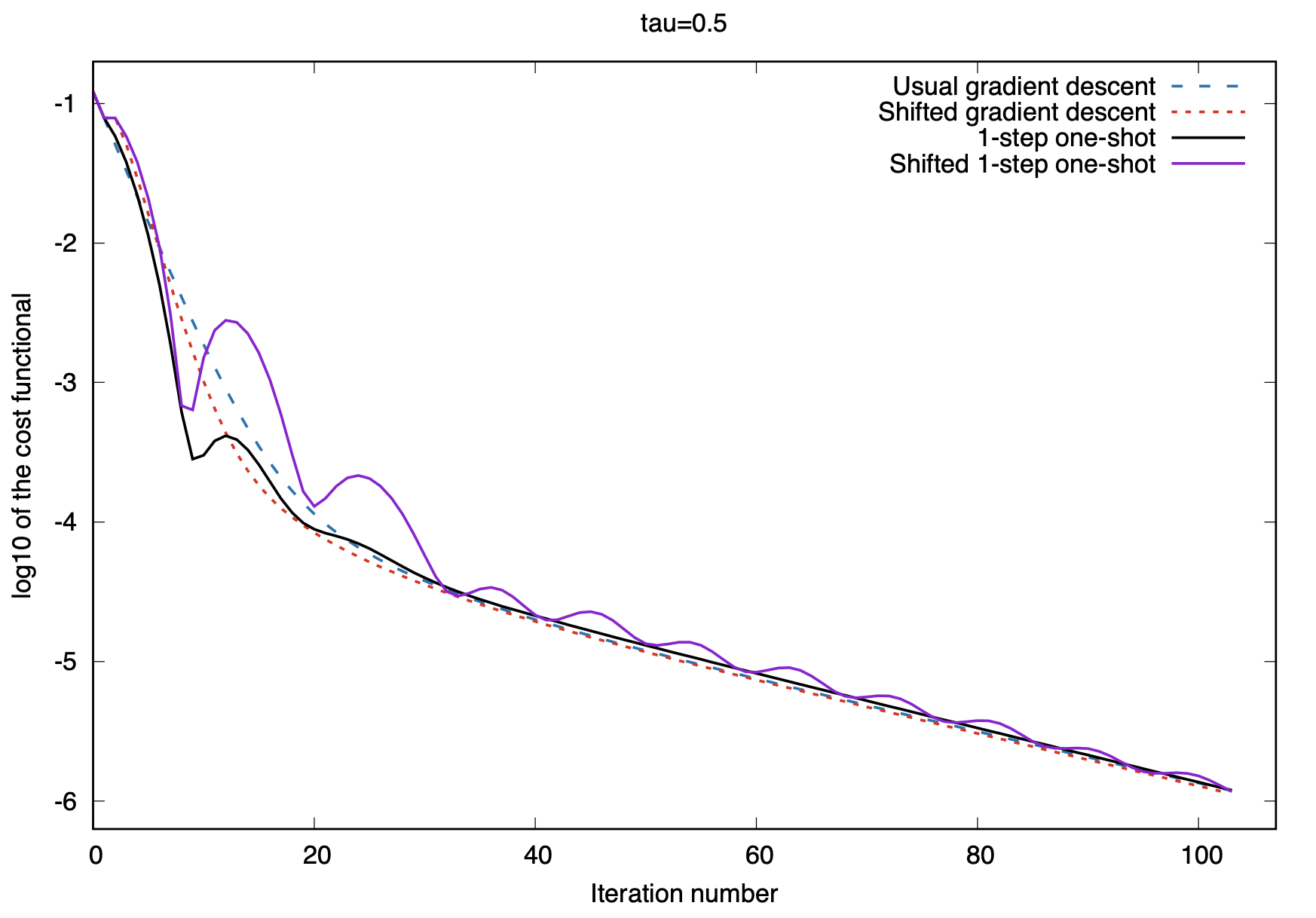}
			\caption{Convergence curves with $\tau=0.5$.}
			\label{fig:all_fix_tau_0.5}
		\end{subfigure} \hfill
		\begin{subfigure}{0.49\columnwidth}
			\includegraphics[width=\linewidth]{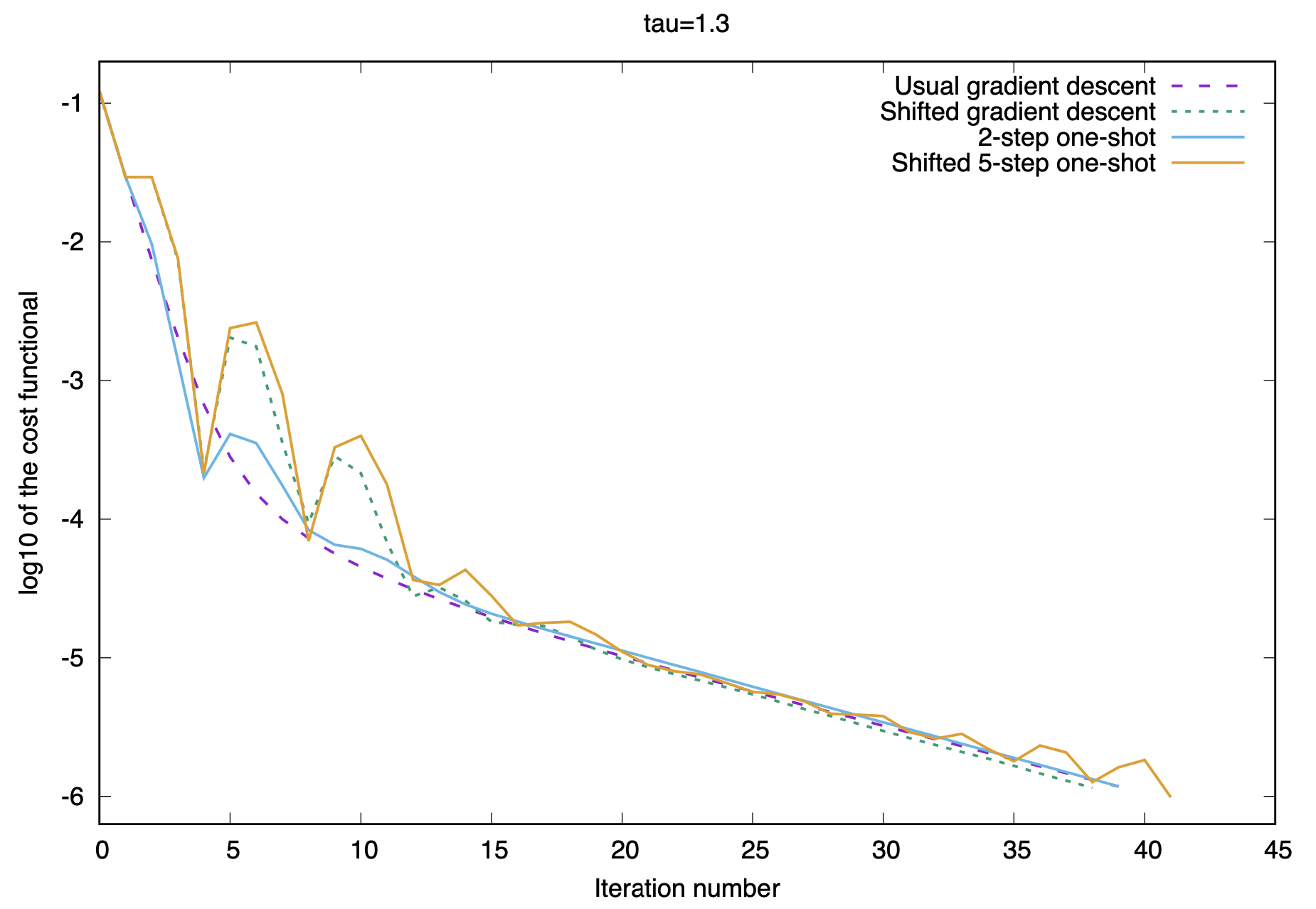}
			\caption{Convergence curves with $\tau=1.3$.}
			\label{fig:all_fix_tau_1.3}
		\end{subfigure}
		\caption{Comparison of usual gradient descent and $k$-step one-shot with shifted gradient descent and shifted $k$-step one-shot.}
		\label{fig:all_fix_tau}
	\end{figure}
	Finally we fix two particular values of $\tau$ and compare all considered methods in Figure \ref{fig:all_fix_tau}. We note that shifted methods present more oscillations with respect to non-shifted ones, especially for larger $\tau$. 

	\section{Conclusion}
We have proved sufficient conditions on the descent step for the convergence of two variants of multi-step one-shot methods. Although these bounds on the descent step are not optimal, to our knowledge no other bounds, explicit in the number of inner iterations, are available in literature for multi-step one-shot methods. Furthermore, we have shown in the numerical experiments that very few inner iterations on the forward and adjoint problems are enough to guarantee good convergence of the inversion algorithm. 

These encouraging numerical results are preliminary in the sense that the considered fixed point iteration is not a practical one, since it involves a direct solve of a problem of the same size as the original forward problem. We will investigate in the future iterative solvers based on domain decomposition methods (see e.g.~\cite{dolean-ddm15}), which are well adapted to large-scale problems. In addition, fixed point iterations could be replaced by more efficient Krylov subspace methods, such as conjugate gradient or GMRES.  

Another interesting issue is how to adapt the number of inner iterations in the course of the outer iterations. Moreover, based on this linear inverse problem study, we plan to tackle non-linear and time-dependent inverse problems.

\bibliographystyle{plain}	
\bibliography{k-shot}

\begin{thebibliography}{10}

\bibitem{barnett83}
S.~Barnett.
\newblock {\em Polynomials and linear control systems}, volume~77 of {\em Pure
  Appl. Math.}
\newblock Marcel Dekker, Inc., New York, NY, 1983.

\bibitem{burger02}
M.~Burger and W.~M{\"u}hlhuber.
\newblock Iterative regularization of parameter identification problems by
  sequential quadratic programming methods.
\newblock {\em Inverse Problems}, 18:943--969, 2002.

\bibitem{dolean-ddm15}
V.~Dolean, P.~Jolivet, and F.~Nataf.
\newblock {\em An {Introduction} to {Domain} {Decomposition} {Methods}:
  {Algorithms}, {Theory}, and {Parallel} {Implementation}}.
\newblock Society for Industrial and Applied Mathematics, Philadelphia, PA,
  2015.

\bibitem{gauger12}
N.~Gauger, A.~Griewank, A.~Hamdi, C.~Kratzenstein, E.~{\"O}zkaya, and
  T.~Slawig.
\newblock Automated extension of fixed point {PDE} solvers for optimal design
  with bounded retardation.
\newblock In {\em Constrained Optimization and Optimal Control for Partial
  Differential Equations, International Series of Numerical Mathematics}, pages
  99--122. Springer Basel, 2012.

\bibitem{greenbaum97}
A.~Greenbaum.
\newblock {\em Iterative {Methods} for {Solving} {Linear} {Systems}}.
\newblock Number~17 in Frontiers in {Applied} {Mathematics}. Soc. for
  Industrial and Applied Math, Philadelphia, 1997.

\bibitem{griewank06}
A.~Griewank.
\newblock Projected {Hessians} for {Preconditioning} in {One}-{Step}
  {One}-{Shot} {Design} {Optimization}.
\newblock In {\em Large-{Scale} {Nonlinear} {Optimization}}, volume~83, pages
  151--171. Springer US, Boston, MA, 2006.
\newblock Series Title: Nonconvex Optimization and Its Applications.

\bibitem{guenther16}
S.~Günther, N.~R. Gauger, and Q.~Wang.
\newblock Simultaneous single-step one-shot optimization with unsteady {PDEs}.
\newblock {\em Journal of Computational and Applied Mathematics}, 294:12--22,
  2016.

\bibitem{haber01}
E.~Haber and U.~M. Ascher.
\newblock Preconditioned all-at-once methods for large, sparse parameter
  estimation problems.
\newblock {\em Inverse Problems}, 17(6):1847--1864, 2001.

\bibitem{hamdi09}
A.~Hamdi and A.~Griewank.
\newblock Reduced quasi-{N}ewton method for simultaneous design and
  optimization.
\newblock {\em Computational Optimization and Applications}, 49(3):521--548,
  2009.

\bibitem{hamdi10}
A.~Hamdi and A.~Griewank.
\newblock Properties of an augmented {L}agrangian for design optimization.
\newblock {\em Optimization Methods and Software}, 25(4):645--664, 2010.

\bibitem{hazra05}
S.B. Hazra, V.~Schulz, J.~Brezillon, and N.R. Gauger.
\newblock Aerodynamic shape optimization using simultaneous
  pseudo-timestepping.
\newblock {\em Journal of Computational Physics}, 204(1):46--64, 2005.

\bibitem{FreeFEM}
F.~Hecht.
\newblock New development in {F}ree{F}em++.
\newblock {\em J. Numer. Math.}, 20(3-4):251--265, 2012.

\bibitem{jury63}
E.I. Jury.
\newblock On the roots of a real polynomial inside the unit circle and a
  stability criterion for linear discrete systems.
\newblock {\em IFAC Proceedings Volumes}, 1(2):142--153, 1963.
\newblock 2nd International IFAC Congress on Automatic and Remote Control:
  Theory, Basle, Switzerland, 1963.

\bibitem{jury64}
E.I. Jury.
\newblock {\em Theory and Applications of the Z-Transform Method}.
\newblock {New York}, 1964.

\bibitem{kaltenbacher14}
B.~Kaltenbacher, A.~Kirchner, and B.~Vexler.
\newblock {Goal oriented adaptivity in the IRGNM for parameter identification
  in PDEs II: all-at-once formulations}.
\newblock {\em Inverse Problems}, 30:045002, 2014.

\bibitem{marden49}
M.~Marden.
\newblock {\em The geometry of the zeros of a polynomial in a complex
  variable}, volume~3 of {\em Math. Surv.}
\newblock American Mathematical Society (AMS), Providence, RI, 1949.

\bibitem{marden66}
M.~Marden.
\newblock {\em Geometry of {Polynomials}}.
\newblock Number~3 in Mathematical {Surveys} and {Monographs}. American Math.
  Soc, Providence, RI, 2nd edition, 1966.

\bibitem{gauger09}
E.~{\"O}zkaya and N.~R. Gauger.
\newblock Single-step {One}-shot {Aerodynamic} {Shape} {Optimization}.
\newblock In {\em Optimal {Control} of {Coupled} {Systems} of {Partial}
  {Differential} {Equations}}, volume 158, pages 191--204. Birkh{\"a}user
  Basel, Basel, 2009.
\newblock Series Title: International Series of Numerical Mathematics.

\bibitem{schulz09}
V.~Schulz and I.~Gherman.
\newblock One-{Shot} {Methods} for {Aerodynamic} {Shape} {Optimization}.
\newblock In {\em {MEGADESIGN} and {MegaOpt} - {German} {Initiatives} for
  {Aerodynamic} {Simulation} and {Optimization} in {Aircraft} {Design}}, volume
  107, pages 207--220. Springer Berlin Heidelberg, Berlin, Heidelberg, 2009.
\newblock Series Title: Notes on Numerical Fluid Mechanics and
  Multidisciplinary Design.

\bibitem{schur17}
I.~Schur.
\newblock {\"U}ber {Potenzreihen}, die im {Innern} des {Einheitskreises}
  beschr{\"a}nkt sind.
\newblock {\em Journal f{\"u}r die reine und angewandte Mathematik (Crelles
  Journal)}, 1917(147):205--232, 1917.

\bibitem{shenoy97}
A.~Shenoy, M.~Heinkenschloss, and E.~M. Cliff.
\newblock Airfoil design by an all-at-once method.
\newblock {\em International Journal of Computational Fluid Dynamics},
  11(1-2):3--25, 1998.

\bibitem{taasan91}
S.~Ta'asan.
\newblock "{One} {Shot}" {Methods} for {Optimal} {Control} of {Distributed}
  {Parameter} {Systems} {I}: {Finite} {Dimensional} {Control}.
\newblock Technical Report 91-2, ICASE, Hampton, 1991.

\bibitem{taasan92}
S.~Ta'asan, G.~Kuruvila, and M.~Salas.
\newblock Aerodynamic design and optimization in one shot.
\newblock In {\em 30th {Aerospace} {Sciences} {Meeting} and {Exhibit}}, Reno,
  NV, U.S.A., 1992. American Institute of Aeronautics and Astronautics.

\bibitem{tarantola82}
A.~Tarantola and B.~Valette.
\newblock Generalized nonlinear inverse problems solved using the least squares
  criterion.
\newblock {\em Reviews of Geophysics}, 20(2):219--232, 1982.

\bibitem{leeuwen13}
T.~van Leeuwen and F.~J. Herrmann.
\newblock {Mitigating local minima in full-waveform inversion by expanding the
  search space}.
\newblock {\em Geophysical Journal International}, 195(1):661--667, 2013.

\bibitem{leeuwen15}
T.~van Leeuwen and F.~J. Herrmann.
\newblock A penalty method for {PDE}-constrained optimization in inverse
  problems.
\newblock {\em Inverse Problems}, 32(1):015007, 2015.

\end{thebibliography}

\newpage
\appendix 
\section{Some useful lemmas}
\label{app:lems}
We state auxiliary results about matrices like those appearing in the eigenvalue equations \eqref{eq-eigen-1-shot n}, \eqref{eq-eigen-1-shot n+1}, \eqref{eq-eigen n}, \eqref{eq-eigen n+1}. 
\begin{lemma}\label{inv(I-T)}
	Let $(\C^{n\times n},\norm{\cdot})$ be a normed space and $T\in\C^{n\times n}$. If $\rho(T)<1$, then 
	$$\sum_{k=0}^{\infty}T^k \text{ converges and } \sum_{k=0}^{\infty}T^k=(I-T)^{-1}.$$
	Moreover, if $\norm{T}<1$, $\norm{(I-T)^{-1}}\le \frac{1}{1-\norm{T}}$.	
\end{lemma}

\begin{lemma}\label{inv(I-T/z)}
	Let $T\in\C^{n\times n}$ such that $\rho(T)<1$. Set
	\begin{equation}
		\label{supbound}
		s(T) \coloneqq \sup_{z\in\C, |z|\ge 1}\norm{\left(I-T/z\right)^{-1}}
	\end{equation}
	then $0<s(T)<+\infty$. Moreover, if $\norm{T}<1$, $0<s(T)\le\cfrac{1}{1-\norm{T}}$.
\end{lemma}
\begin{proof}
	The functional $z\mapsto \norm{\left(I-T/z\right)^{-1}}$, with $z\in\C, |z|\ge 1$, is well-defined and continuous, and we use Lemma~\ref{inv(I-T)}.
\end{proof}

The following lemma says that, for $T\in\C^{n\times n}$ and $\lambda\in\C, |\lambda|\ge 1$, we can decompose
$$\left(I-\cfrac{T}{\lambda}\right)^{-1}=P(\lambda)+\ic Q(\lambda)\quad\text{and}\quad\left(I-\cfrac{T^*}{\lambda}\right)^{-1}=P(\lambda)^*+\ic Q(\lambda)^*$$
and gives bounds for $P(\lambda)$ and $Q(\lambda)$. 
\begin{lemma}\label{decompq}
	Let $T\in\C^{n\times n}$ such that $\rho(T)<1$ and $\lambda\in\C, |\lambda|\ge 1$. Write $\frac{1}{\lambda}=r(\cos\phi+\ic\sin\phi)$ in polar form, where $0<r\le 1$ and $\phi\in[-\pi,\pi]$. Then
	$$\left(I-\cfrac{T}{\lambda}\right)^{-1}=P(\lambda)+\ic Q(\lambda)\quad\text{and}\quad\left(I-\cfrac{T^*}{\lambda}\right)^{-1}=P(\lambda)^*+\ic Q(\lambda)^*$$
	where
	$$P(\lambda)=(I-r\cos\phi \, T)(I-2r\cos\phi \, T +r^2 T^2)^{-1}, \quad Q(\lambda)=r\sin\phi \, T(I-2r\cos\phi \, T +r^2 T^2)^{-1}$$
	are $\C^{n\times n}$-valued functions. We also have the following properties:
	
	\begin{enumerate}[label=(\roman*)]
		\item $\norm{P(\lambda)}\le \left(1+\norm{T}\right)s(T)^2$ and $\norm{Q(\lambda)}\le|\sin\phi|\norm{T}s(T)^2\le\norm{T}s(T)^2.$
	
		\item Moreover if $\norm{T}<1$ then		
		$$\norm{P(\lambda)}\le\cfrac{1}{1-\norm{T}}\quad\mbox{and}\quad \norm{Q(\lambda)} \le\cfrac{\norm{T}}{1-\norm{T}}.$$	
	\end{enumerate} 
\end{lemma}
\begin{proof} The first part of the lemma is verified by direct computation, using
	$$\left(I-T/\lambda\right)^{-1}=\left(I-T/\lambda^*\right)\left[\left(I-T/\lambda\right)\left(I-T/\lambda^*\right)\right]^{-1},$$
	$$\left(I-T^*/\lambda\right)^{-1}=\left[\left(I-T^*/\lambda^*\right)\left(I-T^*/\lambda\right)\right]^{-1}\left(I-T^*/\lambda^*\right)$$
	and
	$$\left(I-T/\lambda\right)\left(I-T/\lambda^*\right)=I-2r\cos\phi \, T +r^2 T^2.$$
	After that, with the help of Lemma \ref{inv(I-T/z)}, it is not difficult to show the inequalities in (i). To prove (ii), first observe that the two series 
	$$\sum_{k=0}^\infty r^k\cos(k\phi) T^k\quad\mbox{and}\quad\sum_{k=1}^\infty r^k\sin(k\phi) T^k$$
	converge. Then, by expanding and simplifying the left-hand sides, we can show that 
	$$\left[\displaystyle\sum_{k=0}^\infty r^k\cos(k\phi) T^k\right](I-2r\cos\phi \, T +r^2 T^2)=I-r\cos\phi \, T$$
	and
	$$\left[\displaystyle\sum_{k=1}^\infty r^k\sin(k\phi) T^k\right](I-2r\cos\phi \, T +r^2 T^2)= r\sin\phi \, T$$
	so $P(\lambda)$ and $Q(\lambda)$ can be expressed as the series above, and the inequalities in (ii) follow.
\end{proof}

In Sections \ref{subsec:complex-1-step} and \ref{subsec:complex-k-step} we identify different cases of $\lambda\in\C$ and we need corresponding estimations, given in the two following lemmas.
Lemma~\ref{gamma123 n} is used for the shifted $k$-step one-shot method and Lemma~\ref{gamma123 n+1} is used for the $k$-step one-shot method. 
\begin{lemma}
	\label{gamma123 n}
	For $\lambda\in\C\setminus{\R}, |\lambda|\ge 1$ we write $\lambda=R(\cos\theta+\ic\sin\theta)$ in polar form where $R\ge 1$, $\theta\in(-\pi,\pi)$, $\theta \ne 0$. 
	\begin{enumerate}[label=(\roman*)]
		\item For $\lambda$ satisfying $\Re(\lambda^3-\lambda^2)\ge 0$, 
		let $\gamma_1=\gamma_1(\lambda)=\left\{\begin{array}{cc}
			1, & \text{if }\Im (\lambda^3-\lambda^2)\ge 0,\\
			-1, & \text{if }\Im (\lambda^3-\lambda^2)<0
		\end{array}\right.$ then
		$$\Re(\lambda^3-\lambda^2)+\gamma_1\Im(\lambda^3-\lambda^2)\ge |\lambda-1|\ge 2|\sin(\theta/2)|.$$
		\item Let $0<\theta_0\le\frac{\pi}{6}$. For $\lambda$ satisfying $\Re(\lambda^3-\lambda^2)<0$ and $\theta\in[\theta_0,\pi-\theta_0]\cup[-\pi+\theta_0,-\theta_0]$, 
		
		let $\gamma_2=\left\{\begin{array}{cc}
			-1, & \text{if }\Im (\lambda^3-\lambda^2)\ge 0,\\
			1, & \text{if }\Im (\lambda^3-\lambda^2)<0
		\end{array}\right.$ then $$-\Re(\lambda^3-\lambda^2)-\gamma_2\Im(\lambda^3-\lambda^2)\ge |\lambda-1| \ge 2\sin({\theta_0}/{2}).$$
		\item Let $0<\theta_0\le\frac{\pi}{6}$ and $\delta_0>0$. For $\lambda$ satisfying $\Re(\lambda^3-\lambda^2)<0$ and $\theta\in(-\theta_0,\theta_0)\backslash\{0\}$, let $\gamma_3=\gamma_3(\mathrm{sign}(\theta))=\left\{\begin{array}{cc}
			\left(\delta_0+\sin\frac{5\theta_0}{2}\right)/\cos\frac{5\theta_0}{2} & \text{if }\theta>0,\\
			-\left(\delta_0+\sin\frac{5\theta_0}{2}\right)/\cos\frac{5\theta_0}{2} & \text{if }\theta<0
		\end{array}\right.$ then
		$$\Re(\lambda^3-\lambda^2)+\gamma_3\Im(\lambda^3-\lambda^2)\ge 2\delta_0|\sin(\theta/2)|.$$
		Moreover, if $0<\theta_0<\frac{\pi}{6}$, we have
		\begin{center}
			$\frac{|\Re(\lambda-1)+\gamma_3\Im(\lambda-1)|}{\Re(\lambda^3-\lambda^2)+\gamma_3\Im(\lambda^3-\lambda^2)}\le\frac{\sqrt{1+\gamma_3^2}}{\delta_0}
			\quad\mbox{and}\quad
			\frac{|\gamma_3\Re(\lambda-1)-\Im(\lambda-1)|}{\Re(\lambda^3-\lambda^2)+\gamma_3\Im(\lambda^3-\lambda^2)}\le\max\left(\frac{\sqrt{1+\gamma_3^2}}{\delta_0},\frac{\sqrt{1+\gamma_3^2}}{\cos3\theta_0}\right).$
		\end{center}
		\item Let $0<\theta_0\le\frac{\pi}{6}$. For $\lambda$ satisfying $\Re(\lambda^3-\lambda^2)<0$ and $\theta \in (\pi-\theta_0,\pi)\cup(-\pi,-\pi+\theta_0)$, we have
		$$-\Re(\lambda^3-\lambda^2)\ge \sin\left(\frac{\pi}{2}-3\theta_0\right)+\cos 2\theta_0,$$
		$$ \cfrac{|\Re(\lambda-1)|}{-\Re(\lambda^3-\lambda^2)}\le \frac{2}{\sin\left(\frac{\pi}{2}-3\theta_0\right)+\cos2\theta_0}\quad\mbox{ and }\quad \cfrac{|\Im(\lambda-1)|}{-\Re(\lambda^3-\lambda^2)}\le\frac{2}{\sin\left(\frac{\pi}{2}-3\theta_0\right)+\cos2\theta_0}.$$
	\end{enumerate}
\end{lemma}
\begin{proof}
	(i) From the definition of $\gamma_1$ we see that $\gamma_1^2=1$, $\gamma_1\Im(\lambda^3-\lambda^2)\ge 0$ and
	$$\begin{array}{lll}
		\left[\Re(\lambda^3-\lambda^2)+\gamma_1\Im(\lambda^3-\lambda^2)\right]^2&=&\left[\Re(\lambda^3-\lambda^2)\right]^2+\left[\Im(\lambda^3-\lambda^2)\right]^2 +2\gamma_1 \Re(\lambda^3-\lambda^2)\Im(\lambda^3-\lambda^2)\\
		&\ge& \left[\Re(\lambda^3-\lambda^2)\right]^2+\left[\Im(\lambda^3-\lambda^2)\right]^2=|\lambda^3-\lambda^2|^2,\\
	\end{array}
	$$
	which yields $\Re(\lambda^3-\lambda^2)+\gamma_1\Im(\lambda^3-\lambda^2)\ge R^2|\lambda-1|$. Finally,
	$$|\lambda-1|=|R\cos\theta-1+\ic R\sin\theta|=\sqrt{R^2+1-2R\cos\theta}\ge\sqrt{2-2\cos\theta}=2|\sin(\theta/2)|$$
	since the function $R\mapsto R^2+1-2R\cos\theta$, for $R\ge 1$, is increasing.
	
	\noindent (ii) In this case we have $\frac{\theta}{2}\in \left[\frac{\theta_0}{2},\frac{\pi}{2}-\frac{\theta_0}{2}\right]\cup\left[-\frac{\pi}{2}+\frac{\theta_0}{2},-\frac{\theta_0}{2}\right]$ so $\left|\sin\frac{\theta}{2}\right|\ge\sin\frac{\theta_0}{2}$. From the definition of $\gamma_2$ we see that $\gamma_2^2=1$ and $\gamma_2\Im(\lambda^3-\lambda^2)\le 0$. Similar to (i), we have
	$-\Re(\lambda^2-\lambda)-\gamma_2\Im(\lambda^2-\lambda)\ge |\lambda-1|\ge 2|\sin(\theta/2)|$, that implies the conclusion.
	
	\noindent (iii) Note that $\cos3\theta>0, -\frac{\pi}{2}<3\theta<\frac{\pi}{2}$, and $\sin 3\theta$ has the same sign as $\theta$ and $\gamma_3$, so we have 
	$$\begin{array}{ll}
		\Re(\lambda^3-\lambda^2)+\gamma_3\Im(\lambda^3-\lambda^2)&=R^2(R\cos 3\theta-\cos 2\theta+\gamma_3 R\sin 3\theta-\gamma_3\sin 2\theta)\\
		&\ge \cos 3\theta-\cos 2\theta+\gamma_3\sin 3\theta-\gamma_3\sin2\theta\\
		&= -2\sin\frac{5\theta}{2}\sin\frac{\theta}{2}+2\gamma_3\cos\frac{5\theta}{2}\sin\frac{\theta}{2} \\
		&= 2\sin\frac{\theta}{2}\left(\gamma_3\cos\frac{5\theta}{2}-\sin\frac{5\theta}{2}\right).
	\end{array}
	$$
	Then we consider two cases: if $0<\theta<\theta_0$ then $\gamma_3>0$, $\left|\sin\frac{\theta}{2}\right|=\sin\frac{\theta}{2}>0$, 
	$0<\frac{5\theta}{2}<\frac{5\theta_0}{2}<\frac{\pi}{2}$ and $\gamma_3\cos\frac{5\theta}{2}-\sin\frac{5\theta}{2}>\gamma_3\cos\frac{5\theta_0}{2}-\sin\frac{5\theta_0}{2}=\delta_0$;	if $-\theta_0<\theta<0$ then $-\gamma_3>0$, $\left|\sin\frac{\theta}{2}\right|=-\sin\frac{\theta}{2}>0$,
	$-\frac{\pi}{2}<-\frac{5\theta_0}{2}<\frac{5\theta}{2}<0$ and $-\gamma_3\cos\frac{5\theta}{2}+\sin\frac{5\theta}{2}>-\gamma_3\cos\frac{5\theta_0}{2}-\sin\frac{5\theta_0}{2}=\delta_0$.
	
	Next, if $0<\theta_0<\frac{\pi}{6}$, we will show that $\frac{|\Re(\lambda-1)+\gamma_3\Im(\lambda-1)|}{\Re(\lambda^3-\lambda^2)+\gamma_3\Im(\lambda^3-\lambda^2)}$ and $\frac{|\gamma_3\Re(\lambda-1)-\Im(\lambda-1)|}{\Re(\lambda^3-\lambda^2)+\gamma_3\Im(\lambda^3-\lambda^2)}$	are both bounded. First,
	$$\begin{array}{ll}
		\cfrac{|\Re(\lambda-1)+\gamma_3\Im(\lambda-1)|}{\Re(\lambda^3-\lambda^2)+\gamma_3\Im(\lambda^3-\lambda^2)}&=\cfrac{|(\cos\theta+\gamma_3\sin\theta)R-1|}{R^2[(\cos 3\theta+\gamma_3\sin 3\theta)R-(\cos2\theta+\gamma_3\sin2\theta)]}\\&\le \cfrac{|(\cos\theta+\gamma_3\sin\theta)R-1|}{(\cos 3\theta+\gamma_3\sin3\theta)R-(\cos2\theta+\gamma_3\sin2\theta)}.
	\end{array}
	$$
	Since $\gamma_3$ does not depend on $R$, let us study $f_1(R)=\left(\frac{aR-1}{bR-c}\right)^2$	where
	$a=\cos\theta+\gamma_3\sin\theta$, $b=\cos 3\theta+\gamma_3\sin 3\theta$ and  $c=\cos2\theta+\gamma_3\sin2\theta$. We observe that:
	\begin{itemize}
		\item $a,b,c>0$. Indeed, $\cos\theta,\cos 2\theta,\cos3\theta>0$, and $\theta$ and $\gamma_3$ have the same sign.
		\item $bR-c>0$ since $\Re(\lambda^3-\lambda^2)+\gamma_3\Im(\lambda^3-\lambda^2)>0$, thus $R>\frac{c}{b}$.
		\item $ac>b$ (equivalently $\frac{c}{b}>\frac{1}{a}$), since $$ac=\cos\theta\cos2\theta+\gamma_3^2\sin\theta\sin2\theta+\gamma_3\sin 3\theta>\cos\theta\cos2\theta-\sin\theta\sin2\theta+\gamma_3\sin 3\theta=b.$$
	\end{itemize}
	Now, $f_1'(R)=2\cdot\frac{aR-1}{bR-c}\cdot\frac{b-ac}{(bR-c)^2}<0$
	for $R>\frac{c}{b}>\frac{1}{a}$ and we would like to have $\frac{c}{b}<1$ so that $f_1(R)\le f_1(1), \forall R\ge 1$. Indeed $\frac{c}{b}<1$ is equivalent to
	$$\cos2\theta+\gamma_3\sin2\theta < \cos3\theta+\gamma_3\sin3\theta  \Leftrightarrow|\gamma_3|>\frac{\left|\sin\frac{5\theta}{2}\right|}{\cos\frac{5\theta}{2}},$$
	which is true since
	$$|\gamma_3|=\frac{\delta_0+\sin\frac{5\theta_0}{2}}{\cos\frac{5\theta_0}{2}}>\frac{\left|\sin\frac{5\theta}{2}\right|}{\cos\frac{5\theta}{2}}+\varepsilon_0 \quad\mbox{where}\quad\varepsilon_0=\frac{\delta_0}{\cos\frac{5\theta_0}{2}}.$$
	Then we study
	$$f_1(1)=\left[ \frac{\cos\theta-1+\gamma_3\sin\theta}{\cos3\theta-\cos2\theta+\gamma_3(\sin3\theta-\sin2\theta)}\right]^2=\left(\cfrac{-\sin\frac{\theta}{2}+\gamma_3\cos\frac{\theta}{2}}{-\gamma_3\sin\frac{5\theta}{2}+\gamma_3^2\cos\frac{5\theta}{2}}\right)^2\gamma_3^2.
	$$
	We have:
	\begin{itemize}
		\item $(-\sin\frac{\theta}{2}+\gamma_3\cos\frac{\theta}{2})^2\le 1+\gamma_3^2$ by Cauchy-Schwarz inequality;
		\item $\gamma_3^2=|\gamma_3|^2>\frac{\gamma_3\sin\frac{5\theta}{2}}{\cos\frac{5\theta}{2}}+\varepsilon_0|\gamma_3|$ that leads to $-\gamma_3\sin\frac{5\theta}{2}+\gamma_3^2\cos\frac{5\theta}{2}>\varepsilon_0\cos\frac{5\theta_0}{2}|\gamma_3|=\delta_0|\gamma_3|;$
	\end{itemize}
	hence $f_1(1)\le\frac{1+\gamma_3^2}{\delta_0^2}$ and finally $\frac{|\Re(\lambda-1)+\gamma_3\Im(\lambda-1)|}{\Re(\lambda^3-\lambda^2)+\gamma_3\Im(\lambda^3-\lambda^2)}\le\frac{\sqrt{1+\gamma_3^2}}{\delta_0}$. Next, we have
	$$\begin{array}{ll}
		\cfrac{|\gamma_3\Re(\lambda-1)-\Im(\lambda-1)|}{\Re(\lambda^2-\lambda)+\gamma_3\Im(\lambda^2-\lambda)}&=\cfrac{|(\gamma_3\cos\theta-\sin\theta)R-\gamma_3|}{R^2[(\cos 3\theta+\gamma_3\sin 3\theta)R-(\cos2\theta+\gamma_3\sin2\theta)]}\\&\le \cfrac{|(\gamma_3\cos\theta-\sin\theta)R-\gamma_3|}{(\cos 3\theta+\gamma_3\sin3\theta)R-(\cos2\theta+\gamma_3\sin2\theta)}.
	\end{array}$$
	Since $\gamma_3$ does not depend on $R$, let us study $f_2(R)=\left(\frac{dR-\gamma_3}{bR-c}\right)^2$
	where $d=\gamma_3\cos\theta-\sin\theta$ and $b, c$ as above. We observe that:
	\begin{itemize}
		\item $\gamma_3b-cd$ and $\theta$ have the same sign. Indeed, $\gamma_3b-cd=(\gamma_3^2+1)\sin\theta\cos2\theta$. Consequently, we always have $(\gamma_3b-cd)\gamma_3>0$.
		
		\item We always have $\frac{\gamma_3}{d}>1$. Indeed, if $\theta>0$ then $d>0$ since $\gamma_3=\frac{\delta_0+\sin\frac{5\theta_0}{2}}{\cos\frac{5\theta_0}{2}}>\frac{\sin\theta}{\cos\theta}$, also 
		$\frac{\gamma_3}{d}=\frac{\gamma_3}{\gamma_3\cos\theta-\sin\theta}>1$; if $\theta<0$ then $d<0$ since 
		$-\gamma_3=\frac{\delta_0+\sin\frac{5\theta_0}{2}}{\cos\frac{5\theta_0}{2}}>-\frac{\sin\theta}{\cos\theta}$,	also $\frac{\gamma_3}{d}=\frac{-\gamma_3}{-\gamma_3\cos\theta+\sin\theta}>1$.
	\end{itemize}
	Now, $f_2'(R)=2\cdot\frac{\frac{d}{\gamma_3}R-1}{bR-c}\cdot\frac{(\gamma_3b-cd)\gamma_3}{(bR-c)^2}$, so, thanks to the above results, $f_2(R)$ decreases for $1\le R < \frac{\gamma_3}{d}$ and increases for $R > \frac{\gamma_3}{d}$.  
	Moreover, like for $f_1(1)$, we can estimate
	$$f_2(1)=\left(\cfrac{-\cos\frac{\theta}{2}-\gamma_3\sin\frac{\theta}{2}}{-\gamma_3\sin\frac{5\theta}{2}+\gamma_3^2\cos\frac{5\theta}{2}}\right)^2\gamma_3^2\le\cfrac{1+\gamma_3^2}{\delta_0^2},
	$$
	and $\lim_{R\to+\infty}f_2(R)=\left(\frac{\gamma_3\cos\theta-\sin\theta}{\cos3\theta+\gamma_3\sin3\theta}\right)^2\le \frac{1+\gamma_3^2}{\cos^2 3\theta_0}$. 
	Therefore
	$$\frac{|\gamma_3\Re(\lambda-1)-\Im(\lambda-1)|}{\Re(\lambda^2-\lambda)+\gamma_3\Im(\lambda^2-\lambda)}\le\max\left(\frac{\sqrt{1+\gamma_3^2}}{\delta_0},\frac{\sqrt{1+\gamma_3^2}}{\cos3\theta_0}\right).$$
	
	\noindent (iv)  Since $\theta \in (\pi-\theta_0,\pi)\cup(-\pi,-\pi+\theta_0)$, we have
	\begin{itemize}
		\item $2\theta \in\left(2\pi-2\theta_0,2\pi\right)\cup\left(-2\pi,-2\pi+2\theta_0\right)\subseteq\left(2\pi-\frac{\pi}{3},2\pi\right)\cup\left(-2\pi,-2\pi+\frac{\pi}{3}\right)$ thus $\cos 2\theta>\cos2\theta_0>0$;
		\item $3\theta \in\left(3\pi-3\theta_0,3\pi\right)\cup\left(-3\pi,-3\pi+3\theta_0\right)\subseteq\left(3\pi-\frac{\pi}{2},3\pi\right)\cup\left(-3\pi,-3\pi+\frac{\pi}{2}\right)$, thus $-\cos 3\theta>-\cos(3\pi-3\theta_0)=\sin\left(\frac{\pi}{2}-3\theta_0\right)\ge 0$;
	\end{itemize}
	So we have
	$$-\Re(\lambda^3-\lambda^2)=R^2(-R\cos 3\theta+\cos2\theta)>\left[\sin\left(\frac{\pi}{2}-3\theta_0\right)+\cos2\theta_0\right]R^2>0.$$
	Finally, $\frac{|\Re(\lambda-1)|}{-\Re(\lambda^3-\lambda^2)}\le\frac{R+1}{\left[\sin\left(\frac{\pi}{2}-3\theta_0\right)+\cos2\theta_0\right]R^2}\le\frac{2}{\sin\left(\frac{\pi}{2}-3\theta_0\right)+\cos2\theta_0}$
	and similarly for $\frac{|\Im(\lambda-1)|}{-\Re(\lambda^3-\lambda^2)}$.
\end{proof}	

\begin{lemma}
	\label{gamma123 n+1}
	For $\lambda\in\C\setminus{\R}, |\lambda|\ge 1$ we write $\lambda=R(\cos\theta+\ic\sin\theta)$ in polar form where $R\ge 1$, $\theta\in(-\pi,\pi)$, $\theta \ne 0$.  
	\begin{enumerate}[label=(\roman*)]
		\item For $\lambda$ satisfying $\Re(\lambda^2-\lambda)\ge 0$, let $\gamma_1=\gamma_1(\lambda)=\left\{\begin{array}{cc}
			1, &\text{if } \Im (\lambda^2-\lambda)\ge 0,\\
			-1, &\text{if } \Im (\lambda^2-\lambda)<0
		\end{array}\right.$ then
		$$\Re(\lambda^2-\lambda)+\gamma_1\Im(\lambda^2-\lambda)\ge |\lambda(\lambda-1)|\ge 2|\sin(\theta/2)|.$$
		\item Let $0<\theta_0\le\frac{\pi}{4}$. For $\lambda$ satisfying $\Re(\lambda^2-\lambda)<0$ and $\theta\in[\theta_0,\pi-\theta_0]\cup[-\pi+\theta_0,-\theta_0]$,  let $\gamma_2=\gamma_2(\lambda)=\left\{\begin{array}{cc}
			-1, &\text{if } \Im (\lambda^2-\lambda)\ge 0,\\
			1, &\text{if } \Im (\lambda^2-\lambda)<0
		\end{array}\right.$ then 
		$$-\Re(\lambda^2-\lambda)-\gamma_2\Im(\lambda^2-\lambda) \ge |\lambda(\lambda-1)| \ge 2\sin({\theta_0}/{2}).$$
		\item Let $0<\theta_0\le\frac{\pi}{4}$ and $\delta_0>0$ . 
		For $\lambda$ satisfying $\Re(\lambda^2-\lambda)<0$ and $\theta\in(-\theta_0,\theta_0)\backslash\{0\}$, let $\gamma_3=\gamma_3(\mathrm{sign}(\theta))=\left\{\begin{array}{cc}
			\left(\delta_0+\sin\frac{3\theta_0}{2}\right)/\cos\frac{3\theta_0}{2} & \text{if }\theta>0,\\
			-\left(\delta_0+\sin\frac{3\theta_0}{2}\right)/\cos\frac{3\theta_0}{2} & \text{if }\theta<0
		\end{array}\right.$ then
		$$\Re(\lambda^2-\lambda)+\gamma_3\Im(\lambda^2-\lambda)\ge 2\delta_0|\sin(\theta/2)|.$$
		Moreover, if $0<\theta_0<\frac{\pi}{4}$ then 
		\begin{center}
			$\frac{|\Re(\lambda-1)+\gamma_3\Im(\lambda-1)|}{\Re(\lambda^2-\lambda)+\gamma_3\Im(\lambda^2-\lambda)}\le\frac{\sqrt{1+\gamma_3^2}}{\delta_0}$ and $\frac{|\gamma_3\Re(\lambda-1)-\Im(\lambda-1)|}{\Re(\lambda^2-\lambda)+\gamma_3\Im(\lambda^2-\lambda)}\le\max\left(\frac{\sqrt{1+\gamma_3^2}}{\delta_0},\frac{\sqrt{1+\gamma_3^2}}{\cos2\theta_0}\right)$.
		\end{center}
		\item Let $0<\theta_0\le\frac{\pi}{4}$. 
		There exists no $\lambda$ satisfying $\Re(\lambda^2-\lambda)<0$ and $\theta \in (\pi-\theta_0,\pi)\cup(-\pi,-\pi+\theta_0)$.
	\end{enumerate}
\end{lemma}
\begin{proof}
	The proofs for (i) and (ii) are similar to those in Lemma \ref{gamma123 n}. 
	
	\noindent (iii) Note that $\cos2\theta>0, -\frac{\pi}{2}<2\theta<\frac{\pi}{2}$, and $\sin 2\theta$ has the same sign as $\theta$ and $\gamma_3$, so we have 
	$$\begin{array}{ll}
		\Re(\lambda^2-\lambda)+\gamma_3\Im(\lambda^2-\lambda)&=R(R\cos 2\theta-\cos\theta+\gamma_3 R\sin 2\theta-\gamma_3\sin \theta)\\
		&\ge \cos 2\theta-\cos \theta+\gamma_3\sin 2\theta-\gamma_3\sin\theta\\
		&= -2\sin\frac{3\theta}{2}\sin\frac{\theta}{2}+2\gamma_3\cos\frac{3\theta}{2}\sin\frac{\theta}{2} \\
		&= 2\sin\frac{\theta}{2}\left(\gamma_3\cos\frac{3\theta}{2}-\sin\frac{3\theta}{2}\right).
	\end{array}
	$$
	Then we consider two cases: if $0<\theta<\theta_0$ then $\gamma_3>0$, $\left|\sin\frac{\theta}{2}\right|=\sin\frac{\theta}{2}>0$, 
	$0<\frac{3\theta}{2}<\frac{3\theta_0}{2}<\frac{\pi}{2}$ and $\gamma_3\cos\frac{3\theta}{2}-\sin\frac{3\theta}{2}>\gamma_3\cos\frac{3\theta_0}{2}-\sin\frac{3\theta_0}{2}=\delta_0$;	if $-\theta_0<\theta<0$ then $-\gamma_3>0$, $\left|\sin\frac{\theta}{2}\right|=-\sin\frac{\theta}{2}>0$, 
	$-\frac{\pi}{2}<-\frac{3\theta_0}{2}<\frac{3\theta}{2}<0$ and $-\gamma_3\cos\frac{3\theta}{2}+\sin\frac{3\theta}{2}>-\gamma_3\cos\frac{3\theta_0}{2}-\sin\frac{3\theta_0}{2}=\delta_0$.
	
	Next, if $0<\theta_0<\frac{\pi}{4}$, we will show that $\frac{|\Re(\lambda-1)+\gamma_3\Im(\lambda-1)|}{\Re(\lambda^2-\lambda)+\gamma_3\Im(\lambda^2-\lambda)}$ and $\frac{|\gamma_3\Re(\lambda-1)-\Im(\lambda-1)|}{\Re(\lambda^2-\lambda)+\gamma_3\Im(\lambda^2-\lambda)}$	are both bounded. First,
	$$\begin{array}{ll}
		\cfrac{|\Re(\lambda-1)+\gamma_3\Im(\lambda-1)|}{\Re(\lambda^2-\lambda)+\gamma_3\Im(\lambda^2-\lambda)}&=\cfrac{|(\cos\theta+\gamma_3\sin\theta)R-1|}{R[(\cos 2\theta+\gamma_3\sin 2\theta)R-(\cos\theta+\gamma_3\sin\theta)]}\\&\le \cfrac{|(\cos\theta+\gamma_3\sin\theta)R-1|}{(\cos 2\theta+\gamma_3\sin2\theta)R-(\cos\theta+\gamma_3\sin\theta)}.
	\end{array}
	$$
	Since $\gamma_3$ does not depend on $R$, let us study $f_1(R)=\left(\frac{aR-1}{bR-a}\right)^2$	where
	$a=\cos\theta+\gamma_3\sin\theta$, $b=\cos 2\theta+\gamma_3\sin 2\theta$. We observe that:
	\begin{itemize}
		\item $a>0$ and $b>0$. Indeed, $\cos\theta>0$, $\cos 2\theta>0$, and $\theta$ and $\gamma_3$ have the same sign.
		\item $bR-a>0$ since $\Re(\lambda^2-\lambda)+\gamma_3\Im(\lambda^2-\lambda)>0$, thus  $R>\frac{a}{b}$.
		\item $a^2>b$ (equivalently $\frac{a}{b}>\frac{1}{a}$), since $a^2=\cos^2\theta+\gamma_3^2\sin^2\theta+\gamma_3\sin 2\theta>\cos^2\theta-\sin^2\theta+\gamma_3\sin 2\theta=b$. 
	\end{itemize}
	Now, $f_1'(R)=2\cdot\frac{aR-1}{bR-a}\cdot\frac{b-a^2}{(bR-a)^2}<0$
	for $R>\frac{a}{b}>\frac{1}{a}$ and we would like to have $\frac{a}{b}<1$ so that $f_1(R)\le f_1(1), \forall R\ge 1$. Indeed $\frac{a}{b}< 1$ is equivalent to
	$$\cos\theta+\gamma_3\sin\theta < \cos2\theta+\gamma_3\sin2\theta  \Leftrightarrow|\gamma_3|>\frac{\left|\sin\frac{3\theta}{2}\right|}{\cos\frac{3\theta}{2}},$$
	which is true since
	$$|\gamma_3|=\frac{\delta_0+\sin\frac{3\theta_0}{2}}{\cos\frac{3\theta_0}{2}}>\frac{\left|\sin\frac{3\theta}{2}\right|}{\cos\frac{3\theta}{2}}+\varepsilon_0 \quad\mbox{where}\quad\varepsilon_0=\frac{\delta_0}{\cos\frac{3\theta_0}{2}}.$$
	Then we study
	$$f_1(1)=\left[ \frac{\cos\theta-1+\gamma_3\sin\theta}{\cos2\theta-\cos\theta+\gamma_3(\sin2\theta-\sin\theta)}\right]^2=\left(\cfrac{-\sin\frac{\theta}{2}+\gamma_3\cos\frac{\theta}{2}}{-\gamma_3\sin\frac{3\theta}{2}+\gamma_3^2\cos\frac{3\theta}{2}}\right)^2\gamma_3^2.
	$$
	We have:
	\begin{itemize}
		\item $(-\sin\frac{\theta}{2}+\gamma_3\cos\frac{\theta}{2})^2\le 1+\gamma_3^2$ by Cauchy-Schwarz inequality;
		\item $\gamma_3^2=|\gamma_3|^2>\frac{\gamma_3\sin\frac{3\theta}{2}}{\cos\frac{3\theta}{2}}+\varepsilon_0|\gamma_3|$ that leads to $-\gamma_3\sin\frac{3\theta}{2}+\gamma_3^2\cos\frac{3\theta}{2}>\varepsilon_0\cos\frac{3\theta}{2}|\gamma_3|=\delta_0|\gamma_3|$;
	\end{itemize}
	hence $f_1(1)\le\frac{1+\gamma_3^2}{\delta_0^2}$ and finally $\frac{|\Re(\lambda-1)+\gamma_3\Im(\lambda-1)|}{\Re(\lambda^2-\lambda)+\gamma_3\Im(\lambda^2-\lambda)}\le\frac{\sqrt{1+\gamma_3^2}}{\delta_0}$. Next, we have
	$$\begin{array}{ll}
		\cfrac{|\gamma_3\Re(\lambda-1)-\Im(\lambda-1)|}{\Re(\lambda^2-\lambda)+\gamma_3\Im(\lambda^2-\lambda)}&=\cfrac{|(\gamma_3\cos\theta-\sin\theta)R-\gamma_3|}{R[(\cos 2\theta+\gamma_3\sin 2\theta)R-(\cos\theta+\gamma_3\sin\theta)]}\\&\le \cfrac{|(\gamma_3\cos\theta-\sin\theta)R-\gamma_3|}{(\cos 2\theta+\gamma_3\sin2\theta)R-(\cos\theta+\gamma_3\sin\theta)}.
	\end{array}$$
	Since $\gamma_3$ does not depend on $R$, let us study $f_2(R)=\left(\frac{cR-\gamma_3}{bR-a}\right)^2$
	where $c=\gamma_3\cos\theta-\sin\theta$ and $a, b$ as above. We observe that:
	\begin{itemize}
		\item $\gamma_3b-ca$ and $\theta$ have the same sign. Indeed, $\gamma_3b-ca=(\gamma_3^2+1)\sin\theta\cos\theta$. Consequently, we always have $(\gamma_3b-ca)\gamma_3>0$.
		
		\item We always have $\frac{\gamma_3}{c}>1$. 	Indeed, if $\theta>0$ then $c>0$ since $\gamma_3=\frac{\delta_0+\sin\frac{3\theta_0}{2}}{\cos\frac{3\theta_0}{2}}>\frac{\sin\theta}{\cos\theta}$, also 
		$\frac{\gamma_3}{c}=\frac{\gamma_3}{\gamma_3\cos\theta-\sin\theta}>1$; if $\theta<0$ then $c<0$ since 
		$-\gamma_3=\frac{\delta_0+\sin\frac{3\theta_0}{2}}{\cos\frac{3\theta_0}{2}}>-\frac{\sin\theta}{\cos\theta}$,	also $\frac{\gamma_3}{c}=\frac{-\gamma_3}{-\gamma_3\cos\theta+\sin\theta}>1$.
	\end{itemize}
	Now, $f_2'(R)=2\cdot\frac{\frac{c}{\gamma_3}R-1}{bR-a}\cdot\frac{(\gamma_3b-ca)\gamma_3}{(bR-a)^2}$, so, thanks to the above results, $f_2(R)$ decreases for $1\le R < \frac{\gamma_3}{c}$ and increases for $R > \frac{\gamma_3}{c}$.  
	Moreover, like for $f_1(1)$, we can estimate
	$$f_2(1)=\left(\cfrac{-\cos\frac{\theta}{2}-\gamma_3\sin\frac{\theta}{2}}{-\gamma_3\sin\frac{3\theta}{2}+\gamma_3^2\cos\frac{3\theta}{2}}\right)^2\gamma_3^2\le\cfrac{1+\gamma_3^2}{\delta_0^2}
	$$
	and $\lim_{R\to+\infty}f_2(R)=\left(\frac{\gamma_3\cos\theta-\sin\theta}{\cos2\theta+\gamma_3\sin2\theta}\right)^2\le \frac{1+\gamma_3^2}{\cos2\theta_0}$. Therefore
	$$\frac{|\gamma_3\Re(\lambda-1)-\Im(\lambda-1)|}{\Re(\lambda^2-\lambda)+\gamma_3\Im(\lambda^2-\lambda)}\le\max\left(\frac{\sqrt{1+\gamma_3^2}}{\delta_0},\frac{\sqrt{1+\gamma_3^2}}{\cos2\theta_0}\right).$$
	
	\noindent (iv)  For $\theta \in (\pi-\theta_0,\pi)\cup(-\pi,-\pi+\theta_0)$, we have $\cos 2\theta>0$ since $2\theta \in\left(\frac{3\pi}{2},2\pi\right)\cup\left(-2\pi,-\frac{3\pi}{2}\right)$, while $\cos \theta<0$. 
	Hence $\Re(\lambda^2-\lambda)=R(R\cos 2\theta-\cos\theta)>0$.
\end{proof}

\section{Descent step for usual and shifted gradient descent}
\label{app:convGD}

\begin{proposition}[Descent step for the usual gradient descent]\label{tau usualgd}
	The usual gradient descent algorithm \eqref{usualgd} converges if
	$$0<\tau<\cfrac{2}{\norm{H(I-B)^{-1}M}^2}.$$
\end{proposition}
\begin{proof}
	The error system for \eqref{usualgd} can be rewritten as
	\begin{equation}
		\label{itermat usualgd}
		\begin{bmatrix}
			p^{n+1}\\ u^{n+1}\\ \sigma^{n+1}
		\end{bmatrix}=\begin{bmatrix}
			-\tau (I-B^*)^{-1}H^*H(I-B)^{-1}MM^* & 0 & (I-B^*)^{-1}H^*H(I-B)^{-1}M \\
			-\tau (I-B)^{-1}MM^* & 0 & (I-B)^{-1}M \\
			-\tau M^* & 0 & I
		\end{bmatrix}
		\begin{bmatrix}
			p^{n}\\ u^{n}\\ \sigma^{n}
		\end{bmatrix}
	\end{equation}
	Recall that a fixed point iteration converges if and only if the spectral radius of its iteration matrix is strictly less than $1$. We can show that:
	\begin{enumerate}[label=(\roman*)]
		\item If $\lambda\in\C \setminus \{0,1\}$ is an eigenvalue of the iteration matrix, then, proceeding as in Proposition \ref{prop:eq-eigen k-shot}, there exists $y\in\C^{n_\sigma}, y\neq 0$ such that
		\begin{equation}\label{eq-eigen usualgd}
			\lambda^2(\lambda-1)+\tau\frac{\norm{H(I-B)^{-1}My}^2}{\norm{y}^2}\lambda^2=0
		\end{equation}
		hence $\lambda=1-\tau\frac{\norm{H(I-B)^{-1}My}^2}{\norm{y}^2}$. If we take $\tau<\frac{2}{\norm{H(I-B)^{-1}M}^2}$ then equation \eqref{eq-eigen usualgd} admits no solution $\lambda$ with $|\lambda|\ge 1$.
		\item $\lambda=1$ is not an eigenvalue of the iteration matrix. To show this, we rewrite iteration \eqref{itermat usualgd} as 
		$$\begin{bmatrix}
			\sigma^{n+1}\\ p^{n+1}\\ u^{n+1}
		\end{bmatrix}=\begin{bmatrix}
			I  & -\tau M^* & 0 \\
			(I-B^*)^{-1}H^*H(I-B)^{-1}M & -\tau (I-B^*)^{-1}H^*H(I-B)^{-1}MM^* & 0 \\
			(I-B)^{-1}M & -\tau (I-B)^{-1}MM^* & 0
		\end{bmatrix}
		\begin{bmatrix}
			\sigma^{n}\\ p^{n}\\ u^{n}
		\end{bmatrix}.$$
	\end{enumerate}
\end{proof}

\begin{proposition}[Convergence of the shifted gradient descent]\label{tau shiftgd}
	The shifted gradient descent algorithm \eqref{shifted-gd}
	converges if
	$$0<\tau<\cfrac{1}{\norm{H(I-B)^{-1}M}^2}.$$
\end{proposition}
\begin{proof}
	The error system for \eqref{shifted-gd} can be rewritten as
	\begin{equation}
		\label{itermat shiftgd}
		\begin{bmatrix}
			p^{n+1}\\ u^{n+1}\\ \sigma^{n+1}
		\end{bmatrix}=\begin{bmatrix}
			0 & 0 & (I-B^*)^{-1}H^*H(I-B)^{-1}M \\
			0 & 0 & (I-B)^{-1}M \\
			-\tau M^* & 0 & I
		\end{bmatrix}
		\begin{bmatrix}
			p^{n}\\ u^{n}\\ \sigma^{n}
		\end{bmatrix}.
	\end{equation}
	Recall that a fixed point iteration converges if and only if the spectral radius of its iteration matrix is strictly less than $1$. We can show that:
	\begin{enumerate}[label=(\roman*)]
		\item If $\lambda\in\C \setminus \{0,1\}$ is an eigenvalue of the iteration matrix, then, proceeding as in Proposition \ref{prop:eq-eigen shift-k-shot}, there exists $y\in\C^{n_\sigma}, y\neq 0$ such that
		\begin{equation}\label{eq-eigen shiftgd}
			\lambda^2(\lambda-1)+\tau\frac{\norm{H(I-B)^{-1}My}^2}{\norm{y}^2}\lambda=0.
		\end{equation}
		By applying Lemma~\ref{marden-ord-3} for
		\[
		a_0=0, \quad a_1=\tau\frac{\norm{H(I-B)^{-1}My}^2}{\norm{y}^2},\quad a_2=-1,
		\]
		we see that equation \eqref{eq-eigen shiftgd} admits no solution $\lambda$ with $|\lambda|\ge 1$ if we take $\tau<\frac{\norm{y}^2}{\norm{H(I-B)^{-1}My}^2}$. Then it is enough to take  $\tau<\frac{1}{\norm{H(I-B)^{-1}M}^2}$.
		\item $\lambda=1$ is not an eigenvalue of the iteration matrix. To show this, we rewrite iteration \eqref{itermat shiftgd} as 
		$$\begin{bmatrix}
			\sigma^{n+1}\\ p^{n+1}\\ u^{n+1}
		\end{bmatrix}=\begin{bmatrix}
			I  & -\tau M^* & 0 \\
			(I-B^*)^{-1}H^*H(I-B)^{-1}M & 0 & 0 \\
			(I-B)^{-1}M & 0 & 0
		\end{bmatrix}
		\begin{bmatrix}
			\sigma^{n}\\ p^{n}\\ u^{n}
		\end{bmatrix}.$$
		and proceed as in Proposition \ref{prop:eq-eigen shift-k-shot}.
	\end{enumerate}
\end{proof}

\section{Convergence study for the scalar case}
\label{app:1D}

\subsection{Notations and preliminary calculation}
In the scalar case, that is when $n_u, n_\sigma, n_f =1$, we change the notation from capital to lower case letters: 
$$B\leftarrow b\in\R, b<1,\quad M\leftarrow m\in\R, m\neq 0, \quad H\leftarrow h\in\R, h\neq 0,$$
\begin{equation}
	\label{t-u-1d}
	T_k\leftarrow t_k=1+b+...+b^{k-1}=\cfrac{1-b^k}{1-b}, \quad U_k\leftarrow u_k=kh^2b^{k-1}
\end{equation}
$$
X_k\leftarrow x_k=\left\{\begin{array}{cl}
	0, & k=1,\\
	h^2[1+2b+3b^2+...+(k-1)b^{k-2}], & k\ge 2. 
\end{array}\right.
$$
The identity $
1+2x+3x^2+...+nx^{n-1}=\left(\frac{1-x^{n+1}}{1-x}\right)'=\frac{1-(n+1)x^n+nx^{n+1}}{(1-x)^2}$ says that 
\begin{equation}
	\label{x-1d}
	x_k=h^2\frac{1-kb^{k-1}+(k-1)b^k}{(1-b)^2}, \quad k\ge 1,
\end{equation}
where we set $b^{k-1}=1$ when $k=1$ and $b=0$. Now for each of algorithms \eqref{usualgd}, \eqref{shifted-gd}, \eqref{alg:k-shot n}, \eqref{alg:k-shot n+1}, we write the iterations for the errors in the scalar case and the corresponding iteration matrix $\mathcal{M}$ such that
$[p^{n+1}, u^{n+1}, \sigma^{n+1}]^\intercal=\mathcal{M}
[p^{n}, u^{n}, \sigma^{n}]^\intercal$.
\begin{itemize}
	\item Usual gradient descent (usual GD):
	\begin{equation}
		\label{usualgd-err-1d}
		\begin{cases}
			\sigma^{n+1}=\sigma^n-\tau mp^n\\
			u^{n}=bu^n+m\sigma^n\\
			p^{n}=bp^n+h^2u^n
		\end{cases}\quad
		\mathcal{M}=\begin{bmatrix}
			-h^2m^2(1-b)^{-2}\tau & 0 & h^2m(1-b)^{-2} \\
			-m^2(1-b)^{-1}\tau & 0 & m(1-b)^{-1} \\
			-m\tau & 0 & 1
		\end{bmatrix}
	\end{equation}
	\item Shifted gradient descent (shifted GD):
	\begin{equation}
		\label{shifted-gd-err-1d}
		\begin{cases}
			\sigma^{n+1}=\sigma^n-\tau mp^n\\
			u^{n+1}=bu^{n+1}+m\sigma^{n}\\
			p^{n+1}=bp^{n+1}+h^2u^{n+1}\\
		\end{cases}\quad
		\mathcal{M}=\begin{bmatrix}
			0 & 0 & h^2m(1-b)^{-2} \\
			0 & 0 & m(1-b)^{-1} \\
			-m\tau & 0 & 1
		\end{bmatrix}
	\end{equation}
	\item $k$-step one-shot:
	\begin{equation}
		\label{k-shot-err-1d n+1}
		\begin{cases}
			\sigma^{n+1}=\sigma^n-\tau mp^n\\
			p^{n+1}=(b^k-\tau m^2x_k)p^n+u_ku^n+mx_k\sigma^n\\
			u^{n+1}=b^ku^n+mt_k\sigma^n-\tau m^2t_kp^n
		\end{cases} \quad
		\mathcal{M}=\begin{bmatrix}
			b^k-m^2x_k\tau & u_k & mx_k \\
			-m^2t_k\tau  & b^k & mt_k \\
			-m\tau  & 0 & 1
		\end{bmatrix}
	\end{equation}
	\item Shifted $k$-step one-shot:
	\begin{equation}
		\label{k-shot-err-1d n}
		\begin{cases}
			\sigma^{n+1}=\sigma^n-\tau mp^n\\
			p^{n+1}=b^kp^n+u_ku^n+mx_k\sigma^n\\
			u^{n+1}=b^ku^n+mt_k\sigma^n
		\end{cases} \quad	
		\mathcal{M}=\begin{bmatrix}
			b^k & u_k & mx_k \\
			0 & b^k & mt_k \\
			-m\tau & 0 & 1
		\end{bmatrix}.
	\end{equation}
\end{itemize}

\subsection{Necessary and sufficient conditions for convergence}

In this simpler scalar case, we will be able to prove sufficient and also necessary conditions on the descent step $\tau$ for convergence. Our strategy to study the spectral radius $\rho(\mathcal{M})$ is as follows:
\begin{enumerate}
	\item Compute $\det(\mathcal{M}-\lambda I)$ to write the eigenvalue equation $\mathcal{P}(\lambda)=0$. For the considered methods, $\mathcal{P}$ turns out to be a polynomial of degree $3$, 
	$
	\mathcal{P}(\lambda)=a_0+a_1\lambda+a_2\lambda^2+\lambda^3
	$, 
	where $a_0,a_1,a_2\in\R$ depend on $h,m,b,\tau$. For the computations, the identity $u_kt_k- b^k x_k+x_k=h^2t_k^2$, which is the scalar version of \eqref{uxt}, can be helpful.
	\item Apply to $\mathcal{P}$ Lemma~\ref{marden-ord-3}, which states a necessary and sufficient condition for a real coefficient polynomial of degree 3 to have all roots inside the unit circle of the complex plane. Then deduce conditions on $\tau$. 
\end{enumerate}
\begin{lemma}\label{marden-ord-3}
	Let $a_0,a_1,a_2\in\R$, then all roots of $\mathcal{P}(z)=a_0+a_1z+a_2z^2+z^3$ stay (strictly) inside the unit circle of the complex plane if and only if
	\begin{gather}
		(a_0-1)(a_0+1)<0, \label{cond1} \\
		(a_0^2-a_2a_0+a_1-1)(a_0^2+a_2a_0-a_1-1)>0, \label{cond2}  \\
		(a_0+a_2-a_1-1)(a_0+a_2+a_1+1)<0. \label{cond3} 
	\end{gather}
\end{lemma}
\noindent The proof of Lemma \ref{marden-ord-3} is in Appendix~\ref{app:marden} and is mainly based on Marden's works \cite{marden66}.

\subsubsection{Descent step for the usual gradient descent}
Here, the coefficients of $\mathcal{P}$ are 
$$a_0=0, \quad a_1=0,\quad  a_2=h^2m^2(1-b)^{-2}\tau-1.$$
Conditions \eqref{cond1} and \eqref{cond2} of Lemma \ref{marden-ord-3} are automatically satisfied. Condition \eqref{cond3} gives
$$0<\tau<\frac{2(1-b)^2}{h^2m^2},$$
that is \eqref{best-tau-usualgd} in the scalar case.

\subsubsection{Descent step for the shifted gradient descent}
Here, the coefficients of $\mathcal{P}$ are 
$$a_0=0, \quad a_1=h^2m^2(1-b)^{-2}\tau,\quad  a_2=-1.$$
Condition \eqref{cond1} of Lemma \ref{marden-ord-3} is automatically satisfied, condition \eqref{cond3} is automatically satisfied for $\tau>0$, and condition \eqref{cond2} gives us 
\[
\tau<\frac{(1-b)^2}{h^2m^2}, 
\]
that is \eqref{best-tau-shifted-gd} in the scalar case.

\subsubsection{Descent step for $k$-step one-shot} 

Here, the coefficients of $\mathcal{P}$ are 
$$
a_0=-s^2,\quad a_1=m^2(h^2t_k^2-x_k)\tau+(s^2+2s),\quad a_2=m^2x_k\tau-(2s+1)
$$
where $s=b^k$. Condition \eqref{cond1} of Lemma \ref{marden-ord-3} is obviously satisfied since $|b|<1$. Next we deal with condition \eqref{cond2}. The computation shows that
\begin{equation}
	\label{cond2-term1 n+1}
	a_0^2-a_2a_0+a_1-1=m^2(h^2t_k^2-x_k+x_ks^2)\tau+\underbrace{(s-1)^3(s+1)}_{<0},
\end{equation}
\begin{equation}
	\label{cond2-term2 n+1}
	a_0^2+a_2a_0-a_1-1=-m^2(h^2t_k^2-x_k+x_ks^2)\tau+\underbrace{(s-1)(s+1)^3}_{<0}
\end{equation}
and
\begin{equation}
	\label{cond2-core n+1}
	h^2t_k^2-x_k+x_ks^2=\frac{h^2b^{k-1}(1-b^k)[k-(k+1)b+kb^{k}-(k-1)b^{k+1}]}{(1-b)^2}.
\end{equation}

\begin{lemma}
	$k-(k+1)b+kb^{k}-(k-1)b^{k+1}>0$, $\forall |b|<1$, $\forall k\ge 1$.
\end{lemma}
\begin{proof}
	We write $k-(k+1)b+kb^{k}-(k-1)b^{k+1}=(1-b)A$ where $A=k+1-\frac{1-b^k}{1-b}+(k-1)b^k$. It suffices to show $A>0$. If $k=1$ then $A=1>0$. If either $k$ is even, or $k\ge 3$ is odd and $0\le b<1$, then $(k-1)b^k\ge 0$ and
	$\bigl|\frac{1-b^k}{1-b}\bigr|=|b^{k-1}+b^{k-2}+...+b+1|\le |b^{k-1}|+|b^{k-2}|+...+|b|+1<k$
	give us the conclusion.	If $k\ge 3$ is odd and $-1<b<0$ then $(k-1)(1+b^k+1)>0$ and $\frac{1-b^k}{1-b}<1$ therefore $A=1+\left(1-\frac{1-b^k}{1-b}\right)+(k-1)(1+b^k)>0$. 	
\end{proof}
\noindent Then, condition \eqref{cond2} imposes 
\begin{itemize}
	\item $\tau<\frac{(1-b)^2(1+b^k)(1-b^k)^2}{h^2m^2b^{k-1}[k-(k+1)b+kb^k-(k-1)b^{k+1}]}$ if $b^{k-1}>0$;
	\item $\tau<\frac{(1-b)^2(1+b^k)^3}{h^2m^2b^{k-1}[-k+(k+1)b-kb^k+(k-1)b^{k+1}]}$ if $b^{k-1}<0$;
	\item no condition on $\tau$ if $k\ge 2$ and $b=0$.
\end{itemize}
Finally we check condition \eqref{cond3}. We have $a_0+a_2+a_1+1= h^2m^2t_k^2\tau>0$ and
$$
a_0+a_2-a_1-1=\frac{h^2m^2(1-2kb^{k-1}+2kb^k-b^{2k})}{(1-b)^2}\tau-2(1+s)^2,
$$
therefore, condition \eqref{cond3} gives
\begin{itemize}
	\item $\tau<\frac{2(1-b)^2(1+b^k)^2}{h^2m^2(1-2kb^{k-1}+2kb^k-b^{2k})}$ if $1-2kb^{k-1}+2kb^k-b^{2k}>0$;
	\item no condition on $\tau$ if $1-2kb^{k-1}+2kb^k-b^{2k}\le 0$.
\end{itemize}
In the following lemma we study the quantity $1-2kb^{k-1}+2kb^k-b^{2k}$ that appears above. 
\begin{lemma}\label{corepoly-cond3}
	Let $f_k(b)=1-2kb^{k-1}+2kb^k-b^{2k}$ for $k\in\N^*$ and $-1\le b\le 1$. 
	\begin{enumerate}[label=(\roman*)]
		\item $f_1(b)=-(1-b)^2<0, \forall -1<b<1$.
		\item $f_2(b)=1-4b+4b^2-b^4$ has a unique solution $b=-1+\sqrt{2}$ in $(-1,1)$; and $f_2(b)>0$ if $-1<b<-1+\sqrt{2}$, $f_2(b)<0$ if $-1+\sqrt{2}<b<1$.
		\item If $k\ge 3$ is odd then $f_k(b)$ has exactly two solutions $b_1(k)<b_2(k)$ in $(-1,1)$; if $k\ge 2$ is even then $f_k(b)$ has a unique solution $b_3(k)$ in $(-1,1)$. Moreover, for every odd $k\ge 3$:
		\begin{itemize}
			\item $-1<b_1(k)<0<b_2(k)<1$;
			\item $f_k(b)>0\Leftrightarrow b_1(k)<b<b_2(k)$;
			\item $f_k(b)<0\Leftrightarrow -1<b<b_1(k)\vee b_2(k)<b<1$.
		\end{itemize}
		and for every even $k\ge 2$:
		\begin{itemize}
			\item $0<b_3(k)<1$;
			\item $f_k(b)>0\Leftrightarrow -1<b<b_3(k)$;
			\item $f_k(b)<0\Leftrightarrow b_3(k)<b<1$.
		\end{itemize}
		\item $\lim\limits_{\substack{k\text{ odd}\\ k\to\infty}}b_1(k)=-1$ and $\lim\limits_{\substack{k\text{ odd}\\ k\to\infty}}b_2(k)=1=\lim\limits_{\substack{k\text{ even}\\ k\to\infty}}b_3(k)=1$.
	\end{enumerate}	
\end{lemma}
\begin{proof}
	(i) and (ii) are easy to verify. 
	\noindent(iii) It remains to consider $k\ge 3$. We have
	$$f'_k(b)=b^{k-2}\left[-2k(k-1)+2k^2b-2kb^{k+1}\right], -1<b<1.$$
	Set
	$$g_k(b)=-2k(k-1)+2k^2b-2kb^{k+1}, \quad -1\le b\le 1, k\ge 3.$$
	
	\noindent\textit{Case 1.} [$k\ge 3$ is odd] 
	By studying the sign of $g'_k(b)$, we find that 
	\begin{itemize}
		\item $g_k$ has a unique solution $v_1(k)$ in $(-1,1)$ and $0<v_1(k)<\sqrt[k]{\frac{k}{k+1}}<1$;
		\item $g_k(b)>0\Leftrightarrow v_1(k)<b<1$;
		\item $g_k(b)<0\Leftrightarrow -1<b<v_1(k)$.
	\end{itemize}
	\noindent Next, 
	by studying the sign of $f'_k(b)$, we find that 
	\begin{itemize}
		\item $f_k(b)$ has exactly two solutions $b_1(k)<b_2(k)$ in $(-1,1)$ and $-1<b_1(k)<0<b_2(k)<1$;
		\item $f_k(b)>0\Leftrightarrow b_1(k)<b<b_2(k)$;
		\item $f_k(b)<0\Leftrightarrow -1<b<b_1(k)\vee b_2(k)<b<1$.
	\end{itemize}
	\noindent\textit{Case 2.} [$k\ge 4$ is even] 
	By studying the sign of $g'_k(b)$, we find that 
	\begin{itemize}
		\item $g_k$ has a unique solution $v_2(k)$ in $(-1,1)$ and $0<v_2(k)<\sqrt[k]{\frac{k}{k+1}}<1$; 
		\item $g_k(b)>0\Leftrightarrow v_2(k)<b<1$;
		\item $g_k(b)<0 \Leftrightarrow 0<b<v_2(k)$.
	\end{itemize}
	\noindent Next, 
	by studying the sign of $f'_k(b)$, we find that 
	\begin{itemize}
		\item $f_k(b)$ has a unique solution $b_3(k)$ in $(-1,1)$ and $0<b_3(k)<1$;
		\item $f_k(b)>0\Leftrightarrow -1<b<b_3(k)$;
		\item $f_k(b)<0\Leftrightarrow b_3(k)<b<1$.
	\end{itemize}
	
	\noindent (iv) We have
	$$f_k\left(\frac{1}{2}\right)=1-\frac{k}{2^{k-1}}-\frac{1}{2k}, \quad\forall k\ge 3\quad\mbox{ and }\quad f_k\left(-\frac{1}{2}\right)=1-\cfrac{3k}{2^{k-1}}-\cfrac{1}{2k}, \quad\forall \text{ odd } k\ge 3,$$
	hence for sufficiently large $k$ we have $f_k\left(\frac{1}{2}\right)>0$ and for sufficiently large odd $k$ we have $f_k\left(-\frac{1}{2}\right)>0$. By the table of signs of $f_k$, we conclude that $b_1(k)<-\frac{1}{2}$ for large odd $k$,  $b_2(k)>\frac{1}{2}$ for large odd $k$ and $b_3(k)>\frac{1}{2}$ for large even $k$.
	
	\noindent\textit{Case 1.} [$k\ge 3$ is odd and sufficiently large] First we work with $b_1(k)$. We have 
	$$1-2kb_1(k)^{k-1}+2kb_1(k)^k-b_1(k)^{2k}=0$$
	and $b_1(k)<-\frac{1}{2}$ so
	$$-b_1(k)^{2k}+2kb_1(k)^k+1=2kb_1(k)^{k-1}=\underbrace{[-2kb_1(k)^k]}_{>0}\cdot\frac{1}{-b_1(k)}<[-2kb_1(k)^k].2=-4kb_1(k)^k,$$
	which leads to
	$$b_1(k)^{2k}-6kb_1(k)^k-1>0\Leftrightarrow[b_1(k)^k-3k]^2>1+9k^2.$$
	Since $-1<b_1(k)<0$ and $k$ is odd, this tells us that
	$$-1<b_1(k)<-(-3k+\sqrt{1+9k^2})^{1/k}=\cfrac{-1}{(3k+\sqrt{1+9k^2})^{1/k}}<\cfrac{-1}{(7k)^{1/k}},$$
	which yields $\lim\limits_{\substack{k\text{ odd}\\ k\to\infty}}b_1(k)=-1$. Next, we have
	$$1-2kb_2(k)^{k-1}+2kb_2(k)^k-b_2(k)^{2k}=0$$ 
	and $b_2(k)>\frac{1}{2}$ so
	$$-b_2(k)^{2k}+2kb_2(k)^k+1=2kb_2(k)^{k-1}=2kb_2(k)^k\cdot\frac{1}{b_2(k)}<4kb_2(k)^k,$$
	which leads to
	$$b_2(k)^{2k}+2kb_2(k)^k-1>0\Leftrightarrow[b_2(k)^k+k]^2>1+k^2.$$
	Since $0<b_2(k)<1$, this tells us that
	$$1>b_2(k)>(-k+\sqrt{1+k^2})^{1/k}=\cfrac{1}{(k+\sqrt{1+k^2})^{1/k}}>\cfrac{1}{(3k)^{1/k}},$$
	which yields $\lim\limits_{\substack{k\text{ even}\\ k\to\infty}}b_2(k)=1$.
	
	\noindent\textit{Case 2.} [$k\ge 4$ is even and sufficiently large] We repeat the same arguments as $b_2(k)$ for $b_3(k)$.
\end{proof}

\noindent In summary, we have the following proposition.

\begin{proposition}[Convergence of $k$-step one-shot]\label{prop:tau-k-one-shot}
	Let $\eta_1(k,b) \coloneqq +\infty$ and
	$$\eta_{21}(k,b) \coloneqq \frac{(1-b)^2(1+b^k)(1-b^k)^2}{b^{k-1}[k-(k+1)b+kb^k-(k-1)b^{k+1}]};$$
	$$\eta_{22}(k,b) \coloneqq \frac{-(1-b)^2(1+b^k)^3}{b^{k-1}[k-(k+1)b+kb^k-(k-1)b^{k+1}]};$$
	$$\eta_3(k,b) \coloneqq \frac{2(1-b)^2(1+b^k)^2}{1-2kb^{k-1}+2kb^k-b^{2k}}$$
	then the necessary and sufficient condition for the convergence of $k$-step one-shot in the scalar case is of the form
	$\tau<\frac{\eta(k,b)}{h^2m^2}$ where $\eta(k,b)$ is defined as follows: 
	\begin{enumerate}[label=(\roman*)]
		\item $\eta(1,b)=\eta_{21}(1,b)=(1-b)^3(1+b),-1<b<1$;
		\item for odd $k\ge 3$,
		$$\eta(k,b)=
		\left\{\begin{array}{cl}
			\eta_{21}(k,b), & -1<b\le b_1(k)\vee b_2(k)\le b<1,\\
			\min\left\{\eta_{21}(k,b),\eta_3(k,b)\right\}, & b_1(k)<b<b_2(k)\wedge b\neq 0,\\
			2, & b=0
		\end{array}\right.$$
		where $-1<b_1(k)<0<b_2(k)<1$ are the two solutions of $$1-2kb^{k-1}+2kb^k-b^{2k}=0, \quad -1<b<1;$$
		\item for even $k\ge 2$,
		$$\eta(k,b)=
		\left\{\begin{array}{cl}
			\eta_{21}(k,b), & b_3(k)\le b<1,\\
			\min\left\{\eta_{21}(k,b),\eta_3(k,b)\right\}, & 0<b<b_3(k),\\
			2, & b=0,\\
			\min\left\{\eta_{22}(k,b),\eta_3(k,b)\right\}, & -1<b<0
		\end{array}\right.$$
		where $0<b_3(k)<1$ is the unique solution of $$1-2kb^{k-1}+2kb^k-b^{2k}=0, \quad -1<b<1.$$
	\end{enumerate}
\end{proposition}

\noindent Note that $\lim\limits_{\substack{k\text{ odd}\\ k\to\infty}}b_1(k)=-1$ and $\lim\limits_{\substack{k\text{ odd}\\ k\to\infty}}b_2(k)=1=\lim\limits_{\substack{k\text{ even}\\ k\to\infty}}b_3(k)$, so the behavior of $\tau$ when $k\rightarrow\infty$ is consistent with the result $\tau<\frac{2(1-b)^2}{h^2m^2},-1<b<1$ for the usual gradient descent. 
For illustrations of the function $\eta(k,b)$ for different $k$ see section~\ref{ssec:plots}. 

\subsubsection{Descent step for shifted $k$-step one-shot}

Here, the coefficients of the polynomial $\mathcal{P}$ of the eigenvalue equation are 
\begin{equation}
	\label{a012 n}
	a_0=h^2m^2v_k\tau-s^2,\quad a_1=h^2m^2y_k\tau+s^2+2s,\quad a_2=-2s-1,
\end{equation}
where $s=b^k$, $y_k=\frac{x_k}{h^2}=\frac{1-kb^{k-1}+(k-1)b^k}{(1-b)^2}$ and $v_k=t_k^2-y_k=\frac{b^{k-1}[k-(k+1)b+b^{k+1}]}{(1-b)^2}$. Note that $v_k$ and $b^{k-1}$ have the same sign, also $v_k=0$ if and only if $k\ge 2$ and $b=0$, since it is easy to show that $k-(k+1)b+b^{k+1}>0$, $\forall |b|<1,\forall k\ge 1$. Then, condition \eqref{cond1} of Lemma \ref{marden-ord-3} imposes
\begin{itemize}
	\item $\tau<\frac{1+s^2}{h^2m^2v_k}=\frac{(1-b)^2(1+b^{2k})}{h^2m^2b^{k-1}[k-(k+1)b+b^{k+1}]}$ if $b^{k-1}>0$;
	\item $\tau<\frac{-1+s^2}{h^2m^2v_k}=\frac{(1-b)^2(-1+b^{2k})}{h^2m^2b^{k-1}[k-(k+1)b+b^{k+1}]}$ if $b^{k-1}<0$;
	\item no condition on $\tau$ if $k\ge 2$ and $b=0$.
\end{itemize}
\noindent Next we study condition \eqref{cond2}. We have
$$a_0^2-a_2a_0+a_1-1=v_k^2(h^2m^2\tau)^2+[(-2s^2+2s+1)v_k+y_k]h^2m^2\tau+\underbrace{(s-1)^3(s+1)}_{<0}$$
and
$$a_0^2+a_2a_0-a_1-1=v_k^2(h^2m^2\tau)^2-[(2s^2+2s+1)v_k+y_k]h^2m^2\tau+\underbrace{(s-1)(s+1)^3}_{<0},$$
each of which, considered as a second order polynomial of $h^2m^2\tau$ if $v_k\neq 0$, has exactly two roots of opposite signs. Therefore if $v_k\neq 0$, condition \eqref{cond2} is equivalent to $(h^2m^2\tau-r_1)(h^2m^2\tau-r_2)>0$ where
$$r_1 \coloneqq \cfrac{(2s^2-2s-1)v_k-y_k+\sqrt{(-4s+5)v_k^2+y_k^2+2(-2s^2+2s+1)v_ky_k}}{2v_k^2}>0$$
and
$$r_2 \coloneqq \frac{(2s^2+2s+1)v_k+y_k+\sqrt{(8s^2+12s+5)v_k^2+y_k^2+2(2s^2+2s+1)v_ky_k}}{2v_k^2}>0.$$
\begin{lemma}
	$r_1$ and $r_2$ cannot be both strictly less than $\frac{1+s^2}{v_k}$. $r_1$ and $r_2$ cannot be both strictly less than $\frac{-1+s^2}{v_k}$.
\end{lemma}
\begin{proof}
	Either $r_1<\frac{1+s^2}{v_k}$ or $r_1<\frac{-1+s^2}{v_k}$ implies $(s^2+4s+1)v_k^2+(s^2+1)v_ky_k>0$. Either $r_2<\frac{1+s^2}{v_k}$ or $r_2<\frac{-1+s^2}{v_k}$ implies $(s^2+4s+1)v_k^2+(s^2+1)v_ky_k<0$.
\end{proof}
\noindent Thanks to this lemma we see that condition \eqref{cond2}, in combination with condition \eqref{cond1}, gives
\begin{itemize}
	\item $\tau<\frac{1}{h^2m^2}\min\{r_1,r_2\}$ if $b^{k-1}\neq 0$;
	\item $\tau<\frac{1}{h^2m^2}$ if $k\ge 2$ and $b=0$.
\end{itemize}
Finally, we have $a_0+a_2+a_1+1=h^2m^2t_k^2\tau>0$ and
$$a_0+a_2-a_1-1=\frac{h^2m^2}{(1-b)^2}[-1+2kb^{k-1}-2kb^k+b^{2k}]\tau-2(1-b^k)^2,$$
thus condition \eqref{cond3} is equivalent to
\begin{itemize}
	\item $\tau<\frac{2(1-b)^2(1-b^k)^2}{h^2m^2(-1+2kb^{k-1}-2kb^k+b^{2k})}$ if $1-2kb^{k-1}+2kb^k-b^{2k}<0$;
	\item no condition on $\tau$ if $1-2kb^{k-1}+2kb^k-b^{2k}\ge 0$.
\end{itemize}
One can look again at Lemma \ref{corepoly-cond3} for the analysis of $1-2kb^{k-1}+2kb^k-b^{2k}$. In summary, we have the following proposition.

\begin{proposition}[Convergence of shifted $k$-step one-shot]\label{prop:tau-shifted-k-one-shot}
	Let
	$$\kappa_{11}(k,b) \coloneqq \cfrac{(1-b)^2(1+b^{2k})}{b^{k-1}[k-(k+1)b+b^{k+1}]};$$
	$$\kappa_{12}(k,b) \coloneqq \cfrac{(1-b)^2(-1+b^{2k})}{b^{k-1}[k-(k+1)b+b^{k+1}]};$$
	$$t_k \coloneqq \cfrac{1-b^k}{1-b},\quad y_k \coloneqq \frac{1-kb^{k-1}+(k-1)b^k}{(1-b)^2},\quad s \coloneqq b^k,\quad v_k \coloneqq t_k^2-y_k,$$
	$$
	\kappa_{21}(k,b) \coloneqq \cfrac{(2s^2-2s-1)v_k-y_k+\sqrt{(-4s+5)v_k^2+y_k^2+2(-2s^2+2s+1)v_ky_k}}{2v_k^2},$$
	$$\kappa_{22}(k,b) \coloneqq \frac{(2s^2+2s+1)v_k+y_k+\sqrt{(8s^2+12s+5)v_k^2+y_k^2+2(2s^2+2s+1)v_ky_k}}{2v_k^2},$$
	$$\kappa_2(k,b) \coloneqq \min\{\kappa_{21}(k,b),\kappa_{22}(k,b)\};$$
	$$\kappa_3(k,b) \coloneqq \frac{2(1-b)^2(1-b^k)^2}{-1+2kb^{k-1}-2kb^k+b^{2k}}$$
	then the necessary and sufficient condition for the convergence of shifted $k$-step one-shot in the scalar case is of the form $\tau<\frac{\kappa(k,b)}{h^2m^2}$ where $\kappa(k,b)$ is defined as follows:
	\begin{enumerate}[label=(\roman*)]
		\item $\kappa(1,b)=\min\left\{\kappa_{11}(1,b),\kappa_2(1,b),\kappa_3(1,b)\right\}$, also note that
		$$\kappa_{11}(1,b)=1+b^2,\quad\kappa_{21}(1,b)=\frac{2b^2-2b-1+\sqrt{-4b+5}}{2},$$ $$\kappa_{22}(1,b)=\frac{2b^2+2b+1+\sqrt{8b^2+12b+5}}{2},\quad\kappa_3(1,b)=2(1-b)^2;$$
		\item for odd $k\ge 3$, 
		$$ \kappa(k,b)=\left\{
		\begin{array}{cl}
			\min\left\{\kappa_{11}(k,b),\kappa_2(k,b),\kappa_3(k,b)\right\},& -1<b<b_1(k)\vee b_2(k)<b<1,\\
			\min\left\{\kappa_{11}(k,b),\kappa_2(k,b)\right\},& b_1(k)\le b\le b_2(k)\wedge b\neq 0,\\
			1, & b=0
		\end{array}
		\right.$$
		where $-1<b_1(k)<0<b_2(k)<1$ are the two solutions of
		$$1-2kb^{k-1}+2kb^k-b^{2k}=0,\quad -1<b<1;$$
		\item for even $k\ge 2$,
		$$\kappa(k,b)=\left\{
		\begin{array}{cl}
			\min\left\{\kappa_{11}(k,b),\kappa_2(k,b),\kappa_3(k,b)\right\},& b_3(k)<b<1,\\
			\min\left\{\kappa_{11}(k,b),\kappa_2(k,b)\right\},& 0<b\le b_3(k),\\
			1, & b=0,\\
			\min\left\{\kappa_{12}(k,b),\kappa_2(k,b)\right\},& -1<b<0
		\end{array}
		\right.$$
		where $0<b_3(k)<1$ is the unique solution of
		$$1-2kb^{k-1}+2kb^k-b^{2k}=0,\quad -1<b<1.$$
	\end{enumerate}			
\end{proposition}
\begin{remark}
	In implementation, we rewrite $\kappa_{21}(k,b)$ as
	$$\cfrac{b(1-b)^2(b^k-1)}{k-(k+1)b+b^{k+1}}+\cfrac{2\cdot\left[-b^k+1+\frac{b(1-b)^2(1-b^k)y_k}{k-(k+1)b+b^{k+1}}\right]}{y_k+v_k+\sqrt{(-4s+5)v_k^2+y_k^2+2(-2s^2+2s+1)v_ky_k}}$$
	to avoid numerical errors. Also in this formula, we see that $\kappa_{21}(k,b)\overset{k\to\infty}{\longrightarrow}(1-b)^2$ (note that $y_k=\frac{1-kb^{k-1}+(k-1)b^k}{(1-b)^2}\overset{k\to\infty}{\longrightarrow}\frac{1}{(1-b)^2}$ and $v_k=t_k^2-y_k\overset{k\to\infty}{\longrightarrow}0$).
\end{remark}
For illustrations of the function $\kappa(k,b)$ for different $k$ see section~\ref{ssec:plots}. 

\subsection{Comparison of the bounds for the descent step}
\label{ssec:plots}

\begin{figure}
	\centering
	\begin{subfigure}{0.4\columnwidth}
		\includegraphics[width=\linewidth]{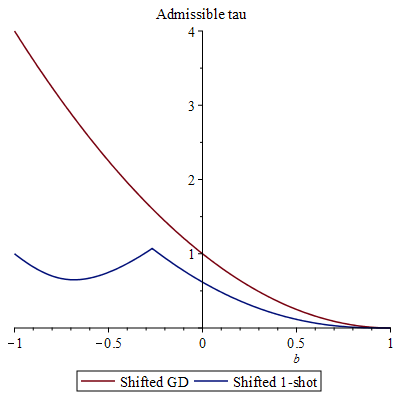}
		\caption{Shifted $1$-step one-shot}
	\end{subfigure} \hfill
	\begin{subfigure}{0.4\columnwidth}
		\includegraphics[width=\linewidth]{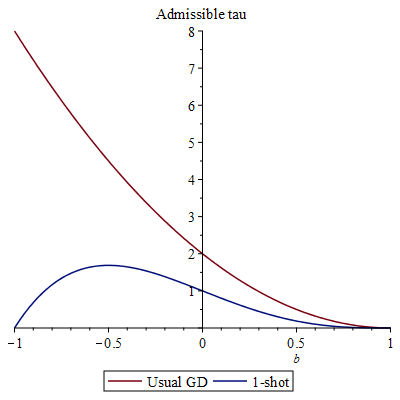}
		\caption{$1$-step one-shot}
	\end{subfigure}
	\begin{subfigure}{0.4\columnwidth}
		\includegraphics[width=\linewidth]{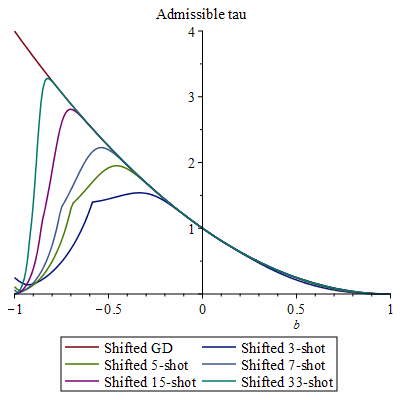}
		\caption{Shifted $k$-step one-shot, odd $k\ge 3$}
	\end{subfigure} \hfill
	\begin{subfigure}{0.4\columnwidth}
		\includegraphics[width=\linewidth]{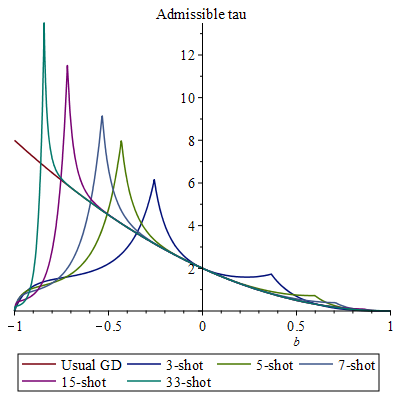}
		\caption{$k$-step one-shot, odd $k\ge 3$}
	\end{subfigure}
	\begin{subfigure}{0.4\columnwidth}
		\includegraphics[width=\linewidth]{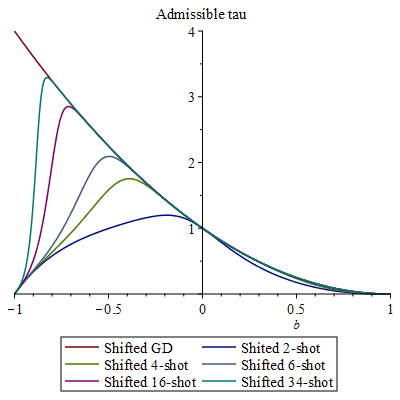}
		\caption{Shifted $k$-step one-shot, even $k\ge 2$}
	\end{subfigure} \hfill
	\begin{subfigure}{0.4\columnwidth}
		\includegraphics[width=\linewidth]{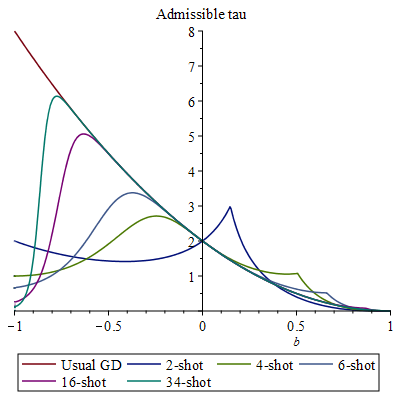}
		\caption{$k$-step one-shot, even $k\ge 2$}
	\end{subfigure}
	\caption{Admissible $\tau$ in the scalar case as a function of $b$.}
	\label{fig:plots1D}
\end{figure}
In summary, in the scalar case, the necessary and sufficient convergence conditions on the descent step $\tau > 0$ are: 
\[
\tau < \frac{2(1-b)^2}{h^2m^2}, \qquad  
\tau < \frac{(1-b)^2}{h^2m^2}, \qquad  
\tau < \frac{\eta(k,b)}{h^2m^2}, \qquad  
\tau < \frac{\kappa(k,b)}{h^2m^2}, 
\]
respectively for usual GD, shifted GD, $k$-step one-shot (with $\eta(k,b)$ given in Proposition~\ref{prop:tau-k-one-shot}), shifted $k$-step one-shot (with $\kappa(k,b)$ given in Proposition~\ref{prop:tau-shifted-k-one-shot}). 
By taking $m=h=1$, in Figure~\ref{fig:plots1D} we plot for different $k$ the functions: $b\mapsto 2(1-b)^2$ (usual GD), $b\mapsto(1-b)^2$ (shifted GD), $b\mapsto\eta(k,b)$ ($k$-step one-shot) and $b\mapsto\kappa(k,b)$ (shifted $k$-step one-shot).

From these plots we can draw two important conclusions. 
First, when $k$ increases the visualized curves for $k$-step one-shot and shifted $k$-step one-shot tend to the corresponding curves for usual and shifted gradient descent, as expected. 
Second, even in this scalar case, it appears difficult to establish a simplified expression for $\eta(k,b)$ in Proposition~\ref{prop:tau-k-one-shot} and $\kappa(k,b)$ in Proposition~\ref{prop:tau-shifted-k-one-shot} to find a practical upper bound for the descent step $\tau$.

\begin{remark}
	For $k\ge 2$, we observe that for some $b$ the admissible range of $\tau$ of $k$-step one-shot is larger than the one of usual GD, that is not intuitive. This is indeed verified numerically using FreeFEM: when $b=0.2$ and $\tau=2.08$, $2$-step one-shot converges while the usual GD does not.
\end{remark}

\section{A proof of Lemma \ref{marden-ord-3} based on Marden's works}
\label{app:marden}

\begin{definition}
	We say that a complex coefficient polynomial has property $\mathscr{P}$ if all its zeros lie (strictly) inside the unit circle $|z|=1$. 
\end{definition}
We recall some definitions from Marden's works \cite{marden66}.

\begin{definition}
	Let $P(z)=a_0+a_1z+...+a_nz^n$ where $a_k\in\R, k=0,...,n$ (we do not require $a_n\neq 0$ here). We define
	$$\tilde{P}(z) \coloneqq a_n+a_{n-1}z+...+a_0z^n$$
	and call it the reverse polynomial of $P$. One can also see that
	$\tilde{P}(z)=z^nP\left(1/z\right)$.
\end{definition}

\begin{definition}
	Let $P(z)=a_0+a_1z+...+a_nz^n$ where $a_k\in\R, k=0,...,n$. We define a polynomial sequence $\{P_k\}_{0\le k\le n}$ where
	$$P_k(z)=a_0^{(k)}+a_1^{(k)}z+...+a_{n-k}^{(k)}z^{n-k}$$
	as follows:
	\begin{itemize}
		\item $P_0=P$;
		\item $P_{k+1}=a_0^{(k)}P_k-a_{n-k}^{(k)}\tilde{P_k}$ for $0\le k\le n-1$.
	\end{itemize}
	Then we define
	$$m_k(P)=a_0^{(1)}a_0^{(2)} \cdots a_0^{(k)}, \quad 1\le k\le n.$$
\end{definition}
The coefficients of these polynomials can be gathered in the following table, that we call Marden's table: 
\begin{center}
	\begin{tabular}[htbp]{c|c|c|c|c|c|c}
		& 1 & $x$ & $x^2$ & $...$ & $x^{n-1}$ & $x^n$\\
		\hline
		$P_0$ & $a_0$ & $a_1$ & $a_2$ & $...$ & $a_{n-1}$ & $a_n$ \\
		$\tilde{P_0}$& $a_n$ & $a_{n-1}$ & $a_{n-2}$ & $...$ & $a_1$ & $a_0$ \\
		\hline
		$P_1$ & $a_0^{(1)}$ & $a_1^{(1)}$ & $a_2^{(1)}$ & $...$ & $a_{n-1}^{(1)}$ & \\
		$\tilde{P}_1$& $a_{n-1}^{(1)}$ & $a_{n-2}^{(1)}$ & $a_{n-3}^{(1)}$ & $...$ & $a_0^{(1)}$ & \\
		\hline
		$\vdots$ & $\vdots$ & $\vdots$ & $\vdots$ & $\vdots$ & $\vdots$ & $\vdots$ \\
		\hline
		$P_{n-1}$ & $a_0^{(n-1)}$ & $a_1^{(n-1)}$ &&&&\\ 
		$\tilde{P}_{n-1}$ & $a_1^{(n-1)}$ & $a_0^{(n-1)}$ &&&&\\ 
		\hline
		$P_n$ & $a_0^{(n)}$ &&&&&\\ 
	\end{tabular}
\end{center}

We have a nice and simple criterion mainly based on the works of Marden \cite{marden49, marden66} and Jury \cite{jury63, jury64}, known as Jury-Marden Criterion:

\begin{theorem}[Jury-Marden Criterion]\label{Marden criterion}
	The polynomial $P$ has property $\mathscr{P}$ if and only if 
	$$a_0^{(1)}<0;\quad a_0^{(k)}>0, \forall \, 2\le k\le n.$$
\end{theorem}
This necessary and sufficient condition is mentioned several times in the literature (see e.g. \cite[Theorem 3.10]{barnett83}), but it is not easy to find an explicit proof, so we provide a proof for the reader's convenience. Before proving this result, we apply Jury-Marden Criterion to a polynomial of degree $3$ and obtain precisely Lemma \ref{marden-ord-3}, that is the following proposition.
\begin{proposition}
	Let
	$P(z)=a_0+a_1z+a_2z^2+z^3, z\in\C$
	where $a_0,a_1,a_2\in\R$. Then $P$ has property $\mathscr{P}$ if and only if
	$$\begin{cases}
		(a_0-1)(a_0+1)<0\\
		(a_0^2-a_2a_0+a_1-1)(a_0^2+a_2a_0-a_1-1)>0\\
		(a_0+a_2-a_1-1)(a_0+a_2+a_1+1)<0.
	\end{cases}$$	
\end{proposition}

\begin{proof}
	By directly applying Jury-Marden Criterion to $P$, we obtain Marden's table as follows:
	\begin{center}
		\begin{tabular}[htbp]{c|c|c|c|c}
			& 1 & $x$ & $x^2$ & $x^3$\\
			\hline
			$P_0=P$ & $a_0$ & $a_1$ & $a_2$ & $1$ \\
			$\tilde{P}_0$& $1$ & $a_2$ & $a_1$ & $a_0$ \\
			\hline
			$P_1$ & $a_0^2-1$ & $a_1a_0-a_2$ & $a_2a_0-a_1$ & \\
			$\tilde{P}_1$ & $a_2a_0-a_1$ & $a_1a_0-a_2$ & $a_0^2-1$ & \\
			\hline
			$P_2$ & $(a_0^2-1)^2-(a_2a_0-a_1)^2$ & $(a_1a_0-a_2)(a_0^2-a_2a_0+a_1-1)$ & & \\
			$\tilde{P}_2$ & $(a_1a_0-a_2)(a_0^2-a_2a_0+a_1-1)$ & $(a_0^2-1)^2-(a_2a_0-a_1)^2$ &&\\
		\end{tabular}
	\end{center}
	and
	$$P_3(x)=\left[(a_0^2-1)^2-(a_2a_0-a_1)^2\right]^2-(a_1a_0-a_2)^2(a_0^2-a_2a_0+a_1-1)^2.\\
	$$
	Hence
	$$\begin{array}{lcl}
		a_0^{(1)}&=&a_0^2-1=(a_0-1)(a_0+1),\\
		a_0^{(2)}&=&(a_0^2-1)^2-(a_2a_0-a_1)^2=(a_0^2-a_2a_0+a_1-1)(a_0^2+a_2a_0-a_1-1),\\
		a_0^{(3)}&=&\left[(a_0^2-1)^2-(a_2a_0-a_1)^2\right]^2-(a_1a_0-a_2)^2(a_0^2-a_2a_0+a_1-1)^2\\
		&=& \left[(a_0^2+a_2a_0-a_1-1)^2-(a_1a_0-a_2)^2\right](a_0^2-a_2a_0+a_1-1)^2\\
		&=&[a_0^2+(a_2-a_1)a_0+a_2-a_1-1][a_0^2+(a_2+a_1)a_0-a_2-a_1-1]\\
		&&(a_0^2-1-a_2a_0+a_1)^2\\
		&=& (a_0+1)(a_0+a_2-a_1-1)(a_0-1)(a_0+a_2+a_1+1)(a_0^2-1-a_2a_0+a_1)^2.
	\end{array}$$
	Then the condition $a_0^{(1)}<0, a_0^{(2)}>0, a_0^{(3)}>0$, after being simplified, is equivalent to three inequalities of the statement. 
\end{proof}

Now, to prove Jury-Marden Criterion, we need the following two results. 

\begin{theorem}[Marden, \cite{marden66}, Theorem 42.1]\label{Marden theorem}
	Let $P$ be a real coefficient polynomial of $n$-th degree. If the sequence
	$$m_1(P),m_2(P),...,m_n(P)$$
	has exactly $p$ negative elements and $n-p$ positive elements (hence no null elements), then $P$ has $p$ complex roots (including multiplicities) inside the unit circle $|z|=1$, no roots on this circle and $n-p$ complex roots (including multiplicities) outside this circle.
\end{theorem}


\begin{lemma}[Schur, \cite{schur17}]\label{Schur2}
	Let $P(z)=a_0+a_1z+...+a_nz^n$ where $a_k\in\R,\forall 1\le k\le n$. 
	Assume that $|a_0|<|a_n|$. Then $\deg\tilde{P}_1=n-1$, and $P$ has property $\mathscr{P}$ if and only if $\tilde{P}_1$ has property $\mathscr{P}$.
\end{lemma}

\begin{proof}[Proof of Jury-Marden Criterion \ref{Marden criterion}]
	The sufficient condition for $P$ having property $\mathscr{P}$ is a direct consequence of Marden's Theorem \ref{Marden theorem}. It remains to prove the necessary one.
	
	For that, we will prove the following statement $M(n)$ by induction: ``For every real-coefficient polynomial $P$ of $n$-th degree having property $\mathscr{P}$, the sequence $a_0^{(1)},...,a_0^{(n)}$ obtained by Marden's algorithm must satisfy
	$$a_0^{(1)}<0,\quad  a_0^{(k)}>0, \forall \, 2\le k\le n."$$
	
	To check $M(1)$, let $P(z)=a_0+a_1z$ where $a_0,a_1\in\R, a_1\neq 0$. Then $P(z)=0\Leftrightarrow z=-a_0/a_1$ and $|-a_0/a_1|<1\Leftrightarrow |a_0|<|a_1|\Leftrightarrow a_0^{(1)}=a_0^2-a_1^2<0$.
	
	Now supposing that $M(n-1)$ is true for some $n\in\N, n\ge 2$, we show that $M(n)$ is true. Let $P(z)=a_0+a_1z+...+a_nz^n$ where $a_k\in\R, k=0,...,n$ and $a_n\neq 0$. Assume that $P$ has property $\mathscr{P}$. First, $a_0^{(1)}=a_0^2-a_n^2<0$. Indeed, let $z_1,z_2,...,z_n$ be the $n$ zeros including multiplicities of $P$, then by Vi\`{e}te's formulas 
	$z_1z_2 \cdots z_n=(-1)^n(a_0/a_n)$.
	Taking the module of both sides of this identity and noting that $P$ has property $\mathscr{P}$, we have $|a_0/a_n|<1$, thus  $a_0^{(1)}=a_0^2-a_n^2<0$. Next, by Lemma \ref{Schur2}, $\tilde{P}_1$ is of $(n-1)$-th degree and it also has property $\mathscr{P}$. Marden's table for $\tilde{P}_1$ can be easily found:
	\begin{center}
		\begin{tabular}[htbp]{c|c|c|c|c|c|c|c}
			& 1 & $x$ & $x^2$ & $...$ & $x^{n-3}$ & $x^{n-2}$ & $x^{n-1}$\\
			\hline
			$\tilde{P}_1$ & $a_{n-1}^{(1)}$ & $a_{n-2}^{(1)}$ & $a_{n-3}^{(1)}$ & $...$ & $a_{2}^{(1)}$ & $a_1^{(1)}$ & $a_0^{(1)}$\\
			$P_1$ & $a_0^{(1)}$ & $a_1^{(1)}$ & $a_2^{(1)}$ & $...$ & $a_{n-3}^{(1)}$ & $a_{n-2}^{(1)}$ & $a_{n-1}^{(1)}$ \\
			\hline
			$-P_2$ & $-a_0^{(2)}$ & $-a_1^{(2)}$ & $-a_2^{(2)}$ & $...$ & $-a_{n-3}^{(2)}$ & $-a_{n-2}^{(2)}$ & \\
			$-\tilde{P}_2$& $-a_{n-2}^{(2)}$ & $-a_{n-3}^{(2)}$ & $a_{n-4}^{(1)}$ & $...$  & $-a_1^{(2)}$ & $-a_0^{(2)}$ & \\
			\hline
			$P_3$ & $a_0^{(3)}$ & $a_1^{(3)}$ & $a_2^{(3)}$ & $...$ & $a_{n-3}^{(3)}$ && \\
			$\tilde{P}_3$ & $a_{n-3}^{(3)}$ & $a_{n-4}^{(3)}$ & $a_{n-5}^{(3)}$ & $...$ & $a_0^{(3)}$ &&  \\
			\hline
			$\vdots$ & $\vdots$ & $\vdots$ & $\vdots$ & $\vdots$ & $\vdots$ & $\vdots$ & $\vdots$  \\
			\hline
			$P_{n-1}$ & $a_0^{(n-1)}$ & $a_1^{(n-1)}$ &&&&&\\ 
			$\tilde{P}_{n-1}$ & $a_1^{(n-1)}$ & $a_0^{(n-1)}$ &&&&&\\ 
			\hline
			$P_n$ & $a_0^{(n)}$ &&&&&&\\ 
		\end{tabular}
	\end{center}
	By $M(n-1)$, we must then have
	$-a_0^{(2)}<0,\; a_0^{(k)}>0, \, \forall 3\le k\le n$.
\end{proof}

\end{document}